\documentclass[pdfmx, a4paper, 11pt]{article}

\title{Moduli space of rank three logarithmic connections on the projective line with three poles}
\author{Takafumi Matsumoto\footnote{Research Institute for Mathematical Sciences, Kyoto University, Kyoto 606-8512, Japan, Email address: tfmatumt@kurims.kyoto-u.ac.jp}}
\date{}

\usepackage[top=25truemm, bottom=30truemm, left=29truemm, right=29truemm]{geometry}
\usepackage{amsmath,amsthm, mathrsfs}
\usepackage{amssymb}
\usepackage{latexsym}
\usepackage{tikz}
\usepackage{tikz-cd}
\usepackage{here}
\usepackage{bm}
\usepackage[all]{xy}
\usepackage{enumerate}

\allowdisplaybreaks[1]

\theoremstyle{definition}
\newtheorem{definition}{Definition}[section]
\newtheorem{theorem}[definition]{Theorem}
\newtheorem{proposition}[definition]{Proposition}
\newtheorem{lemma}[definition]{Lemma}
\newtheorem{corollary}[definition]{Corollary}

\newtheorem{remark}[definition]{Remark}

\newcommand{\mfm}{\mathfrak{m}}

\newcommand{\mcB}{\mathcal{B}}
\newcommand{\mcC}{\mathcal{C}}
\newcommand{\mcD}{\mathcal{D}}
\newcommand{\mcE}{\mathcal{E}}
\newcommand{\mcF}{\mathcal{F}}
\newcommand{\mcG}{\mathcal{G}}
\newcommand{\mcH}{\mathcal{H}}

\newcommand{\mcL}{\mathcal{L}}
\newcommand{\mcM}{\mathcal{M}}
\newcommand{\mcN}{\mathcal{N}}
\newcommand{\mcO}{\mathcal{O}}

\newcommand{\vb}{\mathrel{} \middle | \mathrel{}}

\newcommand{\App}{\textrm{App}}
\newcommand{\Bun}{\textrm{Bun}}

\newcommand{\elm}{\textrm{elm}}

\newcommand{\Ext}{\textrm{Ext}}
\newcommand{\Grass}{\textrm{Grass}}

\newcommand{\Hom}{\textrm{Hom}}
\newcommand{\id}{\textrm{id}}
\newcommand{\Image}{\textrm{Im}\,}
\newcommand{\Ker}{\textrm{Ker}\,}
\newcommand{\len}{\textrm{length}\,}

\newcommand{\Ompl}{\Omega^1_{\pl}}
\newcommand{\Ompld}{\Omega^1_{\pl}(D(\bm{t}))}

\newcommand{\Opl}{\mathcal{O}_{\pl}}

\newcommand{\PC}{\textrm{PC}}

\newcommand{\pl}{\mathbb{P}^1}
\newcommand{\Proj}{\textrm{Proj}\,}
\newcommand{\Quot}{\textrm{Quot}}
\newcommand{\rank}{\textrm{rank}\,}

\newcommand{\res}{\textrm{res}}

\newcommand{\Spec}{\textrm{Spec}\,}

\newcommand{\Sym}{\textrm{Sym}\,} 
 
\newcommand{\tr}{\textrm{tr}}

\begin{document}
	\maketitle
	\begin{abstract}
		In this paper, we describe the moduli space of rank three parabolic logarithmic connections on the projective line with three poles for any local exponents. In particular, we show that the family of moduli spaces of rank three parabolic $\phi$-connections on the projective line with three poles is isomorphic to the family of  $A^{(1)*}_2$-surfaces in Sakai's classification of Painlev\'e equations. Through this description, we investigate the relation between the apparent singularities and underlying parabolic bundles.
	\end{abstract}
	\section{Introduction}
	Our aim is two-fold. First, we derive the whole of the $A^{(1)*}_2$-surface in Sakai's classification of Painlev\'e equations from moduli theory. Second, we give an example of the moduli space of parabolic connections with rank $\geq 3$.
	\subsection{The moduli space of meromorphic connections and the Painlev\'e equations}
	
	H. Sakai \cite{Sa1} provided a geometric approach to the Painlev\'e equations and the discrete Painlev\'e equations. 
	He characterized the good compactification of spaces of initial conditions for the Painlev\'e equations as a certain rational projective surface and classified them according to affine root systems. We call a surface corresponding to an affine root system $R$ a $R$-surface and denote it by $S(R)$. In his framework, the discrete Painlev\'e equations are the dynamical systems generated by the action of the translation part of the corresponding affine Weyl group on the family of rational surfaces, and the Painlev\'e equations appear as a limit of the translation part. Each classified surface $S(R)$ is obtained by blowing up the projective plane $\mathbb{P}^2$ at 9 points, 
	including infinitely near ones, and has a unique effective anti-canonical divisor $Y_{S(R)}$. The following is the list of the types of surfaces and the Painlev\'e equations:
	\[
	\begin{minipage}{16cm}
		\centering
		\begin{tabular}{|c|c|c|c|c|c|c|c|c|}\hline
			surface type& $D^{(1)}_4$&$D^{(1)}_5$&$D^{(1)}_6$&$D^{(1)}_7$&$D^{(1)}_8$&$E^{(1)}_6$&$E^{(1)}_7$&$E^{(1)}_8$ \\ \hline
			Painlev\'e equation& $P_{VI}$&$P_{V}$&$P_{III}^{D^{(1)}_6}$&$P_{III}^{D^{(1)}_7}$&$P_{III}^{D^{(1)}_8}$&$P_{IV}$&$P_{II}$&$P_{I}$ \\ \hline
		\end{tabular}
	\end{minipage}
	\]
	Then the space of initial conditions for the Painlev\'e equation coincides with the surface $S(R)\setminus Y_{S(R)}$, where $R$ is the corresponding affine root system. 
	
	One of the important characteristics of the Painlev\'e equations is that they can be derived from the isomonodromic deformations of systems of linear differential equations.
	For example, the Painlev\'e VI equation is the isomonodromic deformation equation of a rank two linear system with four regular singularities. 
	Moduli spaces of meromorphic connections connect the isomonodromic deformation and the space of initial conditions. 
	The moduli spaces are Poisson, and become holomorphic symplectic varieties after fixing the residue data at each pole.
	The equations of the isomonodromic deformations can be geometrically understood as a Hamiltonian vector field on the moduli space of meromorphic connections through the Riemann-Hilbert correspondence. Thus we can regard the moduli space of meromorphic connections as a space of initial conditions of the equation determined by the isomonodromic deformation. 
	
	In Sakai's theory, the (additive) difference Painlev\'e equations are classified into the following eleven surface types:
	\[
	\begin{array}{ccccccccccc}
		A^{(1)**}_0, &A^{(1)*}_1, &A^{(1)*}_2, &D^{(1)}_4, &D^{(1)}_5, &D^{(1)}_6, &D^{(1)}_7,  &D^{(1)}_8, &E^{(1)}_6, &E^{(1)}_7, &E^{(1)}_8
	\end{array}
	\]
	The surfaces of $D^{(1)}_l$ and $E^{(1)}_l$ types are a compactification of the space of initial conditions for the Painlev\'e equations. In particular, the surface $S(R)\setminus Y_{S(R)}$ for $R=D^{(1)}_l, E^{(1)}_l$ is realized as the moduli space of meromorphic connections. This implies that the difference Painlev\'e equations of $D^{(1)}_l$ and $E^{(1)}_l$ types arise from the discrete deformation of rational systems of linear differential equations. In fact, the difference Painlev\'e equations of $D^{(1)}_l$ and $E^{(1)}_l$ types are obtained by Schlesinger transformations of rational systems of linear differential equations, which are rational gauge transformations shifting the exponents at the poles by integers. This naturally leads to the question: can the difference Painlev\'e equations of $A^{(1)}_l$ types be written in the form of the Schlesinger transformations? This problem is posed by Sakai in \cite{Sa2}. P. Boalch \cite{Bo1} found Fuchsian systems, i.e. logarithmic connections on the trivial bundle over $\pl$, corresponding to the type $A^{(1)}_l$ from the perspective of quiver variety and symmetry:
	\[
	\begin{minipage}{16cm}
		\centering
		\begin{tabular}{|c|c|c|c|c|}\hline
			surface type& $A^{(1)**}_0$&$A^{(1)*}_1$&$A^{(1)*}_2$ \\ \hline
			symmetry type& $E^{(1)}_8$&$E^{(1)}_7$&$E^{(1)}_6$ \\ \hline
			spectral type& $33, 222, 111111$&$22, 1111, 1111$&$111, 111, 111$ \\ \hline
		\end{tabular}
	\end{minipage}
	\]
	The moduli spaces of Fuchsian systems corresponding to $A^{(1)**}_0, A^{(1)*}_1$ and $A^{(1)*}_2$ types are identified with the Kronheimer's ALE spaces of $E_8, E_7$ and $E_6$ types, respectively. The $E_r$-type ALE space is obtained by blowing up $\mathbb{P}^2$ at $r$ points on the smooth locus of a cuspidal cubic and removing the strict transform of the cubic. In \cite{Bo1} he also explained how to obtain the surfaces of $A^{(1)**}_0, A^{(1)*}_1$ and $A^{(1)*}_2$ types from the corresponding ALE spaces, that is, how to partially compactify the moduli space of logarithmic connections on the trivial bundle to get the full moduli space of logarithmic connections of degree zero. Hence the surface $S(R)\setminus Y_{S(R)}$ for $R=A^{(1)**}_0, A^{(1)*}_1, A^{(1)*}_2$ is also realized as the moduli space of meromorphic connections. By the way, D. Arinkin and A. Borodin \cite{AB} also pointed out that rank three logarithmic connections over $\pl$ with three poles correspond to $A^{(1)*}_2$-surfaces from the perspective of difference equations and the Mellin transform. For a quiver-theoretic realization of a Zariski open subset of the moduli space of irregular connections, see \cite{Bo2, Do, HY}. 
	
	\subsection{Realization of  $A^{(1)*}_2$-surfaces as the moduli spaces}
	A natural question is whether the effective anti-canonical divisor $Y_{S(R)}$ is also obtained from the moduli theory. M. Inaba, K. Iwasaki and M.-H. Saito \cite{IIS1} introduced the notion of rank two parabolic logarithmic $\phi$-connections and proved that the moduli space of stable rank two parabolic logarithmic $\phi$-connections on $\pl$ with four poles is isomorphic to a $D^{(1)}_4$-surface which is the good compactification of the space of initial conditions for the Painlev\'e VI equations. On the other hand, among the surfaces in Sakai's classification, there are no cases where the whole of the surface has been derived as a moduli space, except for the Painlev\'e VI case. The first purpose of this paper is to derive the whole of the $A^{(1)*}_2$-surface, whose corresponding Fuchsian systems have the lowest rank among the types $A^{(1)}_l$, as the moduli space. To compactify the moduli space, we introduce the notion of parabolic logarithmic $\phi$-connections for arbitrary rank. This is a modification of rank two parabolic $\phi$-connections in \cite{IIS1} (see Remark \ref{phiconndif}). In Section 2 we construct the moduli space of parabolic $\phi$-connections, that is, we prove the following:
	
	\begin{theorem}\label{MT}
		Let $\tilde{M}_{g,n}$ be a smooth algebraic scheme which is a smooth covering of the coarse moduli space of $n$ pointed irreducible smooth projective curves of genus $g$ over $\mathbb{C}$ and take a universal family $(\mcC, \tilde{\bm{t}})=(\mcC, \tilde{t}_1,\ldots,\tilde{t}_n)$ over $\tilde{M}_{g,n}$. Let $\boldsymbol{\alpha}=\{\alpha^{(k)}_{i,j}\}^{k=1,2}_{1\leq i\leq n, 1\leq j \leq r}$ be a parabolic weight.
		\begin{itemize}
			\setlength{\itemsep}{0cm}
			\item[(1)] There exists a relative fine moduli scheme
			\[
			\overline{M^{\boldsymbol{\alpha}}_{\mcC/\tilde{M}_{g,n}}}(\tilde{\bm{t}},r, d)\longrightarrow
			\tilde{M}_{g,n}\times \mcN_{n,r}
			\] 
			of $\boldsymbol{\alpha}$-stable parabolic logarithmic $\phi$-connections of rank $r$ and degree $d$.
			If $\boldsymbol{\alpha}$ is generic, then $\overline{M^{\boldsymbol{\alpha}}_{\mcC/\tilde{M}_{g,n}}}(\tilde{\bm{t}},r,d)$ is projective over $\tilde{M}_{g,n}\times \mcN$. 
			\item[(2)] Assume that $\alpha^{(1)}_{i,j}=\alpha^{(2)}_{i,j}=:\alpha'_{i,j}$ for any $1\leq i\leq n$ and $1\leq j\leq r$. Then the set
			\[
			U_{\text{isom}}:=
			\left\{(E_1,E_2,\phi,\nabla,l^{(1)}_*,l^{(2)}_*) \in \overline{M^{\boldsymbol{\alpha}}_{\mcC/\tilde{M}_{g,n}}}(\tilde{\bm{t}},r,d) \vb \text{$\phi$ is an isomorphism} \right\}
			\]
			is a Zariski open subset of $\overline{M^{\boldsymbol{\alpha}}_{\mcC/\tilde{M}_{g,n}}}(\tilde{\bm{t}},r,d)$ and
			the natural morphism
			\[
			M^{\boldsymbol{\alpha}'}_{\mcC/\tilde{M}_{g,n}}(\tilde{\bm{t}},r,d)\longrightarrow U_{\text{isom}}, \quad (E,\nabla, l_*) \longmapsto (E, E, \id, \nabla, l_*, l_*) 
			\] 
			is an isomorphism, where $\boldsymbol{\alpha}'=\{\alpha'_{i,j}\}^{1\leq i\leq n}_{1\leq j \leq r}$ and $M^{\boldsymbol{\alpha}'}_{\mcC/\tilde{M}_{g,n}}(\tilde{\bm{t}},r,d)$ is a relative moduli space of $\boldsymbol{\alpha}'$-stable parabolic logarithmic connections connections constructed by M. Inaba \cite{In}.
		\end{itemize}
	\end{theorem}
	
	The construction is based on the method of \cite{IIS1} and \cite{In}. We don't know whether the moduli space of parabolic $\phi$-connections is irreducible or not in general. 
	
	In Sakai's theory, $A^{(1)*}_2$-surfaces are approximately parameterized by a six-dimensional affine space $\mathbb{A}^6$ over $\mathbb{C}$ and a natural action of $W(E^{(1)}_6)$ on $\mathbb{A}^6$ lifts to a regular isomorphism between $A^{(1)*}_2$-surfaces. When a point of $\mathbb{A}^6$ does not lie on reflection hyperplanes of reflections in $W(E^{(1)}_6)$, the corresponding $A^{(1)*}_2$-surface is obtained by blowing up $\mathbb{P}^2$ at 9 distinct points. On the other hand, when a point lies on a reflection hyperplane, we have to blow up $\mathbb{P}^2$ at 9 points, including infinitely near ones. Our goal is to derive the family of $A^{(1)*}_2$-surfaces as the family of moduli spaces of parabolic logarithmic $\phi$-connections. Parabolic structures of logarithmic connections play a role in realizing exceptional curves on $A^{(1)*}_2$-surfaces over reflection hyperplanes.

	We state the main theorem. Put
	\[
	T_3:=\left\{(t_1,t_2,t_3)\in (\pl)^3 \vb  \text{$t_i\neq t_j$ for $i\neq j$} \right\}, 
	\]
	\[
	\mcN(\nu_1,\nu_2,\nu_3):=\{(\nu_{i,j})\in \mathbb{C}^9\mid \nu_{i,0}+\nu_{i,1}+\nu_{i,2}=\nu_{i}, 1\leq i\leq 3\}, 
	\]
	where $\nu_1, \nu_2, \nu_3 \in \mathbb{C}$ and $\nu_1+\nu_2+\nu_3\in \mathbb{Z}$. 
	Take $\bm{t}\in T_3$ and $\boldsymbol{\nu}\in \mcN(\nu_1,\nu_2,\nu_3)$. Let $M^{\boldsymbol{\alpha}}_3(\nu_1,\nu_2,\nu_3)\rightarrow T_3\times \mcN(\nu_1,\nu_2,\nu_3)$ (resp. $\overline{M^{\boldsymbol{\alpha}}_3}(\nu_1,\nu_2,\nu_3)\rightarrow T_3\times \mcN(\nu_1,\nu_2,\nu_3)$) be the family of moduli spaces of $\boldsymbol{\alpha}$-stable parabolic connections (resp. $\phi$-connections), whose fiber $M^{\boldsymbol{\alpha}}_3(\bm{t},\boldsymbol{\nu})$ (resp. $\overline{M^{\boldsymbol{\alpha}}_3}(\bm{t},\boldsymbol{\nu})$) at $(\bm{t}, \boldsymbol{\nu})\in T_3\times \mcN(\nu_1,\nu_2,\nu_3)$ is the moduli space of $\boldsymbol{\alpha}$-stable $\boldsymbol{\nu}$-parabolic connections (resp. $\phi$-connections) over $(\pl, \bm{t})$. Let $S$ be the family of $A^{(1)*}_2$-surfaces parametrized by $T_3 \times\mcN(0,0,2)$ defined in section \ref{familysursec}. 
	\begin{theorem}\label{intromt1} (Theorem \ref{maintheorem})
		Take $\boldsymbol{\alpha}=(\alpha_{i,j})_{1\leq i,j\leq 3}$ such that $0<\alpha_{i,j}\ll 1$ for any $1\leq i, j\leq 3$.
		
		\begin{itemize}
			\item[(1)] There exists an isomorphism  $\overline{M^{\boldsymbol{\alpha}}_3}(0,0,2)\longrightarrow S$ over $T_3 \times\mcN(0,0,2)$. 
			In particular, for each $(\bm{t},\boldsymbol{\nu})\in T_3 \times \mcN(0,0,2)$, the fiber $\overline{M^{\boldsymbol{\alpha}}_3}(\bm{t},\boldsymbol{\nu})$ is isomorphic to an $A^{(1)*}_2$-surface. 
			\item[(2)] Let $Y$ be the closed subscheme of $\overline{M^{\boldsymbol{\alpha}}_3}(0,0,2)$ defined by the conditions $\wedge^3\phi=0$. Then $Y$ is reduced, and for each $(\bm{t},\boldsymbol{\nu})\in T_3 \times \mcN(0,0,2)$ the fiber $Y_{(\bm{t},\boldsymbol{\nu})}$ is the anti-canonical divisor of $\overline{M^{\boldsymbol{\alpha}}_3}(\bm{t},\boldsymbol{\nu})$.
		\end{itemize}
	\end{theorem}
	
	Finding a good coordinate on the space of initial conditions is important to describe the difference Painelev\'e equations explicitly. A. Dzhamay, H. Sakai and T. Takenawa \cite{DST} introduced rational parameters of Fuchsian systems corresponding to the type $A^{(1)*}_2$, which provide a good coordinate on a Zariski open subset of a $A^{(1)*}_2$-surface. They regard an $A^{(1)*}_2$-surface as the surface obtained by blowing up $\pl \times \pl$ at 8 points and gave an explicit correspondence between Fuchsian systems and points on a Zariski open subset of $\pl \times \pl$. In \cite{DT} Dzhamay and Takenawa gave a more detailed exposition of the $A^{(1)*}_2$ case.
	
	To show Theorem \ref{intromt1}, we provide normal forms of $\boldsymbol{\alpha}$-stable rank three parabolic $\phi$-connections over $\pl$ with three poles by using the apparent singularity and its dual parameter (see subsection \ref{nfsec}). The normal forms give us the explicit correspondence between stable parabolic logarithmic $\phi$-connections and points on the whole of the $A^{(1)*}_2$-surface, which provide a coordinate on an $A^{(1)*}_2$-surface (see subsection \ref{expcorr}). Unlike Dzhamay-Sakai-Takenawa, we regard it as the surface obtained by blowing up $\mathbb{P}^2$ at 9 points. The relation between our coordinate and their coordinate is not made. 
	\begin{figure}
		\centering
		\includegraphics[width=9cm]{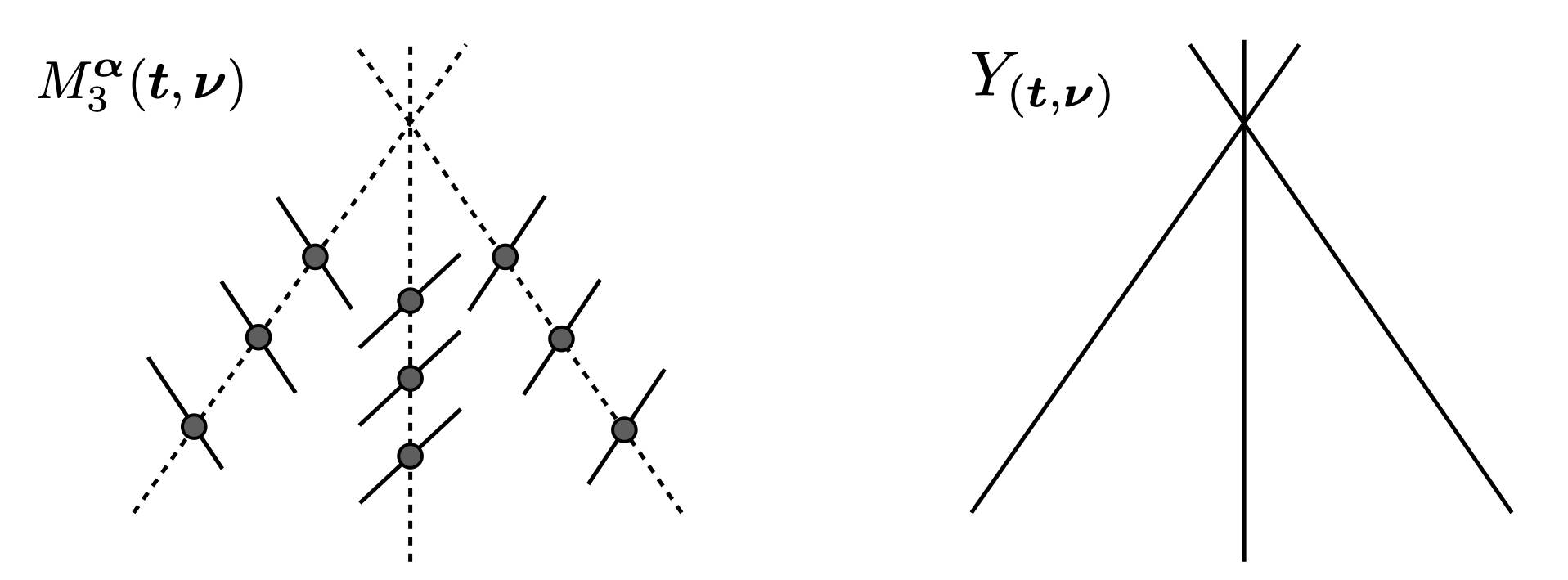}
	\end{figure}
	
	\subsection{Moduli space of parabolic bundles and parabolic connections}
	The moduli space of meromorphic connections has the natural symplectic structure. Giving a Darboux coordinate on the moduli space is important for studying the isomonodromic deformations. One of the methods of introducing a Darboux coordinate on the moduli space of logarithmic connections is by using apparent singularities and dual parameters (for example \cite{AL, DM, IIS2, KS, Ob}). This method is extended to the case of rank two irregular connections over $\pl$ by Diarra-Loray \cite{DL}. 
	
	In the case of rank two parabolic logarithmic connections, it is known that the apparent singularities and underlying parabolic bundles provide a Darboux coordinate on the moduli space of parabolic connections. Let $L$ be a line bundle of degree $r(g-1)+1$ over an irreducible smooth curve $C$ of genus $g$ and $\nabla_L$ be a logarithmic connection on $L$ with poles at $t_1, \ldots, t_n$. We wright by $M^{\boldsymbol{\alpha}}(L,\nabla_L)$ the moduli space of $\boldsymbol{\alpha}$-stable $\boldsymbol{\nu}$-parabolic connections of rank $r$ over $(C, t_1, \ldots, t_n)$ with the trace connection $(L,\nabla_L)$. The moduli space $M^{\boldsymbol{\alpha}}(L,\nabla_L)$ has two rational maps, the apparent map and the bundle map. The apparent map $\App \colon M^{\boldsymbol{\alpha}}(L,\nabla_L)\dashrightarrow \mathbb{P}^N$ was defined in \cite{SS}, where $N$ is the half of the dimension of $M^{\boldsymbol{\alpha}}(L,\nabla_L)$. It is a geometric interpretation of the apparent singularities of systems of linear differential equations. Let $P^{\boldsymbol{\alpha}}(L)$ be the moduli space of $\boldsymbol{\alpha}$-stable parabolic bundles with determinant $L$. The bundle map $\Bun \colon M^{\boldsymbol{\alpha}}(L,\nabla_L)\dashrightarrow P^{\boldsymbol{\alpha}}(L)$ is the map forgetting parabolic structures. We consider the rational map
	\[
	\App \times \Bun \colon M^{\boldsymbol{\alpha}}(L,\nabla_L)\dashrightarrow \mathbb{P}^N\times P^{\boldsymbol{\alpha}}(L).
	\]
	When $r=2$, $\App \times \Bun$ is birational, and both $\App$ and $\Bun$ are Lagrangian fibrations on a Zariski open subset. Thus $\App \times \Bun$ provides a Darboux coordinate on $M^{\boldsymbol{\alpha}}(L,\nabla_L)$. These results were proved by Loray-Saito \cite{LS} when $g=0$, by Fassarella-Loray \cite{FL} and Fassarella-Loray-Muniz \cite{FLM} when $g=1$, and by the author \cite{Ma} when $g\geq 2$. $P^{\boldsymbol{\alpha}}(L)$ is birational to $\mathbb{P}^N$. It follows, for example, from the fact that a generic $\boldsymbol{\alpha}$-parabolic bundle is obtained from an extension of $L$ by $\mcO_C$, which is a key point of the proof of the above results. In particular, we can give a Darboux coordinate on $M^{\boldsymbol{\alpha}}(L,\nabla_L)$ by using a coordinate on $\mathbb{P}^N\times\mathbb{P}^N$. These results are extended to the case of rank two irregular connections by Komyo-Loray-Saito-Szab\'{o} \cite{KLSS}. 
	
	For the case $r\geq 3$, it is not known whether or not $\App \times \Bun$ gives a Darbouex coordinate on a Zariski open subset of the moduli space. The second purpose of this paper is to give an example in which $\App \times \Bun$ is not birational. This implies that $\App \times \Bun$ does not provide a Darbouex coordinate on a Zariski open subset of the moduli space in general. The apparent map $\App$ is constructed by using the filtration of the underlying bundle by subbundles. When $r=2$, the construction is simple, and the relation between the apparent singularities and parabolic bundles can be relatively easily calculated by using the \v{C}ech cohomology.
	On the other hand, when $r\geq 3$, the construction is complicated, and the computation by using the \v{C}ech cohomology is hard. Hence the relation is unclear. In Section 4, we study the simplest case among higher rank cases in another way, that is, we investigate the moduli space of rank three parabolic logarithmic connections over $\pl$ with three poles by determining stable parabolic bundles and writing down a parabolic logarithmic connection and a parabolic Higgs field on any stable parabolic bundle. 
	
	The notion of $\lambda$-connections is the interpolation of Higgs bundles and connections. The moduli space of $\lambda$-connections has a fibration over $\mathbb{C}$ such that the fibers over 0 and 1 are the moduli spaces of Higgs bundles and connections, respectively. The moduli spaces of Higgs bundles and connections are diffeomorphic to each other, but their complex structures are not equivalent. The moduli spaces of $\lambda$-connections can be seen as twistor spaces of suitable hyperk\"ahler manifolds.
	
	Let $(E,l_*)$ be a parabolic bundle and $\nabla$ be a $\boldsymbol{\nu}$-logarithmic connection over $(E,l_*)$. All $\lambda\boldsymbol{\nu}$-logarithmic $\lambda$-connections over $(E,l_*)$ are of the form $\lambda\nabla+\Phi$, where $\Phi$ is a parabolic Higgs field over $(E,l_*)$. The space of all isomorphism classes of $\lambda\boldsymbol{\nu}$-logarithmic $\lambda$-connections over $(E,l_*)$ is $\mathbb{P}(\mathbb{C}\nabla\oplus H)$
	and it can be regarded as a compactification of the space of all $\boldsymbol{\nu}$-logarithmic connections over $(E,l_*)$. Here $H$ is the space of all parabolic Higgs fields over $(E,l_*)$. 
	
	Let $P^{\boldsymbol{\alpha}}(-2)$ be the moduli space of rank three $\boldsymbol{\alpha}$-stable parabolic bundles with degree $-2$ over $(\pl,\bm{t})$ and put
	\[
	M^{\boldsymbol{\alpha}}_3(\bm{t},\boldsymbol{\nu})^0:=\{(E, \nabla, l_*)\in M^{\boldsymbol{\alpha}}_3(\bm{t},\boldsymbol{\nu}) \mid (E, l_*) \in P^{\boldsymbol{\alpha}}(-2)\}. 
	\]
	When $P^{\boldsymbol{\alpha}}(-2)$ is nonempty, $M^{\boldsymbol{\alpha}}_3(\bm{t},\boldsymbol{\nu})^0$ is a Zariski open subset of $M^{\boldsymbol{\alpha}}_3(\bm{t},\boldsymbol{\nu})$. There is a natural $\mathbb{C}^*$ action on the moduli space of $\lambda$-connections over $\boldsymbol{\alpha}$-stable parabolic bundles, and the quotient $\overline{M^w_3(\bm{t},\boldsymbol{\nu})^0}$ is a compactification of $M^{\boldsymbol{\alpha}}_3(\bm{t},\boldsymbol{\nu})$. 
	
	\begin{theorem}\label{introthmtwo}(Theorem \ref{pbthm})
		For a special weight $\boldsymbol{\alpha}$, 
		we have 
		\[
		\overline{M^{\boldsymbol{\alpha}}_3(\bm{t},\boldsymbol{\nu})^0}\cong 
		\left\{
		\begin{array}{lll}
			\pl\times \pl &\nu_{1,0}+\nu_{2,0}+\nu_{3,0}\neq 0\\
			\mathbb{P}(\Opl\oplus \Opl(-2)) &\nu_{1,0}+\nu_{2,0}+\nu_{3,0}= 0.
		\end{array}
		\right.
		\]
	\end{theorem}
	Let $V_0$ be a Zariski open subset of $P^{\boldsymbol{\alpha}}(-2)$ defined in the subsection \ref{pbapp}. The following shows that $\App\times \Bun$ is not birational in general.
	\begin{corollary} (Proposition \ref{appbun})
		Assume that $\nu_{1,0}+\nu_{2,0}+\nu_{3,0}\neq 0$. 
		Then for a special weight $\boldsymbol{\alpha}$, the morphism
		\[
		\App \times \Bun \colon \Bun^{-1}(V_0)\longrightarrow \pl\times V_0
		\]
		is finite and its generic fiber consists of three points. 
	\end{corollary}
	\subsection{Outline of this paper}
	Section 2 is devoted to the construction of the moduli space of parabolic $\phi$-connections. 
	First, we recall the basic definitions and facts of parabolic connections. Second, we introduce parabolic $\phi$-connections and define the moduli functor of parabolic $\phi$-connections. Third, we defined the elementary transformations of parabolic $\phi$-connections. Fourth, we introduce parabolic $\Lambda^1_D$-triples. Finally, we construct the moduli space of parabolic $\Lambda^1_D$-triples and construct the moduli space of parabolic $\phi$-connections as a closed subscheme of the moduli space of parabolic $\Lambda^1_D$-triples.
	
	In section 3, we will prove Theorem \ref{intromt1}. First, we define the apparent singularity of parabolic $\phi$-connections by using a filtration by subbundles. We can see that the apparent singularity of parabolic $\phi$-connections with $\rank \phi =1$ is not uniquely determined. So we consider pairs of a parabolic $\phi$-connection and a subbundle. Then the apparent map is defined on the moduli space $\widehat{M^{\boldsymbol{\alpha}}_3}(\bm{t},\boldsymbol{\nu})$ of such pairs. Second, we define a morphism $\varphi \colon \widehat{M^{\boldsymbol{\alpha}}_3}(\bm{t},\boldsymbol{\nu}) \rightarrow \mathbb{P}(\Ompld\oplus\Opl)$. Third, we provide a normal form of parabolic $\phi$-connections. By using this form we prove the smoothness of $\overline{M^{\boldsymbol{\alpha}}_3}(\bm{t},\boldsymbol{\nu})$. Fourth, we prove Theorem \ref{intromt1}. We prove that the forgetful map $\widehat{M^{\boldsymbol{\alpha}}_3}(\bm{t},\boldsymbol{\nu}) \rightarrow \overline{M^{\boldsymbol{\alpha}}_3}(\bm{t},\boldsymbol{\nu})$ is a blow-up at a point and $\varphi$ is a blow-up at 9 points.
	
	Section 4 is devoted to studying the geometry of the moduli space of parabolic bundles and parabolic connections. First, we consider the moduli space of $w$-stable parabolic bundles. We determine the type of $w$-stable parabolic bundles and investigate a wall-crossing phenomenon. Second, we show Theorem \ref{introthmtwo} by writing down a $\boldsymbol{\nu}$-parabolic connection and a parabolic Higgs field. Moreover, we investigate the relation between two moduli spaces $\overline{M^{\boldsymbol{\alpha}}_3}(\bm{t},\boldsymbol{\nu})$ and $\overline{M^w_3(\bm{t},\boldsymbol{\nu})^0}$. Finally, we study the morphism $\App\times \Bun$.
	
	In appendices, we provide proofs of some propositions. These proofs require complicated computations.
	
	\section{Construction of moduli space of parabolic $\phi$-connections}
	In this section we construct the moduli space of parabolic $\phi$-connections. The construction is based on \cite{IIS1} and \cite{In}. 
	\subsection{Parabolic connections}
	Let $C$ be an irreducible smooth projective curve over $\mathbb{C}$ and $\bm{t}=(t_i)_{1\leq i\leq n}$ be a set of $n$ distinct points of $C$.
	Put $D(\bm{t})=t_1+\cdots +t_n$ and take $\boldsymbol{\nu}=(\nu_{i,j})^{1\leq i\leq n}_{0\leq j \leq r-1} \in \mathbb{C}^{rn}$.
	\begin{definition}
		A $\boldsymbol{\nu}$-parabolic connection of rank $r$ and degree $d$ is a collection $(E,\nabla,l_*=\{l_{i,*}\}_{1 \leq i\leq n})$ consisting of the following data:
		\begin{itemize}
			\setlength{\itemsep}{0cm}
			\item[(1)] $E$ is a vector bundle on $C$ of rank $r$ and degree $d$,
			\item[(2)] $\nabla \colon E \rightarrow E \otimes \Omega_C^1(D(\bm{t}))$ is a logarithmic connection, i.e. $\nabla(fs)=s\otimes df + f \nabla(s)$ for any $f \in \mcO_C, s \in E$, and 
			\item[(3)] $l_{i,*}$ is a filtration $E|_{t_i}=l_{i,0} \supsetneq \cdots \supsetneq  l_{i,r-1} \supsetneq l_{i,r}=\{0\}$ satisfying $(\res_{t_i}(\nabla)-\nu_{i,j}\id )({l_{i,j}})\subset l_{i,j+1}$ for $1\leq i\leq n$ and $0\leq j\leq r-1$.
		\end{itemize}
	\end{definition}
	\begin{proposition} \label{fuchs}(Fuchs relation)
		Let $(E,\nabla,l_*)$ be a $\boldsymbol{\nu}$-parabolic connection of rank $r$ and degree $d$. Then we have 
		\[
		\sum_{i=1}^n\sum_{j=0}^{r-1}\nu_{i,j} +d=0.
		\]
	\end{proposition}
	We put 
	\[
	\mcN_{n,r}(d):=\left\{(\nu_{i,j})^{1\leq i\leq n}_{0\leq j \leq r-1}\in \mathbb{C}^{rn}\vb  \sum_{i=1}^n\sum_{j=0}^{r-1}\nu_{i,j} +d=0\right\}.
	\]
	Let us fix $\boldsymbol{\nu}=(\nu_{i,j})^{1\leq i\leq n}_{0\leq j \leq r-1} \in \mcN_{n,r}(d)$.
	\begin{definition}
		We say that two $\boldsymbol{\nu}$-parabolic connections $(E,\nabla,l_*), (E,\nabla',l'_*)$
		are isomorphic to each other if there is an isomorphisms $\sigma \colon E \overset{\sim}{\longrightarrow} E'$ such that the diagram
		\[
		\begin{tikzcd}
			E \arrow[r,"\nabla"] \arrow[d, "\sigma"']&E \otimes\Omega_C^1(D(\bm{t}))\arrow[d, "\sigma \otimes \id "] \\
			E' \arrow[r,"\nabla'"]&E' \otimes \Omega_C^1(D(\bm{t}))
		\end{tikzcd}
		\]
		is commutative and $\sigma_{t_i}(l_{i,j})=l'_{i,j}$ for $1\leq i\leq n$ and $1\leq j\leq r-1$.
	\end{definition}
	Let $\boldsymbol{\alpha}=\{\alpha_{i,j}\}^{1 \leq i \leq n}_{1 \leq j \leq r}$ be a set of rational numbers satisfying $0 < \alpha_{i,1} <\cdots <\alpha_{i,r} < 1$ for each $i=1,\ldots,n$ and $\alpha_{i,j}\neq \alpha_{i',j'}$ for $(i,j)\neq (i',j')$. We call $\boldsymbol{\alpha}$ a parabolic weight
	\begin{definition}
		A $\boldsymbol{\nu}$-parabolic connection $(E,\nabla,l_*)$ is said to be $\boldsymbol{\alpha}$-stable if for any nonzero subbundle $F \subsetneq E$, the inequality 
		\[
		\frac{\deg F + \sum_{i=1}^{n} \sum_{j=1}^{r} \alpha_{i,j} \dim ((F|_{t_i} \cap l_{i,j-1})/(F|_{t_i} \cap l_{i,j}))}{\rank F}<\frac{\deg E + \sum_{i=1}^{n} \sum_{j=1}^{r} \alpha_{i,j}}{\rank E}
		\]
		holds.
	\end{definition}
	Let $\tilde{M}_{g,n}$ be a smooth algebraic scheme which is a smooth covering of the coarse moduli space of $n$ pointed irreducible smooth projective curves of genus $g$ over $\mathbb{C}$ and take a universal family $(\mcC, \tilde{\bm{t}})=(\mcC, \tilde{t}_1,\ldots,\tilde{t}_n)$ over $\tilde{M}_{g,n}$. 
	\begin{theorem}(Theorem 2.1 \cite{In})
		There exists a relative fine moduli scheme
		\[
		M^{\boldsymbol{\alpha}}_{\mcC/\tilde{M}_{g,n}}(\tilde{\bm{t}},r,d)\longrightarrow \tilde{M}_{g,n}\times \mcN_{n,r}(d)
		\]
		of $\boldsymbol{\alpha}$-stable parabolic connections of rank $r$ and degree $d$, which is smooth and quasi-projective. The fiber $M^{\boldsymbol{\alpha}}_{\mcC_x}(\tilde{t}_x,\boldsymbol{\nu})$ over $(x,\boldsymbol{\nu}) \in \tilde{M}_{g,n}\times \mcN_{n,r}(d)$ is the moduli space of $\boldsymbol{\alpha}$-stable $\boldsymbol{\nu}$-parabolic connections over $(\mcC_x,\tilde{t}_x)$ whose dimension is $2r^2(g-1)+nr(r-1)+2$. 
	\end{theorem}
	
	\subsection{Parabolic $\phi$-connections}
	In this subsection, we introduce the notion of parabolic $\phi$-connections. 
	\begin{definition}\label{phiconn}
		For $\boldsymbol{\nu} \in \mcN_{n,r}(d)$, a $\boldsymbol{\nu}$-parabolic $\phi$-connection of rank $r$ and degree $d$ over $(C,\bm{t})$ is a collection $(E_1,E_2,\phi,\nabla, l^{(1)}_*=\{l^{(1)}_{i,*}\}_{1\leq i\leq n}, l^{(2)}_*=\{l^{(2)}_{j,*}\}_{1\leq j\leq n})$ consisting of the following data:
		\begin{itemize}
			\setlength{\itemsep}{0cm}
			\item[(1)] $E_1$ and $E_2$ are vector bundles on $C$ of rank $r$ and degree $d$,
			\item[(2)] $l^{(k)}_{i,*}$ is a filtration $E_k|_{t_i}=l^{(k)}_{i,0} \supsetneq l^{(k)}_{i,1} \supsetneq  \cdots \supsetneq l^{(k)}_{i,r}=\{0\}$ for $k=1,2$ and $i=1,\ldots,n$,
			\item[(3)] $\phi \colon E_1 \rightarrow E_2$ is a homomorphism such that $\phi_{t_i}(l^{(1)}_{i,j}) \subset l^{(2)}_{i,j}$ for any $1\leq i\leq n$ and $1\leq j\leq r-1$, where $\phi_{t_i}$ is a $\mathbb{C}$-linear homomorphism induced by $\phi$, and 
			\item[(4)] $\nabla \colon E_1 \rightarrow E_2 \otimes \Omega^1_C(D(\bm{t}))$ is a logarithmic $\phi$-connection, i.e. $\nabla(fs)=\phi(s)\otimes df + f \nabla(s)$ for any $f \in \mathcal{O}_C, s \in E_1$, and $\nabla$ satisfies $(\res_{t_i}\nabla-\nu_{i,j}\phi_{t_i})(l^{(1)}_{i,j}) \subset l^{(2)}_{i,j+1}$ for any $1\leq i\leq n$ and $0\leq j\leq r-1$. 
		\end{itemize}
	\end{definition}
	
	\begin{definition}
		We say that two $\boldsymbol{\nu}$-parabolic $\phi$-connections $(E_1,E_2,\phi,\nabla, l^{(1)}_*, l^{(2)}_*), (E'_1,E'_2,\phi',\nabla', l'^{(1)}_*, l'^{(2)}_*)$
		are isomorphic to each other if there are isomorphisms $\sigma_1 \colon E_1 \overset{\sim}{\longrightarrow} E'_1$ and $\sigma_2 \colon E_2 \overset{\sim}{\longrightarrow} E'_2$ such that the diagrams
		\[
		\begin{tikzcd}
			E_1 \arrow[r,"\phi"] \arrow[d, "\sigma_1"']&E_2\arrow[d, "\sigma_2"] &&E_1 \arrow[r,"\nabla"] \arrow[d, "\sigma_1"']&E_2 \otimes  \Omega_C^1(D)\arrow[d, "\sigma_2 \otimes \id "]\\
			E'_1 \arrow[r,"\phi'"]&E'_2 &&E'_1 \arrow[r,"\nabla'"]&E'_2 \otimes \Omega_C^1(D)
		\end{tikzcd}
		\]
		commute and $(\sigma_k)_{t_i}(l^{(k)}_{i,j})=l'^{(k)}_{i,j}$ for $k=1,2$, $1\leq i\leq n$ and $0\leq j\leq r-1$.
	\end{definition}
	
	For a $\boldsymbol{\nu}$-parabolic connection $(E,\nabla,l_*)$, the collection $(E,E,\id,\nabla,l_*,l_*)$ is a $\boldsymbol{\nu}$-parabolic $\phi$-connection. It is easy to see that a $\boldsymbol{\nu}$-parabolic $\phi$-connection whose $\phi$ is an isomorphism is isomorphic to a $\boldsymbol{\nu}$-parabolic $\phi$-connection induced by a $\boldsymbol{\nu}$-parabolic connection. This implies that the moduli space of parabolic connections is a Zariski open subset of the moduli space of parabolic $\phi$-connections and that the locus of parabolic $\phi$-connections whose $\phi$ is not an isomorphism appears as the boundary of the moduli space of parabolic connections.
	
	\begin{remark}\label{phiconndif}
		The notion of rank two parabolic $\phi$-connections was introduced by Inaba, Iwasaki, and Saito (see Definition 2.5. in \cite{IIS1}), slightly different from the present definition. The difference lies in whether or not parabolic structures $l^{(2)}_*$ of $E_2$ are considered. In general we can not canonically obtain parabolic $\phi$-connections in the sense of this paper from parabolic $\phi$-connections in that of \cite{IIS1}. For example, let $(E,\{l_i\}_{1\leq i\leq n})$ be a rank 2 parabolic bundle over $(C,(t_1,\ldots,t_n))$ with the determinant $L$ and $\Phi \colon E\rightarrow E \otimes \Omega_C^1(t_1+\cdots +t_n)$ be a parabolic Higgs bundle of rank 2. Let us fix an isomorphism $\varphi \colon \wedge^2 E\overset{\sim}{\rightarrow}L$. We put $E_1=E_2=E$ and $l^{(1)}_i=l_i$ for $1\leq i\leq n$. Take a point $t_{n+1} \in C\setminus \{t_1,\ldots,t_n\}$. Let $l^{(1)}_{n+1}\subset E|_{t_{n+1}}$ be a one dimensional subspace and $\Psi$ be the composite 
		\[
		E\overset{\Phi}{\rightarrow} E \otimes \Omega_C^1(t_1+\cdots +t_n) \rightarrow E \otimes \Omega_C^1(t_1+\cdots +t_n+t_{n+1}).
		\]
		Then $(E_1,E_2,0,\Psi,\varphi,\{l^{(i)}\}_{1\leq i\leq n+1})$ becomes a parabolic $\phi$-connection in the sense of \cite{IIS1}. However $l^{(2)}_{n+1} \subset E_2|_{t_{n+1}}$ is not uniquely determined by $(E_1,E_2,0,\Psi,\varphi,\{l^{(i)}\}_{1\leq i\leq n+1})$.
		
		We require that a parabolic $\phi$-connection whose $\phi$ is an isomorphism is isomorphic to a parabolic $\phi$-connection induced by a parabolic connection.  
		In the case $r=2$, a one dimensional subspace $l^{(1)}_{i,1}(=l^{(1)}_{i,r-1})$ of $E_1|_{t_i}$ is constrained by the condition $(\res_{t_i}(\nabla)-\nu_{i,j}\phi_{t_i} )({l^{(1)}_{i,1}})=0$. In particular, a parabolic structure of $E_2$ may not be required. When $r\geq 3$, we have to impose $l^{(1)}_{i,1}, \ldots, l^{(1)}_{i,r-2}$ the condition such that a parabolic $\phi$-connection comes from a parabolic connection when $\phi$ is an isomorphism. For this reason, we introduce $l^{(2)}_{i,*}$ and the condition (3) of Definition \ref{phiconn}. 
	\end{remark}
	Let $\gamma$ be a positive integer. Through this section we assume that $\gamma$ is sufficiently large. Take a set of rational numbers $\boldsymbol{\alpha}=\{\alpha^{(k)}_{i,j}\}^{k=1,2}_{1\leq i\leq n, 1\leq j\leq r}$ satisfying  $0 \leq \alpha^{(k)}_{i,1} <\cdots<\alpha^{(k)}_{i,r} <1$ for $k=1,2$ and $i=1,\ldots,n$, and  $\alpha^{(k)}_{i,j}\neq \alpha^{(k)}_{i',j'}$ for $(i,j)\neq (i',j')$.
	\begin{definition}
		A $\boldsymbol{\nu}$-parabolic $\phi$-connection $(E_1,E_2,\phi,\nabla, l^{(1)}_*, l^{(2)}_*)$ is $\boldsymbol{\alpha}$-stable (resp. $\boldsymbol{\alpha}$-semistable) if for any subbundles $F_1 \subseteq E_1, F_2 \subseteq E_2, (F_1,F_2)\neq (0,0)$ satisfying $\phi(F_1)\subset F_2$ and $\nabla(F_1) \subset F_2 \otimes \Omega^1_C(D(\bm{t}))$,  the inequality
		\begin{align*}
			&\frac{\deg F_1+\deg F_2(-\gamma)+\sum_{i=1}^{n}\sum_{j=1}^{r}\alpha^{(1)}_{i,j}d^{(1)}_{i,j}(F_1)+\sum_{i=1}^{n}\sum_{j=1}^{r}\alpha^{(2)}_{i,j}d^{(2)}_{i,j}(F_2)}{\rank F_1+\rank F_2}\\
			&\underset{(\text{resp.}\;\leq)}{<}\frac{\deg E_1+\deg E_2(-\gamma)+\sum_{i=1}^{n}\sum_{j=1}^{r}\alpha^{(1)}_{i,j}d^{(1)}_{i,j}(E_1)+\sum_{i=1}^{n}\sum_{j=1}^{r}\alpha^{(2)}_{i,j}d^{(2)}_{i,j}(E_2)}{\rank E_1+\rank E_2}
		\end{align*}
		holds, where $d^{(k)}_{i,j}(F)=\dim (F|_{t_i}\cap l^{(k)}_{i,j-1})/(F|_{t_i}\cap l^{(k)}_{i,j})$ for a subbundle $F \subset E_k$ and for $k=1,2$.
	\end{definition}
	Take a universal family $(\mcC, \tilde{\bm{t}})=(\mcC, \tilde{t}_1,\ldots,\tilde{t}_n)$ over $\tilde{M}_{g,n}$ and put $D=\tilde{t}_1+\cdots +\tilde{t}_n$ . Then $D$ is an effective Cartier divisor flat over $\tilde{M}_{g,n}$. For simplicity of notation, we use the same character $D$ to denote the pullback of $D$ by the projection $\mcC\times \mcN \rightarrow \mcC$, where $\mcN:=\mcN_{n,r}(d)$. Let $\tilde{\nu}_{i,j} \subset \mathbb{C}\times \tilde{M}_{g,n}\times \mcN$ be the section defined by
	\[
	\tilde{M}_{g,n} \times \mcN \hookrightarrow \mathbb{C} \times \tilde{M}_{g,n} \times \mcN; \quad (x, (\nu_{k,l})^{1\leq k\leq n}_{0\leq l \leq r-1}) \mapsto (\nu_{i,j}, x, (\nu_{k,l})^{1\leq k\leq n}_{0\leq l \leq r-1}) .
	\]
	\begin{definition}
		We define the moduli functor $\overline{\mcM^{\boldsymbol{\alpha}}_{\mcC/\tilde{M}_{g,n}}}(\tilde{\bm{t}},r,d)$ from the category of locally noetherian schemes over $\tilde{M}_{g,n}\times \mcN$ to the category of sets by
		\[
		\overline{\mcM^{\boldsymbol{\alpha}}_{\mcC/\tilde{M}_{g,n}}}(\tilde{\bm{t}},r,d)(S):=\{(E_1,E_2,\phi,\nabla,l^{(1)}_*,l^{(2)}_*)\}/\sim,
		\]
		where $S$ is a locally noetherian scheme over $\tilde{M}_{g,n}\times \mcN$ and
		\begin{itemize}
			\setlength{\itemsep}{0cm}
			\item[(1)]$E_1,E_2$ are vector bundles on $(\mcC\times \mcN)_S:=(\mcC\times \mcN)\times_{\tilde{M}_{g,n}\times \mcN} S$ such that for any geometric point $s$ of $S$, $\rank (E_1)_s=\rank (E_2)_s =r$ and $\deg (E_1)_s=\deg (E_2)_s=d$,
			\item[(2)] for each $k=1,2$, $E_k|_{(\tilde{t}_i)_S}=l^{(k)}_{i,0} \supsetneq \cdots \supsetneq l^{(k)}_{i,r-1} \supsetneq l^{(k)}_{i,r}=0$ is a filtration  by subbundles,
			\item[(3)] $\phi \colon E_1\rightarrow E_2$ is a homomorphism such that $\phi_{(\tilde{t}_i)_S}(l^{(1)}_{i,j})\subset l^{(2)}_{i,j}$ for each $k=1,2$, $1\leq i\leq n$ and $1\leq j\leq r-1$,
			\item[(4)] $\nabla \colon  E_1\rightarrow E_2\otimes \Omega^1_ {(\mcC\times \mcN)_S/S}(D_S)$ is a relative logarithmic $\phi$-connection such that $(\res_{(\tilde{t}_i)_S}\nabla -(\tilde{\nu}_{i,j})_S\phi_{(\tilde{t}_i)_S})(l^{(1)}_{i,j})\subset l^{(2)}_{i,j+1}$ for each $k=1,2$, $1\leq i\leq n$ and $0\leq j\leq r-1$,
			\item[(5)] for any geometric point $s$ of $S$, the parabolic $\phi$-connection $((E_1)_s, (E_2)_s, \phi_s, \nabla_s,(l^{(1)}_*)_s,(l^{(2)}_*)_s)$ is $\boldsymbol{\alpha}$-stable. 
		\end{itemize}
	\end{definition}
	
	\subsection{Elementary transformations of parabolic $\phi$-connections}\label{eletr}
	Let $(E_1,E_2,\phi,\nabla,l^{(1)}_*,l^{(2)}_*)$ be a $\boldsymbol{\nu}$-parabolic $\phi$-connection of rank $r$ and degree $d$ over $(C,\bm{t})$. We construct a new parabolic $\phi$-connection as follows. Let us fix integers $1\leq p\leq n$ and $0\leq q\leq r$. Put $E'_k:=\ker (E_k\rightarrow E_k|_{t_p}/l^{(k)}_{p,q})$ for $k=1,2$.
	Then $E'_k$ is a locally free sheaf of rank $r$ and degree $d-q$, and we have $\phi(E'_1)\subset E'_2$ and $\nabla(E'_1)\subset E'_2\otimes \Omega_C^1(D(\bm{t}))$. Let $\phi'\colon E'_1\rightarrow E'_2$ and $\nabla'\colon E'_1\rightarrow E'_2\otimes \Omega_C^1(D(\bm{t}))$ be the restrictions of $\phi$ and $\nabla$, respectively. Let $l^{(k)}_{p,j}(-t_p)$ be the subspace of $E_k(-t_p)|_{t_p}$ induced by $l^{(k)}_{p,j}\subset E_k|_{t_p}$. The surjection $E'_k \overset{}{\rightarrow } l^{(k)}_{p,q}$ induces an exact sequence 
	\[
	0 \longrightarrow l^{(k)}_{p,q}(-t_p)\longrightarrow E_k(-t_p)|_{t_p} \overset{\iota^{(k)}}{\longrightarrow }E'_k|_{t_p} \overset{\pi^{(k)}}{\longrightarrow } l^{(k)}_{p,q} \longrightarrow 0.
	\]
	Put
	\[
	l'^{(k)}_{i,j}:=\left\{
	\begin{array}{ll}
		l^{(k)}_{i,j}&i\neq p \\
		(\pi^{(k)})^{-1}(l^{(k)}_{p,q+j})& i=p, \; 0 \leq j \leq r-q \\
		\iota^{(k)}(l^{(k)}_{p,j-r+q}(-t_p)) & i=p,\; r-q \leq j \leq r,
	\end{array}
	\right.
	\]
	\[
	\nu'_{i,j}:=\left\{
	\begin{array}{ll}
		\nu_{i,j} &i\neq p\\
		\nu_{i,q+j}& i=p, \; 0 \leq j \leq r-q-1 \\
		\nu_{i,j-r+q}+1 & i=p, \; r-q \leq j \leq r-1.
	\end{array}
	\right.
	\]
	Then we can see that $(E'_1,E'_2,\phi',\nabla',l'^{(1)}_*,l'^{(2)}_*)$ is a $\boldsymbol{\nu}'$-parabolic $\phi$-connection of rank $r$ and degree $d-q$ over $(C,\bm{t})$. Put
	\[
	\elm^{(k)}_{p,q}(E_1,E_2,\phi,\nabla,l^{(1)}_*,l^{(2)}_*):=(E'_1,E'_2,\phi',\nabla',l'^{(1)}_*,l'^{(2)}_*).
	\]
	$\elm^{(k)}_{p,q}$ induces a morphism of functors
	\[
	\elm^{(k)}_{p,q}\colon \overline{\mcM^{\boldsymbol{\alpha}}_{\mcC/\tilde{M}_{g,n}}}(\tilde{\bm{t}},r,d)\longrightarrow \overline{\mcM^{\boldsymbol{\alpha}'}_{\mcC/\tilde{M}_{g,n}}}(\tilde{\bm{t}},r,d-q),
	\]
	where $\boldsymbol{\alpha}'$ is a suitable parabolic weight. 
	Let $b^{(k)}_{p}$ be a morphism of functors defined by tensoring with $(\mcO_C(t_p),d)$, i.e.
	\[
	b^{(k)}_{p}\colon \overline{\mcM^{\boldsymbol{\alpha}}_{\mcC/\tilde{M}_{g,n}}}(\tilde{\bm{t}},r,d)\longrightarrow \overline{\mcM^{\boldsymbol{\alpha}}_{\mcC/\tilde{M}_{g,n}}}(\tilde{\bm{t}},r,d+r),\; \bm{E}\longmapsto\bm{E}\otimes (\mcO_C(t_p),d).
	\]
	Then we can see that 
	\[
	b^{(k)}_p\circ \elm^{(k)}_{p,r-q}\circ \elm^{(k)}_{p,q}=\id, \quad \elm^{(k)}_{p,q}\circ b^{(k)}_p\circ \elm^{(k)}_{p,r-q}=\id.
	\]
	So $\elm^{(k)}_{p,q}$ is an isomorphism. Hence we can freely change degree. This is important to prove that the moduli space of stable parabolic $\phi$-connections is fine.
	
	\subsection{Parabolic $\Lambda^1_{D}$-triple}
	Let $D$ be an effective Cartier divisor on $C$. We define an $\mcO_C$-bimodule structure on $\Lambda^1_{D}=\mcO_C \oplus (\Omega_C^1(D))^\vee$ by
	\begin{align*}
		(a,v)f:=(fa+\langle v,df\rangle,fv) , \; f(a,v):=(fa,fv)
	\end{align*}
	for $a,f \in \mcO_C$ and $v\in (\Omega_C^1(D))^\vee$, where $\langle\;,\;\rangle\colon (\Omega_C^1(D))^\vee \times \Omega_C^1(D) \rightarrow \mcO_C$ is the canonical pairing. 
	Let $\phi\colon E_1 \rightarrow E_2$ be a homomorphism of vector bundles on $C$ and $\nabla\colon E_1 \rightarrow E_2\otimes \Omega_C^1(D)$  be a $\phi$-connection. We define $\Phi \colon \Lambda^1_D \otimes_{\mcO_X}E_1 \rightarrow E_2$ by $\Phi((a,v)\otimes s)=a\phi(s)+\langle v,\nabla s\rangle$. Then we can easily see that $\Phi$ becomes a left $\mcO_C$-homomorphism. Conversely, let $\Phi \colon \Lambda^1_D \otimes_{\mcO_X}E_1 \rightarrow E_2$ be a left $\mcO_C$-homomorphism. We define a homomorphism $\phi \colon E_1\rightarrow E_2$ by $\phi(s)=\Phi((1,0)\otimes s)$. Let $\nabla\colon E_1 \rightarrow E_2\otimes \Omega_C^1(D)$ be a map satisfying $\Phi((0,v)\otimes s)=\langle v,\nabla s\rangle$ for any $v\in (\Omega_C^1(D))^\vee$ and $s\in E_1$. Then $\nabla$ is uniquely determined and $\nabla$ becomes a $\phi$-connection. The above correspondence is inverse to each other.
	\begin{definition}
		A parabolic $\Lambda^1_D$-triple is a collection $(E_1,E_2,\Phi, F_*(E_1), F_*(E_2))$ consisting of the following data:
		\begin{itemize}
			\setlength{\itemsep}{0cm}
			\item[(1)] $E_1$ and $E_2$ are vector bundles on $C$ of rank $r$ and degree $d$.
			\item[(2)] $F_*(E_k)$ is a filtration $E_k=F_1(E_k) \supset F_2(E_k) \supset  \cdots \supset F_{l_i}(E_k)\supset F_{l_i+1}(E_k)=E_k(-D)$ for $k=1,2$.
			\item[(3)] $\Phi \colon \Lambda^1_D \otimes_{\mcO_X}E_1 \rightarrow E_2$ is a left $\mcO_C$-homomorphism. 
		\end{itemize}
	\end{definition}
	\begin{remark}
		A parabolic $\Lambda^1_D$-triple in \cite{IIS1} is a collection $(E_1,E_2,\Phi, F_*(E_1))$ consisting of vector bundles $E_1, E_2$, a left $\mcO_C$-homomorphism $\Phi \colon \Lambda^1_D \otimes E_1 \rightarrow E_2$ and a filtration $F_*(E_1)$ of $E_1$. So forgetting a filtration $F_*(E_2)$ of a present parabolic $\Lambda^1_D$-triple $(E_1,E_2,\Phi,F_*(E_1),F_*(E_2))$, we obtain a parabolic $\Lambda^1_D$-triple $(E_1,E_2,\Phi, F_*(E_1))$ in their sense.
	\end{remark}
	\begin{definition}
		A parabolic $\Lambda^1_D$-triple $(E'_1,E'_2,\Phi',F_*(E'_1),F_*(E'_2))$ is said to be a parabolic $\Lambda^1_D$-subtriple of $(E_1,E_2,\Phi,F_*(E_1),F_*(E_2))$ if $E'_1 \subset E_1$, $E'_2 \subset E_2$, $\Phi'=\Phi|_{\Lambda^1_D \otimes_{\mcO_X} E'_1} $, $F_i(E'_1) \subset F_i(E_1)$ and $F_i(E'_2) \subset F_i(E_2)$.
	\end{definition}
	For each $k=1,2$, let $\boldsymbol{\beta}^{(k)}=\{\beta^{(k)}_i\}_{1\leq i\leq l_k}$ be a collection of rational numbers with  $0\leq \beta^{(k)}_1<\cdots<\beta^{(k)}_{l_k}<1$.
	
	For a parabolic $\Lambda^1_D$-triple $(E_1,E_2,\Phi,F_*(E_1),F_*(E_2))$, we put 
	\begin{align*}
		&\mu_{\boldsymbol{\beta}}((E_1,E_2,\Phi,F_*(E_1),F_*(E_2)))\\
		:=&\frac{\deg E_1(-D)+\deg E_2(-D)-\gamma \deg \mcO_X(1)\rank E_2}{\rank E_1+\rank E_2}\\
		&\quad+\frac{\sum_{i=1}^{l_1}\beta^{(1)}_{i}\,\len F_i(E_1)/ F_{i+1}(E_1)+\sum_{i=1}^{l_2}\beta^{(2)}_{i}\,\len F_i(E_2)/ F_{i+1}(E_2)}{\rank E_1+\rank E_2}.
	\end{align*}
	\begin{definition}
		A parabolic $\Lambda^1_D$-triple $(E_1,E_2,\Phi,F_*(E_1),F_*(E_2))$ is ${\boldsymbol{\beta}}$-stable if for any nonzero proper parabolic subtriple $(E'_1,E'_2,\Phi',F_*(E'_1),F_*(E'_2))$ of $(E_1,E_2,\Phi,F_*(E_1),F_*(E_2))$, the inequality
		\[
		\mu_{\boldsymbol{\beta}}((E'_1,E'_2,\Phi',F_*(E'_1),F_*(E'_2)))< \mu_{\boldsymbol{\beta}}((E_1,E_2,\Phi,F_*(E_1),F_*(E_2)))
		\]
		holds.
	\end{definition}
	Let $S$ be a connected noetherian scheme and $\pi_S \colon X\rightarrow S$ be a smooth projective morphism whose geometric fibers are irreducible smooth curves of genus $g$. Let $D \subset X$ be a relative effective Cartier divisor for $\pi_S$. 
	\begin{definition}
		We define the moduli functor $\overline{\mcM^{D,\boldsymbol{\beta}}_{X/S}}(r,d,\bm{d}_1=\{d^{(1)}_i\}_{2\leq i\leq l_1},\bm{d}_2=\{d^{(2)}_i\}_{2\leq i\leq l_2})$ of the category of locally noetherian schemes over $S$ to the category of sets by
		\[
		\overline{\mcM^{D,\boldsymbol{\beta}}_{X/S}}(r,d,\bm{d}_1,\bm{d}_2)(T):=\{(E_1,E_2,\Phi,F_*(E_1),F_*(E_2))\}/\sim
		\]
		where $T$ is a locally noetherian scheme over $S$ and
		\begin{itemize}
			\setlength{\itemsep}{0cm}
			\item[(1)]$E_1,E_2$ are vector bundles on $X\times_S T$ such that for any geometric point $s$ of $T$, $\rank (E_1)_s=\rank (E_2)_s =r$ and $\deg (E_1)_s=\deg (E_2)_s=d$,
			\item[(2)] $\Phi \colon \Lambda^1_{D/S} \otimes E_1 \rightarrow E_2$ is a homomorphism of left $\mcO_{X\times_S T}$-modules,
			\item[(3)] For each $k=1,2$, $E_k=F_1(E_k) \supset \cdots \supset F_{l_k}(E_k) \supset F_{l_k+1}(E_k)=E_k(-D_T)$ is a filtration of $E_1$ by coherent subsheaves such that each $E_k/F_{i}(E_k)$ is flat over $T$ and for any geometric point $s$ of $T$ and $2\leq i \leq l_k$, $\len (E_k/F_{i}(E_k))_s=d^{(k)}_i$,
			\item[(4)] for any geometric point $s$ of $T$, the parabolic $\Lambda^1_{D_s}$-triple $((E_1)_s, (E_2)_s, \Phi_s,F_*(E_1)_s,F_*(E_2)_s)$ is $\boldsymbol{\beta}$-stable. 
		\end{itemize}
	\end{definition}
	\subsection{Construction of moduli spaces}

	We introduce propositions and a lemma.
	\begin{proposition}\label{bound}
		The family of geometric points of $\overline{\mcM^{D,\boldsymbol{\beta},\gamma}_{X/S}}(r,d,\bm{d}_1,\bm{d}_2)$ is bounded.
	\end{proposition}
	\begin{proof}
		See Proposition 5.1 in \cite{IIS1}.
	\end{proof}
	\begin{proposition}\label{H0}
		Put $\beta^{(1)}_{l_1+1}=\beta^{(2)}_{l_2+1}=1$ and $\epsilon^{(k)}_i=\beta^{(k)}_{i+1}-\beta^{(k)}_i$ for $k=1,2$ and $1\leq i\leq l_k$. There exists an integer $m_0$ such that for any geometric point $(E_1,E_2,\Phi,F_*(E_1),F_*(E_2))$ of $\overline{\mcM^{D,\boldsymbol{\beta}}_{X/S}}(r,d,\bm{d}_1,\bm{d}_2)(K)$, the inequality 
		\begin{align*}
			&\frac{\beta^{(1)}_1 h^0(E'_1(m))+\beta^{(2)}_1 h^0(E'_2(m-\gamma))+\sum_{i=1}^{l_1}\epsilon^{(1)}_i h^0(F_{i+1}(E'_1)(m)))+\sum_{i=1}^{l_2}\epsilon^{(2)}_i h^0(F_{i+1}(E'_2)(m-\gamma)))}{\rank E'_1+ \rank E'_2}\\
			&< \frac{\beta^{(1)}_1h^0(E_1(m))+\beta^{(2)}_1 h^0(E_2(m-\gamma))+\sum_{i=1}^{l_1}\epsilon^{(1)}_i h^0(F_{i+1}(E_1)(m)))+\sum_{i=1}^{l_2}\epsilon^{(2)}_i h^0(F_{i+1}(E_2)(m-\gamma)))}{\rank E_1+ \rank E_2}
		\end{align*}
		holds for any proper nonzero parabolic $\Lambda^1_{D_K}$-subtriple $(E'_1,E'_2,\Phi,F_*(E'_1),F_*(E'_2))$ and any integer $m\geq m_0$.
	\end{proposition}
	\begin{proof}
		See Proposition 5.2 in \cite{IIS1}.
	\end{proof}
	\begin{proposition}\label{open}
		Let $T$ be a noetherian scheme over $S$ and $(E_1,E_2,\Phi,F_*(E_1),F_*(E_2))$ be a flat family of parabolic $\Lambda^1_{D_T/T}$-triples on $X\times_ST$ over $T$. Then there exists an open subscheme $T^s$ of $T$ such that 
		\[
		T^s(K)=\{s \in T(K) \mid \text{$(E_1,E_2,\Phi,F_*(E_1),F_*(E_2)) \otimes k(s)$ is $\boldsymbol{\beta}$-stable.}\}
		\]
		for any algebraically closed field $K$.
	\end{proposition}
	\begin{proof}
		See Proposition 5.3 in \cite{IIS1}.
	\end{proof}
	\begin{proposition}\label{AK}(EGA III (7.7.8), (7.7.9) or \cite{AK} (1.1))
		Let $f \colon X\rightarrow S$ be a proper morphism of noetherian schemes, and let $I$ and $F$ be two coherent $\mcO_X$-modules with $F$ flat over $S$. Then there exist a coherent $\mcO_S$ module $H(I,F)$ and an element $h(I,F)$ of $\Hom_X(I,F\otimes_S H(I,F))$ which represents the functor 
		\[
		M \longmapsto \Hom_X(I,F\otimes_{\mcO_S} M)
		\] 
		defined on the category of quasi-coherent $\mcO_S$-modules $M$, and the formation of the pair commutes with base change; in other words, the Yoneda map defined by $h(I,F)$
		\[
		y\colon \Hom_T(H(I,F)_T,M)\longmapsto \Hom_{X_T}(I_T,F \otimes_{\mcO_S} M)
		\]
		is an isomorphism for every $S$-scheme $T$ and every quasi-coherent $\mcO_T$-module $M$.
	\end{proposition}
	\begin{lemma} \label{Yo}(Lemma 4.3 \cite{Yo})
		Let $f \colon X\rightarrow S$ be a proper morphism of noetherian schemes and let $\phi \colon I\rightarrow F$ be an $\mcO_X$-homomorphism of coherent $\mcO_S$-modules with $F$ flat over $S$. Then there exists a unique closed subscheme $Z$ of $S$ such that for all morphism $g \colon T \rightarrow S$, $g^*(\phi)=0$ if and only if $g$ factors through $Z$.
	\end{lemma}
	We construct the moduli space of parabolic $\Lambda_D^1$-triples. Let $S$ be a connected noetherian scheme and $\pi_S \colon X\rightarrow S$ be a smooth projective morphism whose geometric fibers are irreducible smooth curves of genus $g$. Let $D \subset X$ be a relative effective Cartier divisor for $\pi_S$. Let $P(m)=rd_Xm+d+r(1-g)$ where $d_X =\deg \mcO_{X_s}(1)$ for $s \in S$. We take an integer $m_0$ in Proposition \ref{H0}. We may assume that for any $m \geq m_0$, $h^k(F_i(E_1)(m))=h^k(F_j(E_2)(m-\gamma))=0$ for $k>0$, $1\leq i\leq  l_1+1$, $1\leq j\leq l_2+1$, and $F_i(E_1)(m_0), F_j(E_2)(m_0-\gamma)$ are generated by their global sections for any geometric point $(E_1,E_2,\Phi,F_*(E_1),F_*(E_2))$ of  $\overline{\mcM^{D,\boldsymbol{\beta}}_{X/S}}(r,d,\bm{d}_1,\bm{d}_2)$ by Proposition \ref{bound}. Put $n_1=P(m_0)$ and $n_2=P(m_0-\gamma)$. Let $V_1, V_2$ be free $\mcO_S$-modules of rank $n_1,n_2$, respectively. Let $Q^{(1)}$ be the Quot-scheme $\Quot^{P}_{V_1\otimes \mcO_S(-m_0)/X/S}$ and $V_1 \otimes  \mcO_{X_{Q^{(1)}}}(-m_0) \rightarrow \mcE_1$ be the universal quotient sheaf. Let $Q^{(2)}=\Quot^{P}_{V_2\otimes \mcO_S(-m_0+\gamma)/X/S}$ and $V_2 \otimes  \mcO_{X_{Q^{(2)}}}(-m_0+\gamma) \rightarrow \mcE_2$ be the universal quotient sheaf. Put $d^{(1)}_{l_1+1}=d^{(2)}_{l_2+1}=rn$.  For $k=1,2$ and $2\leq i\leq l_k+1$, let $Q^{(k)}_i:=\Quot^{d^{(k)}_{i}}_{\mcE_k/X_{Q^{(k)}}/Q^{(k)}}$ and  $F_i(\mcE_k) \subset \mcE_k$ be the universal subsheaf. We define $Q$ as the maximal closed subscheme of 
	\[
	Q^{(1)}_2 \times_{Q^{(1)}}\cdots  \times_{Q^{(1)}}Q^{(1)}_{l_1+1} \times Q^{(2)}_2 \times_{Q^{(2)}}\cdots  \times_{Q^{(2)}} Q^{(2)}_{l_2+1}
	\]
	such that there exist filtrations
	\[
	(\mcE_1)_Q\otimes \mcO_{X_Q}(-D_Q) =F_{l_1+1}(\mcE_1)_Q\subset F_{l_1}(\mcE_1)_Q\subset\cdots\subset F_{2}(\mcE_1)_Q\subset F_{1}(\mcE_1)_Q:=(\mcE_1)_Q 
	\]
	and
	\[
	(\mcE_2)_Q\otimes \mcO_{X_Q}(-D_Q)= F_{l_2+1}(\mcE_2)_Q\subset F_{l_2}(\mcE_2)_Q\subset\cdots\subset F_{2}(\mcE_2)_Q\subset F_{1}(\mcE_2)_Q:=(\mcE_2)_Q.
	\]
	By Proposition \ref{AK} there exists a coherent sheaf $\mcH$ on $Q$ such that  for any noetherian scheme $T$ over $Q$ and for any quasi-coherent $\mcO_T$-module $\mcF$, there exists a functorial isomorphism 
	\[
	\Hom_{T}(\mcH_T,\mcF) \cong \Hom_{X_T}(\Lambda^1_{D/S}\otimes_{\mcO_X}(\mcE_1)_T,(\mcE_2)_T\otimes_{\mcO_T} \mcF).
	\]
	Let $\bm{V}=\Spec \Sym_{\mcO_Q}(\mcH)$, where $\Sym_{\mcO_Q} (\mcH)$ is the symmetric algebra of $\mcH$ on $Q$. Then the homomorphism
	\[
	\tilde{\Phi} \colon \Lambda^1_{D/S}\otimes_{\mcO_X}(\mcE_1)_{\bm{V}} \longrightarrow (\mcE_2)_{\bm{V}}
	\]
	corresponding to the natural homomorphism $\mcH_{\bm{V}}\rightarrow \mcO_{\bm{V}}$ is the universal homomorphism. Put
	\[
	R^s:=
	\left\{ s \in \bm{V} \vb
	\begin{minipage}{10cm}
		$(V_1)_s \rightarrow H^0((\mcE_1)_s(m_0)), (V_2)_s \rightarrow H^0((\mcE_2)_s(m_0-\gamma))$ are isomorphisms, and $((\mcE_1)_s, (\mcE_2)_s, \tilde{\Phi}_s,F_*(\mcE_1)_s, F_*(\mcE_2)_s)$ is $\boldsymbol{\beta}$-stable
	\end{minipage}
	\right\}.
	\]
	By Proposition \ref{open}, $R^s$ is an open subscheme of $\bm{V}$. For $y \in R^s$ and vector subspaces $V'_1 \subset V_1$ and $V'_2 \subset V_2$, let $E'_1(V'_1,V'_2,y)$ be the image of $V'_1\otimes \mcO_X(-m_0) \rightarrow (\mcE_1)_y$ and $E'_2(V'_1,V'_2,y)$ be the image of $\Lambda^1_{D/S}\otimes V'_1\otimes \mcO_X(-m_0) \oplus V'_2 \otimes \mcO_X(-m_0+\gamma) \rightarrow (\mcE_2)_y$. Since the family
	\[
	\mcF =\{(E(V'_1,V'_2,y)_1,E(V'_1,V'_2,y)_2) \mid y \in R^s, V'_1 \subset V_1, V'_2 \subset V_2\}
	\]
	is bounded, there exists an integer $m_1\geq m_0$ such that for all $m \geq m_1$ and all members $(E(V'_1,V'_2,y)_1,E(V'_1,V'_2,y)_2) \in \mcF$, 
	\[
	V'_1\otimes H^0(\mcO_{X_y}(m)) \rightarrow H^0(E(V'_1,V'_2,y)_1(m+m_0))
	\]
	and
	\[
	V'_1\otimes H^0(\mcO_{X_y}(m_0+m-\gamma)\otimes \Lambda^1_{D_y} \otimes \mcO_{X_y}(-m_0)) \oplus V'_2 \otimes H^0(\mcO_{X_y}(m))\rightarrow H^0(E(V'_1,V'_2,y)_2(m_0+m-\gamma))
	\]
	are surjective, $H^i(\mcO_{X_y}(m_0+m-\gamma)\otimes \Lambda^1_{D_y}\otimes \mcO_{X_y}(-m_0))=0, H^i(\mcO_{X_y}(m))=0$ for $i>0$,  and the inequality
	\begin{equation}\label{stableineq}
		\begin{split}
			&(r'_1+r'_2)d_X\biggl\{h^0(E_1(m_0))+h^0(E_2(m_0-\gamma)) -\sum_{i=1}^{l_1}\epsilon^{(1)}_id^{(1)}_{i+1}-\sum_{j=1}^{l_2}\epsilon^{(2)}_jd^{(2)}_{j+1}\biggr\}\\
			&\quad-2rd_X\biggl\{ h^0(E'_1(m_0))+h^0(E'_2(m_0-\gamma)) -\sum_{i=1}^{l_1}\epsilon^{(1)}_i\Big(h^0(E'_1(m_0))-h^0(F_{i+1}(E'_1)(m_0))\Big)\\
			&\qquad \qquad \qquad-\sum_{j=1}^{l_2}\epsilon^{(2)}_j\Big(h^0(E'_2(m_0-\gamma))-h^0(F_{j+1}(E'_2)(m_0-\gamma))\Big)\biggr\}\\
			>&m^{-1}\Big(\dim V_1+\dim V_2 -\sum_{i=1}^{l_1}\epsilon^{(1)}_id^{(1)}_{i+1}-\sum_{j=1}^{l_2}\epsilon^{(2)}_jd^{(2)}_{j+1}\Big)\Big(\dim V'_1+\dim V'_2-\chi(E'_1(m_0))-\chi(E'_2(m_0-\gamma))\Big)
		\end{split}
	\end{equation}
	holds for $(0,0) \subsetneq (V'_1,V'_2) \subsetneq ((V_1)_y,(V_2)_y)$, where $E'_k=E(V'_1,V'_2,y)_k$ and $F_{i+1}(E'_k)=E'_k \cap F_{i+1}(\mcE_k)_y$ for $k=1,2$ and $1\leq i\leq l_k$. We note that the left hand side of (\ref{stableineq}) is positive since $m_0$ is an integer in Proposition \ref{H0}.
	The composite
	\[
	V_1 \otimes \Lambda^1_{D/S} \otimes \mcO_{X_{R^s}}(-m_0) \longrightarrow \Lambda^1_{D/S} \otimes (\mcE_1)_{R^s}  \overset{\tilde{\Phi}}{\longrightarrow} (\mcE_2)_{R^s}
	\]
	induces a homomorphism
	\[
	V_1\otimes W_1 \otimes \mcO_{R^s} \longrightarrow (\pi_{R^s})_*(\mcE_2(m_0+m_1-\gamma)_{R^s}),
	\]
	where $W_1=(\pi_S)_*(\mcO_X(m_0+m_1 -\gamma) \otimes \Lambda^1_{D/S} \otimes \mcO_X(-m_0))$ and $\pi_{R^s}\colon X_{R^s}:=X\times_SR^s\rightarrow R^s$ be the projection, and the quotient $V_2 \otimes \mcO_{X_{R^s}}(-m_0+\gamma) \rightarrow (\mcE_2)_{R^s}$ induces a homomorphism
	\[
	V_2\otimes W_2 \otimes \mcO_{R^s} \longrightarrow (\pi_{R^s})_*(\mcE_2(m_0+m_1-\gamma)_{R^s})
	\]
	where $W_2=(\pi_S)_*(\mcO_X(m_1))$. These homomorphisms induce a quotient bundle
	\begin{equation}\label{quot2}
		(V_1 \otimes W_1 \oplus V_2 \otimes W_2)\otimes \mcO_{R^s}\longrightarrow (\pi_{R^s})_*(\mcE_2(m_0+m_1-\gamma)_{R^s}).
	\end{equation}
	Taking $m_1$ sufficiently large, we obtain the surjectivity of this homomorphism and the canonical homomorphism
	\begin{equation}\label{quot1}
		V_1 \otimes W_2 \otimes \mcO_{R^s} \longrightarrow (\pi_{R^s})_*(\mcE_1(m_0+m_1)_{R^s}).
	\end{equation}
	The canonical homomorphisms
	\begin{equation}\label{quot1i}
		V_1 \otimes \mcO_{R^s} \longrightarrow (\pi_{R^s})_*((\mcE_1/F_i(\mcE_1))(m_0)_{R^s}),
	\end{equation}
	\begin{equation}\label{quot2i}
		V_2 \otimes \mcO_{R^s} \longrightarrow (\pi_{R^s})_*((\mcE_2/F_i(\mcE_2))(m_0-\gamma)_{R^s})
	\end{equation}
	are surjective. Indeed, set
	\[
	\mcG_1=\ker (V_1 \otimes \mcO_{X_{R^s}}(-m_0) \rightarrow (\mcE_1)_{R^s}),
	\] 
	\[
	\mcG^{(1)}_i=\ker (V_1 \otimes \mcO_{X_{R^s}}(-m_0) \rightarrow (\mcE_1/F_i(\mcE_1))_{R^s}).
	\] 
	Then we obtain a commutative diagram
	\[
	\begin{tikzcd}
		V_1\otimes \mcO_{R^s} \ar[d,"="] \ar[r] &(\pi_{R^s})_*(\mcE_1(m_0))_{R^s}\ar[r,"\delta"]  \ar[d]& R^1\pi_{R^s*}(\mcG_1(m_0)) \ar[d]\\
		V_1\otimes \mcO_{R^s} \ar[r] &(\pi_{R^s})_*(\mcE_1/F_i(\mcE_1)(m_0))_{R^s}\ar[r]&R^1\pi_{R^s*} (\mcG^{(1)}_i(m_0))
	\end{tikzcd}.
	\]
	Since $H^1(F_i(\mcE_1)_y(m_0))=0$ and $V_1 \cong H^0((\mcE_1)_y(m_0))$ for any $y \in R^s$, the middle homomorphism is surjective and $\delta=0$. So the homomorphism $V_1\otimes \mcO_{R^s} \rightarrow (\pi_{R^s})_*(\mcE_1/F_i(\mcE_1)(m_0))_{R^s}$ is surjective. Similarly, we obtain the surjectivity of the homomorphism $V_2 \otimes \mcO_{R^s} \rightarrow (\pi_{R^s})_*(\mcE_2/F_i(\mcE_2)(m_0-\gamma)_{R^s})$. 
	The quotients (\ref{quot2}), (\ref{quot1}), (\ref{quot1i}) and (\ref{quot2i}) determine a morphism
	\[
	\iota\colon R^s \longrightarrow \Grass_{r_2}(V_1\otimes W_1 \oplus V_2 \otimes W_2) \times \Grass_{r_1}(V_1 \otimes W_2) \times \prod_{i=1}^{l_1} \Grass_{d^{(1)}_{i+1}}(V_1) \times \prod_{i=1}^{l_2} \Grass_{d^{(2)}_{i+1}}(V_2),
	\]
	where $r_1=h^0(\mcE_1(m_0+m_1)_y), r_2=h^0(\mcE_2(m_0+m_1-\gamma)_y)$ for any $y \in R^s$. We can see that $\iota$ is a closed immersion.
	
	Let $G:=(GL(V_1)\times_S GL(V_2))/(\bm{G}_m \times S)$. Here $\bm{G}_m \times S$ is the subgroup of  $GL(V_1) \times_S GL(V_2)$ consisting of  all scalar matrices. The group $G$ acts canonically on $R^s$ and on $\Grass_{r_2}(V_1\otimes W_1 \oplus V_2 \otimes W_2) \times \Grass_{r_1}(V_1 \otimes W_2) \times \prod_{i=1}^{l_1} \Grass_{d^{(1)}_{i+1}}(V_1) \times \prod_{i=1}^{l_2} \Grass_{d^{(2)}_{i+1}}(V_2)$. We can see that $\iota$ is a $G$-equivariant immersion. Let $\mcO_{\Grass_{r_2}(V_1\otimes W_1 \oplus V_2 \otimes W_2)}(1)$, $\mcO_{\Grass_{r_1}(V_1 \otimes W_2)}(1)$, $\mcO_{\Grass_{d^{(1)}_i}(V_1)}(1)$, $\mcO_{\Grass_{d^{(2)}_i}(V_2)}(1)$ be the $S$-ample line bundle on $\Grass_{r_2}(V_1\otimes W_1 \oplus V_2 \otimes W_2)$, $\Grass_{r_1}(V_1 \otimes W_2)$, $\Grass_{d^{(1)}_i}(V_1)$, $\Grass_{d^{(2)}_i}(V_2)$,  respectively,  induced by Pl\"ucker embedding. For $i=1,\ldots, l_1$ and $j=1,\ldots, l_2$,  we define positive rational numbers $\xi,  \xi^{(1)}_i, \xi^{(2)}_j$ by 
	\begin{equation}\label{xidef}
		\xi=P(m_0)+P(m_0-\gamma)-\sum_{i=1}^{l_1}\epsilon^{(1)}_id^{(1)}_{i+1}-\sum_{j=1}^{l_2}\epsilon^{(2)}_jd^{(2)}_{j+1}, \quad\xi^{(1)}_i=2rd_Xm_1\epsilon^{(1)}_i, \quad \xi^{(2)}_i=2rd_Xm_1\epsilon^{(2)}_i.
	\end{equation}
	Put
	\[
	L:=\iota^*\Big(\mcO_{\Grass_{r_2}(V_1\otimes W_1 \oplus V_2 \otimes W_2)}(\xi)\otimes \mcO_{\Grass_{r_1}(V_1 \otimes W_2)}(\xi)\otimes \bigotimes_{i=1}^{l_1}\mcO_{\Grass_{d^{(1)}_{i+1}}(V_1)}(\xi^{(1)}_i) \otimes \bigotimes_{j=1}^{l_2}\mcO_{\Grass_{d^{(2)}_{j+1}}(V_2)}(\xi^{(2)}_j) \Big). 
	\]
	Then $L$ is a $\mathbb{Q}$-line bundle on $R^s$ and for some positive integer $N$, $L^{\otimes N}$ becomes a $G$ -linearized $S$-ample line bundle on $R^s$.
	\begin{proposition}\label{prostab}
		All points of $R^s$ are properly stable with respect to the action of $G$ and the $G$-linearized $S$-ample line bundle $L^{\otimes N}$.
	\end{proposition}
	
	By Proposition \ref{prostab}, there exists a geometric quotient $R^s/G$. 
	\begin{theorem}
		$\overline{M^{D,\boldsymbol{\beta}}_{X/S}}(r,d,\bm{d}_1,\bm{d}_2):=R^s/G$ is a coarse moduli scheme of $\overline{\mcM^{D,\boldsymbol{\beta}}_{X/S}}(r,d,\bm{d}_1,\bm{d}_2)$.
	\end{theorem}
	
	\begin{lemma}\label{scalar}
		Take any geometric point $(E_1,E_2,\Phi,F_*(E_1),F_*(E_2)) \in \overline{\mcM^{D,\boldsymbol{\beta}}_{X/S}}(r,d,\bm{d}_1,\bm{d}_2)(K)$. Then for any endomorphisms $f_1 \colon E_1 \rightarrow E_1, f_2\colon E_2 \rightarrow E_2$ satisfying $\Phi\circ (1\otimes f_1)=f_2\circ \Phi$, $f_1(F_{j+1}(E_1))\subset F_{j+1}(E_1)\;(1\leq j\leq l_1)$ and $f_2(F_{j+1}(E_2))\subset F_{j+1}(E_2)\;(1\leq j\leq l_2)$, there exists $c \in K$ such that $(f_1,f_2)=(c\cdot\id_{E_1}, c \cdot \id_{E_2})$.
	\end{lemma}
	\begin{proof}
		See Lemma 5.1 in \cite{IIS1}.
	\end{proof}
	\begin{proposition}\label{proj}
		Let $R$ be a discrete valuation ring over $S$ with the residue field $k=R/\mfm$ and the quotient field $K$. Let $(E_1,E_2,\Phi,F_*(E_1), F_*(E_2))$ be a semistable parabolic $\Lambda^1_{D_K}$-triple on $X_K$. Then there exists a flat family $(\tilde{E_1},\tilde{E_2},\tilde{\Phi},F_*(\tilde{E_1}), F_*(\tilde{E_2}))$ of  parabolic $\Lambda^1_{D_R}$-triples on $X_R$ over $R$ such that $(E_1,E_2,\Phi,F_*(E_1), F_*(E_2)) \cong (\tilde{E_1},\tilde{E_2},\tilde{\Phi},F_*(\tilde{E_1}), F_*(\tilde{E_2}))\otimes_R K$ and $(\tilde{E_1},\tilde{E_2},\tilde{\Phi},F_*(\tilde{E_1}), F_*(\tilde{E_2}))\otimes_R k$ is semistable.
	\end{proposition}
	\begin{proof}
		See Proposition 5.5 in \cite{IIS1}.
	\end{proof}
	\begin{proof}[Proof of Theorem \ref{MT}]
		Put $l_1=l_2=rn$ and $d^{(1)}_{i}=d^{(2)}_{i}=i-1$ for $2\leq i\leq rn+1$. Put $\{\beta^{(k)}_i\}_{1\leq i\leq rn}=\{\alpha^{(k)}_{i,j}\}^{1\leq i\leq n}_{1\leq j \leq r}$ for each $k=1,2$. For a parabolic $\phi$-connection $(E_1,E_2,\phi,\nabla,l^{(1)}_*,l^{(2)}_*)$ over $(C,\bm{t})$, we define a parabolic $\Lambda^1_{D}$-triple $(E_1,E_2,\Phi,F_*(E_1), F_*(E_2))$ as follows: Let $\Phi \colon \Lambda^1_D\otimes E_1 \rightarrow E_2$ be a left $\mcO_C$-homomorphism induced by $\phi$ and $\nabla$. For each $1\leq p\leq rn$, there exists a unique pair of integers $(i,j)$ such that $1\leq i\leq n$, $1\leq j\leq r$ and $\beta^{(1)}_p=\alpha^{(1)}_{i,j}$. Then we put $F_1(E_1):=E_1$ and  $F_{p+1}(E_1):=\ker (F_{p}(E_1)\rightarrow E_1|_{t_{i}}/l^{(1)}_{i,j})$. In a similar way we define $F_{p}(E_2)$ for $1\leq p\leq rn+1$. By the definition of the stability  we can see that $(E_1,E_2,\phi,\nabla,l^{(1)}_*,l^{(2)}_*)$ is $\boldsymbol{\alpha}$-stable if and only if $(E_1,E_2,\Phi,F_*(E_1), F_*(E_2))$  is $\boldsymbol{\beta}$-stable. The above correspondence determines a morphism of functors
		\[
		\iota \colon \overline{\mcM^{\boldsymbol{\alpha}}_{\mcC/\tilde{M}_{g,n}}}(\tilde{\bm{t}},r,d)\longrightarrow \overline{\mcM^{D,\boldsymbol{\beta}}_{\mcC\times \mcN/\tilde{M}_{g,n}\times \mcN}}(r,d,\bm{d}_1,\bm{d}_2).
		\]
		We can see that $\iota$ is a closed immersion by Lemma \ref{Yo}. So there exists a closed subscheme $Z \subset R^s$ such that 
		\[
		h_Z=h_{R^s}\times_{\overline{\mcM^{D,\boldsymbol{\beta}}_{\mcC\times \mcN/\tilde{M}_{g,n}\times \mcN}}(r,d,\bm{d}_1,\bm{d}_2)}\overline{\mcM^{\boldsymbol{\alpha}}_{\mcC/\tilde{M}_{g,n}}}(\tilde{\bm{t}},r,d),
		\]
		where $h_Z=\Hom_{\tilde{M}_{g,n} \times \mcN}(-,Z)$. $Z$ is invariant by the action of $G$. By Lemma \ref{scalar}, the quotient $R^s\rightarrow \overline{M^{D,\boldsymbol{\beta}}_{\mcC\times \mcN/\tilde{M}_{g,n}\times \mcN}}(r,d,\bm{d}_1,\bm{d}_2)$ is a principal $G$-bundle. So $Z/G$ is a closed subscheme of $\overline{M^{D,\boldsymbol{\beta}}_{\mcC\times \mcN/\tilde{M}_{g,n}\times \mcN}}(r,d,\bm{d}_1,\bm{d}_2)$ which is just the coarse moduli scheme of $\overline{M^{\boldsymbol{\alpha}}_{\mcC/\tilde{M}_{g,n}}}(\tilde{\bm{t}},r,d)$. 
		
		When $r$ and $d$ are coprime, we can see that  $\overline{M^{\boldsymbol{\alpha}}_{\mcC/\tilde{M}_{g,n}}}(\tilde{\bm{t}},r,d)$ is fine by Lemma \ref{scalar} and the standard argument. For general $d$, there is an isomorphism $\sigma \colon \overline{M^{\boldsymbol{\alpha}}_{\mcC/\tilde{M}_{g,n}}}(\tilde{\bm{t}},r,d)\rightarrow \overline{M^{\boldsymbol{\alpha}'}_{\mcC/\tilde{M}_{g,n}}}(\tilde{\bm{t}},r,d')$ induced an elementary transformation, where $r$ and $d'$ are coprime. Then we obtain a universal family over $\overline{M^{\boldsymbol{\alpha}}_{\mcC/\tilde{M}_{g,n}}}(\tilde{\bm{t}},r,d)\times_{\tilde{M}_{g,n}\times \mcN}(\mcC\times \mcN)$ by pulling back a universal family over $\overline{M^{\boldsymbol{\alpha'}}_{\mcC/\tilde{M}_{g,n}}}(\tilde{\bm{t}},r,d')\times_{\tilde{M}_{g,n}\times \mcN}(\mcC\times \mcN)$ through $\sigma$. So $\overline{M^{\boldsymbol{\alpha}}_{\mcC/\tilde{M}_{g,n}}}(\tilde{\bm{t}},r,d)$ is fine for arbitrary $d$.
		
		It follows from Proposition \ref{proj} that  $\overline{M^{\boldsymbol{\alpha}}_{\mcC/\tilde{M}_{g,n}}}(\tilde{\bm{t}},r,d)\rightarrow \tilde{M}_{g,n}\times \mcN$ is projective for generic $\boldsymbol{\alpha}$.
	\end{proof}
	\section{Explicit description of moduli spaces of parabolic logarithmic connections}\label{secdes}
	In this section, we describe the moduli space of rank 3 parabolic logarithmic connections on $\pl$ with 3 poles. Through this section, we may assume that  $\boldsymbol{\alpha}=(\alpha_{i,j})_{1\leq i,j\leq 3}$ and $\gamma$ satisfies $0<\alpha_{i,j}\ll 1$ for any $1\leq i, j\leq 3$ and $\gamma \gg 0$. We put $\mcN:=\mcN(0,0,2)$. 
	\subsection{The family of $A^{(1)*}_2$-surfaces and main theorem}\label{familysursec}
	In this subsection, we construct a family of  $A^{(1)*}_2$-surfaces parameterized by $T_3\times \mcN$ and state the main theorem.
	
	Let $\tilde{t}_i \subset \pl \times T_3 \times \mcN$ be the section defined by 
	\[
	T_3 \times \mcN \hookrightarrow \pl \times T_3 \times \mcN; \quad ((t_j)_{1\leq j\leq 3}, (\nu_{m,n})^{1\leq m\leq 3}_{0\leq n \leq 2}) \mapsto (t_i, (t_j)_{1\leq j\leq 3}, (\nu_{m,n})^{1\leq n\leq 3}_{0\leq n \leq 2}) 
	\]
	for $i=1,2,3$ and $D(\tilde{\bm{t}})=\tilde{t}_1+\tilde{t}_2+\tilde{t}_3$ be a relative effective Cartier divisor for the projection  $\pl \times T_3 \times \mcN \rightarrow T_3 \times \mcN$. Put 
	\[
	\mcE:=\Omega_{\pl\times T_3 \times \mcN/T_3\times \mcN}^1(D(\tilde{\bm{t}}))\oplus \mcO_{\pl\times T_3 \times \mcN}.
	\]
	Let 
	\[
	\pi \colon \mathbb{P}(\mcE) \longrightarrow \pl\times T_3 \times \mcN
	\]
	be the projection, where $\mathbb{P}(\mcE):=\Proj \Sym (\mcE^\vee)$. We note that for each $x\in T_3 \times \mcN$, there is an isomorphism $(\Omega_{\pl\times T_3 \times \mcN/T_3\times \mcN}^1(D(\tilde{\bm{t}})))_x\cong \Omega_{\pl}^1(D(\tilde{\bm{t}})_x)\cong \Opl(1)$ and so $\mathbb{P}(\mcE_x)$ is a Hirzebruch surface of degree 1. Let $\tilde{D}_0 \subset \mathbb{P}(\mcE)$ be the section over $\pl\times T_3 \times \mcN$ defined by the injection $\Omega_{\pl\times T_3 \times \mcN/T_3\times \mcN}^1(D(\tilde{\bm{t}}))\hookrightarrow\mcE$ and $\tilde{D}_i \subset \mathbb{P}(\mcE)$ be the inverse image of $\tilde{t}_i$. Put $\mcL=\mcO_{\mathbb{P}(\mcE)}(\tilde{D}_0+\tilde{D}_1)$. Let 
	\[
	\varpi \colon\mathbb{P}(\mcE)\overset{\pi}{\longrightarrow} \pl\times T_3 \times \mcN\longrightarrow T_3\times \mcN
	\]
	be the projection and take a closed point $x\in T_3\times \mcN$. Since $\tilde{D}_0$ and $\tilde{D}_1$ are flat over $T_3\times \mcN$, $(\tilde{D}_0)_x$ and $(\tilde{D}_1)_x$ are effective Cartier divisors on $\mathbb{P}(\mcE_x)$, and so $\mcL_x\cong \mcO_{\mathbb{P}(\mcE_x)}((\tilde{D}_0)_x+ (\tilde{D}_1)_x)$. The section  $(\tilde{D}_0)_x\subset \mathbb{P}(\mcE_x)$ is a $(-1)$-curve by definition, so we get a morphism $f\colon \mathbb{P}(\mcE_x)\rightarrow \mathbb{P}^2$ by contracting $(\tilde{D}_0)_x$. By the projection formula $R^if_*\mcL_x\cong \mcO_{\mathbb{P}^2}(1)\otimes R^if_*\mcO_{\mathbb{P}(\mcE_x)}$, we have $H^i(\mathbb{P}(\mcE_x), \mcL_x)\cong H^i(\mathbb{P}^2, \mcO_{\mathbb{P}^2}(1))=0$ for any $i>0$, which leads to $\dim H^0(\mathbb{P}(\mcE_x), \mcL_x)=3$ by Riemann-Roch theorem. Hence $\varpi_*\mcL$ is a rank 3 locally free sheaf on $T_3\times \mcN$. Since $\mcL_x$ is generated by global section, the canonical homomorphism $\varpi^*\varpi_*\mcL \rightarrow \mcL$ is surjective, so we obtain a morphism $\rho\colon \mathbb{P}(\mcE)\rightarrow \mathbb{P}(\varpi_*\mcL)$ over $T_3\times \mcN$. Let $W$ be the scheme theoretic image of $\rho\colon \tilde{D}_0\rightarrow \mathbb{P}(\varpi_*\mcL)$. Since $\tilde{D}_0$ is proper over $T_3\times \mcN$, $W$ is a closed subvariety of  $\mathbb{P}(\varpi_*\mcL)$. $W_x$ consists of one point because $\deg_{(\tilde{D}_0)_x}\mcL|_{(\tilde{D}_0)_x}=(\tilde{D}_0)_x.((\tilde{D}_0)_x+(\tilde{D}_1)_x)=0$. We can see that $\mathbb{P}(\mcE)\setminus (\tilde{D}_0)\rightarrow \mathbb{P}(\varpi_*\mcL)\setminus W$ is an isomorphism by the proof of Theorem V.2.17. in \cite{Ha}, and $\mathbb{P}(\mcE)$ is isomorphic to the blow-up of $\mathbb{P}(\varpi_*\mcL)$ along $W$. By the residue map 
	\[
	\res_{\tilde{t}_i}\colon \Omega_{\pl\times T_3 \times \mcN/T_3\times \mcN}^1(D(\tilde{\bm{t}}))|_{\tilde{t}_i}\longrightarrow\mcO_{\tilde{t}_i},
	\]
	we obtain an isomorphism $\tilde{D}_i \overset{\sim}{\rightarrow} \pl\times T_3 \times \mcN$. For each $i=1,2,3$ and $j=0,1,2$, let $\tilde{b}_{i,j}$ be the section of $\tilde{D}_i$ over $T_3\times \mcN$ defined by 
	\[
	\{((\nu_{i,j}+\res_{t_i}(\tfrac{dz}{z-t_3}):1), (t_k)_k, (\nu_{m,n})_{m,n})\} \subset \pl\times T_3 \times \mcN.
	\]
	Let $\tilde{\mcB}_j$ denote the reduced induced structure on $\tilde{b}_{1,j}\cup \tilde{b}_{2,j}\cup\tilde{b}_{3,j}$ for $j=0,1,2$. Then we can naturally regard $\rho(\tilde{\mcB}_i)$ as a closed subvariety of $\mathbb{P}(\varpi_*\mcL)$, and it is isomorphic to $\tilde{\mcB}_i$. So we use the same character $\tilde{\mcB}_i$ to denote $\rho(\tilde{\mcB}_i)$ for simplicity of notation. Let $g_2\colon S_2\rightarrow \mathbb{P}(\varpi_*\mcL)$ be the blow-up along $\tilde{\mcB}_2$, $g_1\colon S_1 \rightarrow S_2$ be the blow-up along the strict transform of $\tilde{\mcB}_1$ and $g \colon S\rightarrow S_1$ be the blow-up along the strict transform of $\tilde{\mcB}_0$. Then for each closed point $(\bm{t},\boldsymbol{\nu})\in T_3\times \mcN$, the fiber $S_{(\bm{t},\boldsymbol{\nu})}$ is a surface obtained by blowing up three points on each of three lines meeting at a single point on $\mathbb{P}((\varpi_*\mcL)_{(\bm{t}, \boldsymbol{\nu})})\cong \mathbb{P}^2$. Let $Bl_W\colon Z\rightarrow S$ be the blow-up along $W$. $Z$ is also obtained by repeating the blow-up of $\mathbb{P}(\mcE)$ .
	
	Let $\widehat{M^{\boldsymbol{\alpha}}_3}(0,0,2)$ be the moduli space of pairs of an $\boldsymbol{\alpha}$-stable parabolic $\phi$-connection and a certain subbundle (see subsection \ref{appsec}), and $\PC \colon  \widehat{M^{\boldsymbol{\alpha}}_3}(0,0,2) \rightarrow \overline{M^{\boldsymbol{\alpha}}_3}(0,0,2)$ be the morphism defined by forgetting subbundles. 
	Our aim is to prove the following theorem.
	\begin{theorem}\label{maintheorem}
		Take $\boldsymbol{\alpha}=(\alpha_{i,j})_{1\leq i,j\leq 3}$ and $\gamma$ such that $0<\alpha_{i,j}\ll 1$ for any $1\leq i, j\leq 3$ and $\gamma \gg 0$.
		
		\begin{itemize}
			\item[(1)] The closed subscheme $Y_{\leq 1}$ defined by $\rank\phi \leq 1$ is reduced. The forgetful map $\PC \colon  \widehat{M^{\boldsymbol{\alpha}}_3}(0,0,2) \rightarrow \overline{M^{\boldsymbol{\alpha}}_3}(0,0,2)$ is the blow-up along $Y_{\leq1}$. 
			\item[(2)] There exists an isomorphism $\widehat{M^{\boldsymbol{\alpha}}_3}(0,0,2)\overset{\sim}{\longrightarrow} Z$ and $\overline{M^{\boldsymbol{\alpha}}_3}(0,0,2)\overset{\sim}{\longrightarrow} S$ over $T_3 \times\mcN$ such that the diagram
			\[
			\begin{tikzcd}
				\widehat{M^{\boldsymbol{\alpha}}_3}(0,0,2)\ar[r, "\sim"]\ar[d, "\PC"']&Z\ar[d, "Bl_W"]\\
				\overline{M^{\boldsymbol{\alpha}}_3}(0,0,2)\ar[r, "\sim"]&S
			\end{tikzcd}
			\]
			commutes.
			In particular, $\overline{M^{\boldsymbol{\alpha}}_3}(\bm{t},\boldsymbol{\nu})$ is isomorphic to an $A^{(1)*}_2$-surface for each $(\bm{t},\boldsymbol{\nu})\in T_3 \times \mcN$. 
			\item[(3)] Let $Y$ be the closed subscheme of $\overline{M^{\boldsymbol{\alpha}}_3}(0,0,2)$ defined by the conditions $\wedge^3\phi=0$. Then $Y$ is reduced and $M^{\boldsymbol{\alpha}}_3(0,0,2)\cong \overline{M^{\boldsymbol{\alpha}}_3}(0,0,2)\setminus Y$. Moreover, for each $(\bm{t},\boldsymbol{\nu})\in T_3 \times \mcN$, the fiber $Y_{(\bm{t},\boldsymbol{\nu})}$ is the anti-canonical divisor of   $\overline{M^{\boldsymbol{\alpha}}_3}(\bm{t},\boldsymbol{\nu})$.
		\end{itemize}
	\end{theorem}
	\begin{remark}\label{anynu}
		Theorem \ref{maintheorem} implies a description for all $\boldsymbol{\nu}$. Take $\nu_1, \nu_2, \nu_3 \in \mathbb{C}$ satisfying $\nu_1+\nu_2+\nu_3=2$. Put $L:=\Opl$ and
		\[
		\nabla_L:=d+\frac{1}{3}\left(\frac{\nu_1}{z-t_1}+\frac{\nu_2}{z-t_2}+\frac{\nu_3-2}{z-t_3}\right)dz.
		\]
		Then the morphism defined by
		\[
		\overline{M^{\boldsymbol{\alpha}}_3}(0,0,2)\longrightarrow \overline{M^{\boldsymbol{\alpha}}_3}(\nu_1,\nu_2,\nu_3), \; (E_1,E_2,\phi,\nabla,l^{(1)}_*,l^{(2)}_*)\longmapsto (E_1,E_2,\phi,\nabla,l^{(1)}_*,l^{(2)}_*)\otimes (L,\nabla_L)
		\]
		is an isomorphism. When $\deg E_1=\deg E_2\neq -2$, elementary transformations give isomorphisms of moduli spaces (see subsection \ref{eletr}).
	\end{remark}
	\begin{figure}
		\centering
		\includegraphics[width=0.5\linewidth]{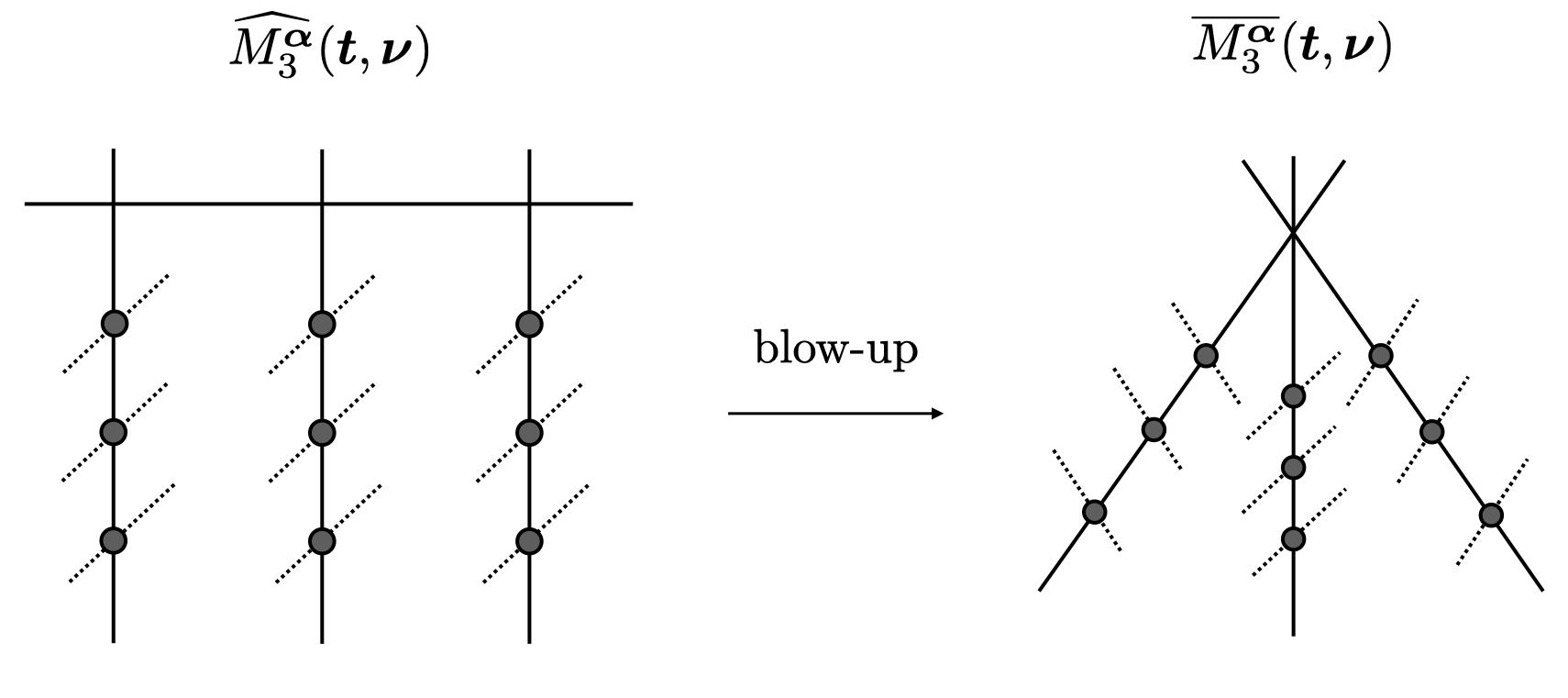}
	\end{figure}
	
	\subsection{The apparent map}\label{appsec}
	
	Take $\bm{t}=(t_i)_{1\leq i\leq 3}\in T_3, \boldsymbol{\nu}\in \mcN$ and put $D=t_1+t_2+t_3$. We assume that $0<\alpha_{i,j}\ll 1$ for any $1\leq i, j\leq 3$ and $\gamma \gg 0$. For simplicity of notation, we write $\overline{M}$ instead of $\overline{M^{\boldsymbol{\alpha}}_3}(\bm{t},\boldsymbol{\nu})$. First, we prove three lemmas and two propositions to define the apparent singularity and the apparent map.
	
	Let $(E_1,E_2,\phi,\nabla,l^{(1)}_*,l^{(2)}_*)$ be a $\boldsymbol{\nu}$-parabolic $\phi$-connection, and $F_1 \subset E_1$ and $F_2 \subset E_2$ be subbundles such that $(F_1,F_2)\neq (0,0)$. We put 
	\[
	\mu_{\boldsymbol{\alpha}}(F_1,F_2):=\frac{\deg F_1(-D)+\deg F_2(-D) - \gamma\, \rank F_2+\sum_{i=1}^{3}\sum_{j=1}^{3}\alpha_{i,j}(d^{(1)}_{i,j}(F_1)+d^{(2)}_{i,j}(F_2))}{\rank F_ 1+ \rank F_2 }, 
	\]
	where $d^{(k)}_{i,j}(F)=\dim (F|_{t_i}\cap l^{(k)}_{i,j-1})/(F|_{t_i}\cap l^{(k)}_{i,j})$.
	\begin{lemma}\label{stablem1}
		Let $(F_1, F_2) \subset (E_1, E_2)$ be a pair of subbundles with non-negative degree. If $(F_1,F_2)$ satisfies $\phi(F_1) \subset F_2$, $\nabla(F_1) \subset F_2 \otimes \Ompl(D)$ and $\rank F_1>\rank F_2$, then $(F_1, F_2)$ is an $\boldsymbol{\alpha}$-destabilizing pair of $(E_1,E_2,\phi,\nabla,l^{(1)}_*,l^{(2)}_*)$. 
	\end{lemma}
	\begin{proof}
		We have 
		\begin{align*}
			\qquad\mu_{\boldsymbol{\alpha}}(F_1,F_2)-\mu_{\boldsymbol{\alpha}}(E_1,E_2)
			&=\frac{\rank F_1-\rank F_2}{2(\rank E_1+\rank E_2)}\gamma+\frac{\deg F_1+\deg F_2}{\rank F_1+\rank F_2}-\frac{\deg E_1}{\rank E_1}\\
			&\qquad+\frac{\sum_{k=1}^{2}\sum_{i=1}^{3}\sum_{j=1}^{3}\alpha_{i,j}d^{(k)}_{i,j}(F_k)}{\rank F_1+\rank F_2}-\frac{\sum_{k=1}^{2}\sum_{i=1}^{3}\sum_{j=1}^{3}\alpha_{i,j}}{\rank E_1+\rank E_2}.
		\end{align*}
		Now $\gamma \gg 0$, so under the assumption, we obtain $\mu_{\boldsymbol{\alpha}}(F_1,F_2)-\mu_{\boldsymbol{\alpha}}(E_1,E_2)>0$.
	\end{proof}
	\begin{lemma}\label{stablem2}
		Let $(F_1,F_2) \subset (E_1,E_2)$ be a pair of non-zero subbundles of rank $r'<r$. If $(F_1,F_2)$ satisfy $\phi(F_1) \subset F_2$, $\nabla(F_1) \subset F_2 \otimes \Ompl(D(\bm{t}))$ and $\mu(F_1)+\mu(F_2)\geq -1$, then $(F_1, F_2)$ is an $ \boldsymbol{\alpha}$-destabilizing pair of $(E_1,E_2,\phi,\nabla,l^{(1)}_*, l^{(2)}_*)$. Here for nonzero vector bundle $F$, $\mu(F)=\deg F/\rank F$.
	\end{lemma}
	\begin{proof}
		We have
		\begin{align*}
			\mu_{\boldsymbol{\alpha}}(F_1,F_2)-\mu_{\boldsymbol{\alpha}}(E_1,E_2)=\frac{1}{2}\left\{\mu(F_1)+\mu(F_2)+\frac{4}{3}+\frac{\sum_{k=1}^{2}\sum_{i=1}^{3}\sum_{j=1}^{3}\alpha_{i,j}(3d^{(k)}_{i,j}(F_k)-r')}{3r'}\right\}.
		\end{align*}
		If $\mu(F_1)+\mu(F_2)\geq -1$, we obtain $\mu_{\boldsymbol{\alpha}}(F_1,F_2)-\mu_{\boldsymbol{\alpha}}(E_1,E_2)>0$.
	\end{proof}
	
	We give the proof of the following in Appendix B.
	\begin{proposition} \label{bdlform}
		For any $\boldsymbol{\alpha}$-stable $\boldsymbol{\nu}$-parabolic $\phi$-connection $(E_1,E_2,\phi,\nabla,l^{(1)}_*,l^{(2)}_*)$ of rank 3 and degree $-2$, we have
		\[
		E_1 \cong E_2 \cong \Opl \oplus \Opl(-1) \oplus \Opl(-1).
		\]
		So $E_1$ and $E_2$ have a unique trivial line subbundle.
	\end{proposition}
	
	\begin{lemma}\label{zerostab}
		Let $F$ be a unique trivial line subbundle of $E_1$. If $\phi|_{F}=0$, then $(E_1,E_2,\phi,\nabla,l^{(1)}_*,l^{(2)}_*)$ is $\boldsymbol{\alpha}$-unstable. In particular, if $\rank \phi=0$, i.e. $\phi=0$, then $(E_1,E_2,\phi,\nabla,l^{(1)}_*,l^{(2)}_*)$ is $\boldsymbol{\alpha}$-unstable.
	\end{lemma}
	
	\begin{proof}
		If $\phi|_{F}=0$, then the composite
		\[
		f \colon F \longrightarrow E_1 \overset{\nabla}{\longrightarrow} E_2 \otimes \Omega^1_{\pl}(D) 
		\]
		becomes a homomorphism. If $f=0$, then $(F, 0)$ breaks the stability. If $f \neq 0$, then $(F, (\Image f) \otimes (\Omega^1_{\pl}(D))^\vee)$ breaks the stability. 
	\end{proof}
	The following is the key proposition to define the apparent singularity.
	\begin{proposition}\label{appfil}
		Take $(E_1,E_2,\phi,\nabla,l^{(1)}_*,l^{(2)}_*)\in \overline{M}$. Then there exists a filtration $E_k=F^{(k)}_0\supsetneq F^{(k)}_1\supsetneq F^{(k)}_2\supsetneq F^{(k)}_3=0$ by subbundles for $k=1,2$ such that 
		\begin{equation}\label{filcon1}
			F^{(1)}_1\cong F^{(2)}_1\cong \Opl\oplus \Opl(-1), \; F^{(1)}_2\cong F^{(2)}_2\cong \Opl,
		\end{equation}
		and
		\begin{equation}\label{filcon2}
			\phi(F^{(1)}_i)\subset F^{(2)}_i, \;\nabla(F^{(1)}_{i+1})\subset F^{(2)}_{i}\otimes \Omega_{\pl}^1(D)
		\end{equation}
		for any $0\leq i\leq  2$.
		Subbundles $F^{(1)}_2, F^{(2)}_1, F^{(2)}_2$ satisfying the above conditions are uniquely determined. If $\rank \phi=2$ and $3$, then $F^{(1)}_1$ is also unique. If $\rank \phi=1$, then there is a one-to-one correspondence between the set of all such $F^{(1)}_1$ and $\mathbb{P}^1$.
	\end{proposition}
	\begin{proof}
		By Proposition \ref{bdlform}, $E_1$ and $E_2$ have a unique line subbundle which is isomorphic to the trivial line bundle. Let $F^{(k)}_2$ be the such line subbundle of $E_k$ for $k=1,2$. Then we have $\phi(F^{(1)}_2) \subset F^{(2)}_2$ by Proposition \ref{bdlform}, and so
		the composite
		\[
		f_2\colon \Opl \cong F^{(1)}_2 \hookrightarrow E_1 \overset{\nabla}{\longrightarrow} E_2 \otimes \Ompl (D) \rightarrow E_2/F^{(2)}_2 \otimes \Ompl(D) \cong \Opl\oplus \Opl
		\]
		becomes a homomorphism. If $f_2=0$, then $(F^{(1)}_2,F^{(2)}_2)$ breaks the stability of $(E_1,E_2,\phi,\nabla,l^{(1)}_*,l^{(2)}_*)$. So $f_2$ is not zero. Let 
		\[
		F^{(2)}_1=\ker (E_2\otimes \Ompl (D)\rightarrow(E_2/F^{(2)}_2\otimes \Ompl (D))/\Image f_2) \otimes  (\Omega_{\pl}^1(D))^\vee.
		\]
		Then we have $F^{(2)}_1\cong \Opl\oplus \Opl(-1)$ and $\nabla(F^{(1)}_2)\subset F^{(2)}_1\otimes \Omega_{\pl}^1(D)$. Let $K:=\ker (\phi \colon E_1 \rightarrow E_2/F^{(2)}_1)$. If $\rank \phi =2, 3$, then we have $K\cong \Opl\oplus \Opl(-1)$. Put $F^{(1)}_1=K$. We then obtain desired filtrations. The uniqueness of a filtration satisfying the above condition is clear. If $\rank \phi =1$, then $K=E_1$ by Lemma \ref{zerostab}. Take a subbundle $F^{(1)}_1\subset E_1$ which is isomorphic to $\Opl\oplus \Opl(-1)$. Then we have $\phi(F^{(1)}_1)\subset F^{(2)}_1$. We can see that there is a one-to-one correspondence between the set of such subbundles $F^{(1)}_1$  and  $\mathbb{P}\Hom(\Opl(-1), E_1/F^{(1)}_2)\cong \pl$.
	\end{proof}
	
	We define the apparent singularity. Let $E_k=F^{(k)}_0\supsetneq F^{(k)}_1\supsetneq F^{(k)}_2\supsetneq F^{(k)}_3=0$ be a filtration in Proposition \ref{appfil}. We define $f_1$ by
	\[
	f_1 \colon F^{(1)}_1 \hookrightarrow E_1 \overset{\nabla}{\longrightarrow} E_2 \otimes \Omega_{\pl}^1(D) \rightarrow E_2/F^{(2)}_1 \otimes \Omega_{\pl}^1(D). 
	\]
	Then $f_1$ becomes a homomorphism. If $f_1=0$, then we find $(F^{(1)}_1,  F^{(2)}_1)$ breaks the stability by Lemma \ref{stablem2}. So $f_1$ is not zero, and it implies that the induced homomorphism
	\[
	u \colon \Opl(-1)\cong F^{(1)}_1/F^{(1)}_2 \rightarrow E_1 \overset{\nabla}{\longrightarrow} E_2 \otimes \Omega_{\pl}^1(D) \rightarrow E_2/F^{(2)}_1 \otimes \Omega_{\pl}^1(D) \cong \Opl
	\]
	is also not zero because $\nabla(F^{(1)}_2) \subset F^{(2)}_1 \otimes \Omega_{\pl}^1(D)$. Since $u \in \Hom (\Opl(-1), \Opl) \cong H^0(\pl, \Opl(1))$, there exists a unique point $q \in \pl$ such that $u_q=0$. 
	
	\begin{definition}
		We call the zero $q$ of $u$ the apparent singularity of $(E_1,E_2,\phi,\nabla,l^{(1)}_*,l^{(2)}_*, F^{(1)}_1)$, and let $q$ denote $\App (E_1,E_2,\phi,\nabla,l^{(1)}_*,l^{(2)}_*, F^{(1)}_1)$. 
	\end{definition}
	Let $\widehat{M^{\boldsymbol{\alpha}}_3}(\bm{t},\boldsymbol{\nu})$ be the moduli space of pairs of a parabolic $\phi$-connections and a subbundle $F^{(1)}_1$, i.e.
	\[
	\widehat{M^{\boldsymbol{\alpha}}_3}(\bm{t},\boldsymbol{\nu}):=\{(E_1,E_2,\phi,\nabla, l^{(1)}_*,l^{(2)}_*,F^{(1)}_1)\}/\sim.
	\]
	We can construct $\widehat{M^{\boldsymbol{\alpha}}_3}(\bm{t},\boldsymbol{\nu})$ as follows.
	Let $(\tilde{E}_1,\tilde{E}_2,\tilde{\phi},\tilde{\nabla}, \tilde{l}^{(1)}_*, \tilde{l}^{(2)}_*)$ be a universal family over $\overline{M}\times \pl$ and $\tilde{F}^{(k)}_2 \subset \tilde{E}_k$ be a unique subbundle such that  $(\tilde{F}^{(k)}_2)_x \cong \Opl$ for each $x \in  \overline{M}$. Put
	\[
	\tilde{f}_2\colon \tilde{F}^{(1)}_2 \hookrightarrow \tilde{E}_1 \overset{\tilde{\nabla}}{\longrightarrow} \tilde{E}_2 \otimes \Ompl (D) \rightarrow \tilde{E}_2/\tilde{F}^{(2)}_2 \otimes \Ompl(D)
	\]
	and 
	\[
	\tilde{F}^{(2)}_1=\ker (\tilde{E}_2\otimes \Ompl (D)\rightarrow(\tilde{E}_2/\tilde{F}^{(2)}_2\otimes \Ompl (D))/\Image \tilde{f}_2) \otimes  \Omega_{\pl}^1(D)^\vee.
	\]
	Let $p_1\colon \overline{M}\times \pl\rightarrow \overline{M}$ and $p_2\colon \overline{M}\times \pl\rightarrow \pl $ be the projection and put $\mcG:=(p_1)_*\mcH om(p_2^*\Opl(-1), \tilde{E}_1/\tilde{F}^{(1)}_2)$. Then we have the natural isomorphism 
	\[
	\mcG|_x\cong \Hom(\Opl(-1),(\tilde{E}_1/\tilde{F}^{(1)}_2)_x)\cong \Hom(\Opl(-1),\Opl(-1)^{\oplus 2}).
	\]
	Let $\varpi \colon \mathbb{P}(\mcG)=\Proj \Sym (\mcG^\vee)\rightarrow \overline{M}$ be the projection and $[\sigma]$ be the homothety class of  a nonzero element $\sigma \in \mcG|_x$. Put
	\[
	\widehat{M^{\boldsymbol{\alpha}}_3}(\bm{t},\boldsymbol{\nu}):=\left\{[\sigma]\in \mathbb{P}(\mcG) \vb
	\begin{minipage}{8cm}
		the composite $\Opl(-1)\overset{\sigma}{\rightarrow} (\tilde{E}_1/\tilde{F}^{(1)}_2)_x\overset{\phi}{\rightarrow} (\tilde{E}_2/\tilde{F}^{(2)}_1)_x$ is zero, where $x=\varpi([\sigma])$
	\end{minipage} \right\}.
	\]
	Then $\widehat{M^{\boldsymbol{\alpha}}_3}(\bm{t},\boldsymbol{\nu})$ is a closed subscheme of $\mathbb{P}(\mcG)$ and desired one. We can see that the map 
	\[
	\widehat{M^{\boldsymbol{\alpha}}_3}(\bm{t},\boldsymbol{\nu}) \rightarrow \pl, \ (E_1,E_2,\phi,\nabla,l^{(1)}_*,l^{(2)}_*, F^{(1)}_1) \mapsto \App (E_1,E_2,\phi,\nabla,l^{(1)}_*,l^{(2)}_*, F^{(1)}_1)
	\]
	is a morphism. We call the morphism the apparent map and write it by $\App$.
	\subsection{Construction of the morphism $\varphi\colon \widehat{M^{\boldsymbol{\alpha}}_3}(\bm{t},\boldsymbol{\nu})\rightarrow \mathbb{P}(\Ompld \oplus \Opl)$}\label{morsec}
	
	For simplicity of notation, we write $\widehat{M}$ instead of $\widehat{M^{\boldsymbol{\alpha}}_3}(\bm{t},\boldsymbol{\nu})$. Take $(E_1,E_2,\phi,\nabla, l^{(1)}_*,l^{(2)}_*,F^{(1)}_1)\in \widehat{M}$ and put $q:=\App(E_1,E_2,\phi,\nabla, l^{(1)}_*,l^{(2)}_*,F^{(1)}_1)$. 
	Let $p_2 \colon E_2 \rightarrow E_2/F^{(2)}_1$ be the quotient and let us fix an isomorphism $E_2/F^{(2)}_1\cong \Opl(-t_3)$. We define a homomorphism $B \colon E_1 \rightarrow  E_2/F^{(2)}_1 \otimes \Omega_{\pl}^1(D)$ by $B(a):=(p_2 \otimes \id) \nabla(a)-d(p_2\phi(a))$ for $a \in E_1$, where $d$ is the canonical connection on $\Opl(-t_3)$. Since $\nabla(F^{(1)}_2)  \subset F^{(2)}_1 \otimes \Omega_{\pl}^1(D)$ and $u_q=0$, $B_q$ induces a homomorphism $h_1 \colon (E_1/F^{(1)}_1)|_q \rightarrow (E_2/F^{(2)}_1\otimes\Omega_{\pl}^1(D))|_q$ which makes the diagram 
	\begin{equation}\label{h1diag}
		\begin{tikzcd}
			0 \arrow[r]&F^{(1)}_1|_q\arrow[r]\arrow[rd,"0"]&E_1|_q \arrow[r]\arrow[d, "B_q"]&(E_1/F^{(1)}_1)|_q\arrow[r] \arrow[ld, "h_1"]&0 \\
			&&(E_2/F^{(2)}_1\otimes \Omega_{\pl}^1(D))|_q&&
		\end{tikzcd}
	\end{equation}
	commute. Let $h_2\colon (E_1/F^{(1)}_1)|_q \rightarrow (E_2/F^{(2)}_1)|_q$ be the homomorphism induced by $\phi$. Then $h_1,h_2$ determine a homomorphism 
	\[
	\iota \colon (E_1/F^{(1)}_1)|_q \longrightarrow ((E_2/F^{(2)}_1\otimes\Ompld)\oplus E_2/F^{(2)}_1)|_q, \quad a \mapsto (h_1(a),h_2(a)).
	\]
	\begin{lemma}\label{iota}
		$\iota$ is injective. 
	\end{lemma}
	\begin{proof}
		If $\rank \phi=3$, then $h_2$ is not zero. In fact, if $h_2=0$, then $\phi(E_1)\subset F^{(2)}_1$ since $\phi\colon \Opl(-1)\cong E_1/F^{(1)}_1 \rightarrow E_2/F^{(2)}_1\cong \Opl(-1)$ is zero. It is a contradiction. So $\iota$ is injective.
		
		Consider the case $\rank \phi=2$. Assume that $h_2=0$. We take a local basis $e^{(1)}_0, e^{(1)}_1, e^{(1)}_2$ (resp. $e^{(2)}_0, e^{(2)}_1, e^{(2)}_2$) of $E_1$ (resp. $E_2$) such that $e^{(1)}_2$ generates $F^{(1)}_2$ and $e^{(1)}_1, e^{(1)}_2$ generate $F^{(1)}_1$ (resp. $e^{(2)}_2$ generates $F^{(2)}_2$ and $e^{(2)}_1, e^{(2)}_2$ generate $F^{(2)}_1$). By taking bases well, $\phi$ and $\nabla$ are represented by matrices
		\[
		\phi(e^{(1)}_2, e^{(1)}_1, e^{(1)}_0)=(e^{(2)}_2, e^{(2)}_1, e^{(2)}_0)
		\begin{pmatrix}
			1&0&0 \\
			0&\phi_{22}&\phi_{23} \\
			0&0&0
		\end{pmatrix},
		\]
		\[
		\nabla(e^{(1)}_2, e^{(1)}_1, e^{(1)}_0)=(e^{(2)}_2, e^{(2)}_1, e^{(2)}_0)
		\begin{pmatrix}
			0& a_{12}(z)& a_{13}(z)\\
			1& a_{22}(z)& a_{23}(z)\\
			0& a_{32}(z)& a_{33}(z)
		\end{pmatrix}
		\frac{dz}{h(z)},
		\]
		where $z$ is an inhomogeneous coordinate on $\pl=\Spec \mathbb{C}[z]\cup \{\infty\}$ and $h(z)=(z-t_1)(z-t_2)(z-t_3)$ and $\phi_{22}, \phi_{23} \in \mathbb{C}$. Suppose that $\phi_{22}=0$. Then we may assume that $\phi_{23}=1$.
		For each $i=1,2,3$, $a_{32}(t_i)$ must be zero because the polynomial
		\[
		\left| \res_{t_i}\nabla-\lambda \phi\right|
		=\frac{1}{h'(t_i)}
		\begin{vmatrix}
			-h'(t_i)\lambda& a_{12}(t_i)& a_{13}(t_i)\\
			1& a_{22}(t_i)& a_{23}(t_i)-h'(t_i)\lambda\\
			0& a_{32}(t_i)& a_{33}(t_i)
		\end{vmatrix}
		\]
		in $\lambda$ is identically zero by Lemma \ref{detzero} and $h'(t_i)a_{32}(t_i)$ is the second order coefficient of $\left| \res_{t_i}\nabla-\lambda \phi\right|$. Here $'=d/dz$. Since $a_{32}(z)\in  H^0(\Opl(1))$, we obtain $a_{32}(z)=0$. Then $(F^{(1)}_1, F^{(2)}_1)$ breaks the stability of $(E_1,E_2,\phi, \nabla, l^{(1)}_*,l^{(2)}_*)$. Suppose that $\phi_{22}\neq 0$. Then we may assume that $\phi_{23}=0$. In the same way as the above, we can see that $a_{33}(z)=0$. So $(F^{(1)}_2\oplus E_1/F^{(1)}_1, F^{(2)}_1)$ breaks the stability of $(E_1,E_2,\phi, \nabla, l^{(1)}_*,l^{(2)}_*)$. Hence $h_2\neq 0$ and so $\iota$ is injective.
		
		Finally, we consider the case $\rank \phi =1$. Let $f\colon E_1/F^{(1)}_2\rightarrow E_2/F^{(2)}_1 \otimes \Ompld$ be the homomorphism induced by $\nabla$. Since $\phi(E_1)\subset F^{(2)}_2\subset F^{(2)}_1$, the map $f$ becomes a homomorphism. If $h_1=0$, then we have $f|_q=0$ by the diagram (\ref{h1diag}). If $f=0$, then $(E_1,F^{(2)}_1)$ breaks the stability, so $f\neq 0$. Since $E_1/F^{(1)}_2 \cong \Opl(-1)^{\oplus 2}$,  $E_2/F^{(2)}_1 \otimes \Ompld \cong \Opl$ and $f|_q=0$, we have $\ker f \cong \Opl(-1)$. Put $G:=\ker (E_1 \rightarrow (E_1/F^{(1)}_2)/\ker f)$. Then $G\cong \Opl\oplus \Opl(-1)$ and so $(G, F^{(2)}_1)$ breaks the stability. Hence $h_1\neq 0$ and so $\iota$ is injective.
	\end{proof}
	\begin{lemma}\label{detzero}
		For each $i$, the polynomial $\left| \res_{t_i}\nabla-\lambda \phi_{t_i}\right|$ in $\lambda$ has the form
		\[
		(\wedge^3\phi_{t_i})(\nu_{i,0}-\lambda)(\nu_{i,1}-\lambda)(\nu_{i,2}-\lambda).
		\]
	\end{lemma}
	\begin{proof}
		We take a basis $v^{(1)}_0, v^{(1)}_1, v^{(1)}_2$ (resp. $v^{(2)}_0, v^{(2)}_1, v^{(2)}_2$) of $E_1|_{t_i}$ (resp. $E_2|_{t_i}$) such that $v^{(1)}_2$ generates $l^{(1)}_2$ and $v^{(1)}_1, v^{(1)}_2$ generate $l^{(1)}_1$ (resp. $v^{(2)}_2$ generates $l^{(2)}_2$ and $v^{(2)}_1, v^{(2)}_2$ generate $l^{(2)}_1$). Then $\phi_{t_i}$ and $\res_{t_i}\nabla$ are represented by matrices
		\[
		\phi_{t_i}(v^{(1)}_2, v^{(1)}_1, v^{(1)}_0)=(v^{(2)}_2, v^{(2)}_1, v^{(2)}_0)
		\begin{pmatrix}
			\phi_{11}&\phi_{12}&\phi_{13}\\
			0&\phi_{22}&\phi_{23} \\
			0&0&\phi_{33}
		\end{pmatrix}, 
		\]
		\[
		\res_{t_i}\nabla(v^{(1)}_2, v^{(1)}_1, v^{(1)}_0)=(v^{(2)}_2, v^{(2)}_1, v^{(2)}_0)
		\begin{pmatrix}
			a_{11}& a_{12}& a_{13}\\
			0& a_{22}& a_{23}\\
			0& 0& a_{33}
		\end{pmatrix}
		\]
		because $\phi_{t_i}$ and $\res_{t_i}\nabla$ are parabolic. Since $(\res_{t_i}\nabla- \nu_{i,j}\phi_{t_i})(l^{(1)}_{i,j}) \subset l^{(2)}_{i,j+1}$ for $j=0, 1,2$, we have $a_{11}=\nu_{i,0}\phi_{11}$, $a_{22}=\nu_{i,1}\phi_{22}$ and $a_{33}=\nu_{i,2}\phi_{33}$. So we have
		\[
		\left| \res_{t_i}\nabla-\lambda \phi_{t_i}\right|=\phi_{11}\phi_{22}\phi_{33}(\nu_{i,0}-\lambda)(\nu_{i,1}-\lambda)(\nu_{i,2}-\lambda).
		\]
	\end{proof}
	By Lemma \ref{iota}, the map $\iota$ determines a point $\varphi(E_1,E_2,\phi,\nabla, l^{(1)}_*,l^{(2)}_*,F^{(1)}_1)$ of $\mathbb{P}(\Ompld \oplus \Opl)$. We can see that the map 
	\begin{equation}\label{varphidef}
		\varphi \colon \widehat{M}\longrightarrow \mathbb{P}(\Omega_{\pl}^1(D) \oplus \Opl)
	\end{equation}
	is a morphism. We prove later that $\varphi$ can be factored into a composition of blow-ups (see Proposition \ref{allneq}, Proposition \ref{twoeq}, Proposition \ref{alleq}).
	\subsection{Normal forms of $\boldsymbol{\alpha}$-stable parabolic $\phi$-connections}\label{nfsec}
	Take $(E_1,E_2,\phi,\nabla, l^{(1)}_*,l^{(2)}_*,F^{(1)}_1) \in \widehat{M}$.
	For $k=1,2$, let $E_k\supsetneq F^{(k)}_1\supsetneq F^{(k)}_2\supsetneq 0$ be a filtration in Proposition \ref{appfil}. We take a local basis $e^{(1)}_0, e^{(1)}_1, e^{(1)}_2$ (resp. $e^{(2)}_0, e^{(2)}_1, e^{(2)}_2$) of $E_1$ (resp. $E_2$) such that $e^{(1)}_2$ generates $F^{(1)}_2$ and $e^{(1)}_1, e^{(1)}_2$ generate $F^{(1)}_1$ (resp. $e^{(2)}_2$ generates $F^{(2)}_2$ and $e^{(2)}_1, e^{(2)}_2$ generate $F^{(2)}_1$). Let $z$ be a fixed inhomogeneous coordinate on $\pl=\Spec \mathbb{C}[z] \cup \{\infty\}$. Then $\phi$ and $\nabla$ are represented by matrices
	\[
	\phi(e^{(1)}_2, e^{(1)}_1, e^{(1)}_0)=(e^{(2)}_2, e^{(2)}_1, e^{(2)}_0)
	\begin{pmatrix}
		\phi_{11}&\phi_{12}&\phi_{13}\\
		0&\phi_{22}&\phi_{23}\\
		0&0&\phi_{33}
	\end{pmatrix}, 
	\]
	\[
	\small
	\nabla(e^{(1)}_2, e^{(1)}_1, e^{(1)}_0)=(e^{(2)}_2, e^{(2)}_1, e^{(2)}_0)
	\begin{pmatrix}
		a_{11}(z)& a_{12}(z)& a_{13}(z)\\
		a_{21}& \phi_{22}(z-t_1)(z-t_2)+a_{22}(z)& \phi_{23}(z-t_1)(z-t_2)+a_{23}(z)\\
		0& a_{32}(z)& \phi_{33}(z-t_1)(z-t_2)+a_{33}(z)
	\end{pmatrix}
	\frac{dz}{h(z)},
	\]
	where $\phi_{11},\phi_{22},\phi_{23},\phi_{33} \in H^0(\Opl), \phi_{12}, \phi_{13}\in H^0(\Opl(1))$, $a_{11},a_{22},a_{23},a_{32},a_{33} \in H^0(\Opl(1))$,$a_{21}\in H^0(\Opl)$, and $h(z)=(z-t_1)(z-t_2)(z-t_3)$. By taking $e^{(1)}_0, e^{(1)}_1, e^{(2)}_0, e^{(2)}_1$ well, we may assume that $\phi_{12}=\phi_{13}=0, a_{11}(z)=0$ and $a_{21}=1$. Then we have $a_{12}, a_{13}\in H^0(\Opl(2))$. Let $q$ be the apparent singular point of $(E_1,E_2,\phi,\nabla, l^{(1)}_*,l^{(2)}_*,F^{(1)}_1)$.
	\begin{lemma}\label{rk3formlem}
		Assume that $\rank \phi=3$, i.e. $\wedge^3\phi \neq 0$. Then $\phi$ and $\nabla$ have the forms
		\begin{equation}\label{rk3form}
			\phi=
			\begin{pmatrix}
				1& 0& 0\\
				0& 1& 0\\
				0& 0& 1
			\end{pmatrix}, 
			\nabla=d+
			\begin{pmatrix}
				0& a_{12}(z)& a_{13}(z)\\
				1& (z-t_1)(z-t_2)-p& 0\\
				0& z-q& (z-t_1)(z-t_2)+p
			\end{pmatrix}
			\frac{dz}{h(z)},
		\end{equation}
		respectively, where $p\in \mathbb{C}$ and $a_{12}(z),a_{13}(z)$ are quadratic polynomials in $z$ satisfying
		\begin{equation}\label{rk3coe1}
			a_{12}(t_i)=-h'(t_i)^2(\nu_{i,0}\nu_{i,1}+\nu_{i,1}\nu_{i,2}+\nu_{i,2}\nu_{i,0}-(\res_{t_i}(\tfrac{dz}{z-t_3}))^2)-p^2,
		\end{equation}
		\begin{equation}\label{rk3coe0}
			(t_i-q)a_{13}(t_i)=\prod_{j=0}^2(h'(t_i)(\nu_{i,j}-\res_{t_i}(\tfrac{dz}{z-t_3}))-p)
		\end{equation}
		for any $i=1,2,3$. Here $'=d/dz$.
	\end{lemma}
	\begin{proof}
		Applying $\phi^{-1}$ to $E_2$, we may assume that $\phi =\id$. Put
		\[
		C=
		\begin{pmatrix}
			1& 0& c_{13}(z)\\
			0& 1& c_{23}\\
			0& 0& 1
		\end{pmatrix},
		\]
		where $c_{13}(z) \in H^0(\Opl(1))$ and $c_{23} \in H^0(\Opl)$. Then we have
		\begin{align*}
			&C\circ \nabla \circ C^{-1}\\
			=&d+
			\begin{pmatrix}
				0& a_{12}(z)+c_{13}(z-q)& a_{13}(z)-c_{23}a_{12}(z)+c_{13}(z)a_{33}(z)-c_{13}(z)c_{23}(z-q)-h(z)c'_{13}(z)\\
				1& a_{22}(z)+c_{23}(z-q)& a_{23}(z)-c_{23}a_{22}(z)-c_{13}(z)+c_{23}a_{33}(z)-c_{23}^2(z-q)\\
				0& z-q& a_{33}(z)-c_{23}(z-q)
			\end{pmatrix}
			\frac{dz}{h(z)}.
		\end{align*}
		So we may assume that $a_{23}(z)=0$ and $a_{33}(z)$ changes into the form $(z-t_1)(z-t_2)+p$. Since $\res_{t_i}\tr \nabla =2\res_{t_i}(\frac{dz}{z-t_3})$, we have $a_{22}(z)=(z-t_1)(z-t_2)-p$. So we obtain the desired form
		\[
		\nabla
		=
		d+
		\begin{pmatrix}
			0& a_{12}(z)& a_{13}(z)\\
			1& (z-t_1)(z-t_2)-p& 0\\
			0& z-q& (z-t_1)(z-t_2)+p
		\end{pmatrix}\frac{dz}{h(z)}.
		\]
		By Lemma \ref{detzero},  we can see that $a_{12}(z)$ and $a_{13}(z)$ satisfy the conditions (\ref{rk3coe1}) and (\ref{rk3coe0}) for each $i=1,2,3$.
	\end{proof}
	\begin{remark}
		The polynomial $a_{12}(z)$ is uniquely determined by $p$. When $q\neq t_1, t_2, t_3$, $a_{13}(z)$ is also uniquely determined by $q$ and $p$. When $q=t_i$, $p$ is equal to one of $h'(t_i)(\nu_{i,0}-\res_{t_i}(\tfrac{dz}{z-t_3})), h'(t_i)(\nu_{i,1}-\res_{t_i}(\tfrac{dz}{z-t_3})), h'(t_i)(\nu_{i,2}-\res_{t_i}(\tfrac{dz}{z-t_3}))$ and $a_{13}(t_i)$ takes any complex number. When $p=h'(t_i)(\nu_{i,j}-\res_{t_i}(\tfrac{dz}{z-t_3}))$, we have $(\res_{t_i}\oplus \id)(\varphi(E,\nabla,l_*))=(\nu_{i,j}-\res_{t_i}(\tfrac{dz}{z-t_3}):1)$, where $\res_{t_i}\oplus \id\colon \mathbb{P}(\Omega_{\pl}^1(D(\bm{t}))\oplus \mcO_{\pl})|_{t_i}\rightarrow \mathbb{P}(\mcO_{\pl}|_{t_i}\oplus \mcO_{\pl}|_{t_i})$ is a natural isomorphism. The choices of $a_{13}(t_i)$ give exceptional curves of the first kind on the moduli space of parabolic connections (see Proposition 2.20, 2.21, and  2.22). 
	\end{remark}
	\begin{lemma}\label{rk2formlem}
		Assume that $\rank \phi=2$. Then $\phi$ and $\nabla$ have the forms
		\begin{equation}\label{rk2form}
			\phi=
			\begin{pmatrix}
				1&0&0 \\
				0&0&0 \\
				0&0&1
			\end{pmatrix}, \ 
			\nabla=\phi\otimes d+
			\begin{pmatrix}
				0&0&\prod_{j\neq i}(z-t_j) \\
				1&0&0\\
				0&z-t_i&(z-t_1)(z-t_2)+p
			\end{pmatrix}
			\frac{dz}{h(z)},
		\end{equation}
		respectively.
	\end{lemma}
	\begin{proof}
		By the proof of Lemma \ref{iota}, we have $\phi_{33}\neq 0$. So we may assume that $\phi$ is of the form (\ref{rk2form}).
		Applying an automorphism of $E_1,E_2$ given by the form
		\[
		\begin{pmatrix}
			1&0&-a_{23}(z) \\
			0&1&0 \\
			0&0&1
		\end{pmatrix},
		\]
		$\nabla$ changes into the form
		\[
		\begin{pmatrix}
			0& a_{12}(z)+a_{23}(z)a_{32}(z)& a_{13}(z)+a_{23}(z)a_{33}(z)-h(z)a'_{23}(z)\\
			1& a_{22}(z)& 0\\
			0& a_{32}(z)& a_{33}(z)
		\end{pmatrix}
		\frac{dz}{h(z)}.
		\]
		So we may assume without loss of generality that $a_{23}(z)=0$.
		Using an argument of the proof of Lemma \ref{iota}, we obtain $a_{12}(z)=a_{22}(z)=0$ and $a_{32}(t_i)a_{13}(t_i)=0$ for $i=1,2,3$. If $a_{32}(z)$ is identically zero, then $(F^{(1)}_1,  F^{(2)}_1)$ breaks the stability. If $a_{13}(z)$ is identically zero, then $(E_1/F^{(1)}_2, E_2/F^{(2)}_1)$  breaks the stability. So there exists unique $i \in \{1,2,3\}$ such that $a_{32}(t_i)=0$, which implies $a_{13}(t_j)=0$ for $j\neq i$. Applying suitable automorphisms, we obtain the desired form (\ref{rk2form}).
	\end{proof}
	\begin{lemma}\label{rk1formlem}
		Assume that $\rank \phi=1$. Then $\phi$ and $\nabla$ have the forms
		\begin{equation}\label{rk1form}
			\phi=
			\begin{pmatrix}
				1&0&0 \\
				0&0&0 \\
				0&0&0
			\end{pmatrix}, \ 
			\nabla=\phi\otimes d+
			\begin{pmatrix}
				0&\prod_{j\neq i}(z-t_j)&0 \\
				1&0&0\\
				0&z-q&z-t_i
			\end{pmatrix}
			\frac{dz}{h(z)},
		\end{equation}
		respectively, where $t_i\neq q$.
	\end{lemma}
	\begin{proof}
		By Lemma \ref{zerostab} and the assumption, $\phi$ and $\nabla$ have the forms
		\[
		\phi=
		\begin{pmatrix}
			1&0&0 \\
			0&0&0 \\
			0&0&0
		\end{pmatrix}, \quad 
		\nabla=\phi\otimes d+
		\begin{pmatrix}
			0& a_{12}(z)& a_{13}(z)\\
			1& a_{22}(z)& a_{23}(z)\\
			0& z-q& a_{33}(z)
		\end{pmatrix}
		\frac{dz}{h(z)}, 
		\]
		where $a_{12},a_{13}\in H^0(\Opl(2))$ and $a_{22},a_{23},a_{33} \in H^0(\Opl(1))$. If $a_{33}(q)=0$, then we may assume that $a_{33}(z)=0$ by applying an automorphism of $E_1$, which implies that $(F^{(1)}_2\oplus E_1/F^{(1)}_1, F^{(2)}_1)$ breaks the stability of $(E_1,E_2,\phi,\nabla, l^{(1)}_*,l^{(2)}_*)$.  Hence we have $a_{33}(q)\neq 0$. Let us fix $i\in \{1,2,3\}$ satisfying $t_i\neq q$. 
		Applying an automorphism of $E_1$ given by the form
		\[
		\begin{pmatrix}
			1& 0& 0\\
			0& 1& 1-a_{33}(q)^{-1}a'_{33}(q)(q-t_i)\\
			0& 0& a_{33}(q)^{-1}(q-t_i)
		\end{pmatrix},
		\]
		the $\phi$-connection $\nabla$ changes into the form
		\[
		\phi\otimes d+
		\begin{pmatrix}
			0& a_{12}(z)& a_{13}(z)\\
			1& a_{22}(z)& a_{23}(z)\\
			0& z-q& z-t_i
		\end{pmatrix}.
		\]
		We consider the polynomial
		\begin{equation}\label{rk1eigen}
			\left| \res_{t_j}\nabla-\lambda \phi_{t_j}\right|
			=\frac{1}{h'(t_j)^3}
			\begin{vmatrix}
				-h'(t_j)\lambda& a_{12}(t_j)& a_{13}(t_j)\\
				1& a_{22}(t_j)& a_{23}(t_j)\\
				0& t_j-q& t_j-t_i
			\end{vmatrix}
		\end{equation}
		in $\lambda$. By Lemma \ref{detzero}, the polynomial (\ref{rk1eigen}) is identically zero, that is, we have 
		\begin{equation}\label{rk1rel1}
			(t_j-t_i)a_{22}(t_j)-(t_j-q)a_{23}(t_j)=0,
		\end{equation}
		\begin{equation}\label{rk1rel2}
			(t_j-t_i)a_{12}(t_j)-(t_j-q)a_{13}(t_j)=0
		\end{equation}
		for any $j$. By (\ref{rk1rel1}) and (\ref{rk1rel2}),  we have $a_{13}(t_i)=a_{23}(t_i)=0$.
		Applying a suitable automorphism of $E_2$, we may assume without loss of generality that $a_{13}(z)=a_{23}(z)=0$. Then we have $a_{22}(t_j)=0$ for $j\neq i$ by (\ref{rk1rel1}), and it implies that $a_{22}(z)=0$. By (\ref{rk1rel2}), we have $a_{12}(t_i)=0$ for $j\neq i$. If $a_{12}(z)$ is identically zero, then $(E_1/F^{(1)}_2, E_2/F^{(2)}_1)$ breaks the stability. So $\phi$ and $\nabla$ have the forms (\ref{rk1form}).
	\end{proof}
	\begin{remark}\label{rk1rem}
		Let $(E_1,E_2,\phi,\nabla, l^{(1)}_*,l^{(2)}_*,F^{(1)}_1)$, $(E'_1,E'_2,\phi',\nabla', l'^{(1)}_*,l'^{(2)}_*,F'^{(1)}_1)$ be $\boldsymbol{\nu}$-parabolic $\phi$-connections such that $\rank \phi=\rank \phi'=1$. Then $(E_1,E_2,\phi,\nabla, l^{(1)}_*,l^{(2)}_*)$ and $(E'_1,E'_2,\phi',\nabla', l'^{(1)}_*,l'^{(2)}_*)$ are isomorphic to each other. In other words, the locus on $\overline{M}$ defined by $\rank \phi =1$ consists of one point. In fact, applying automorphisms of $E_1,E_2$, $\phi$ and $\nabla$ change into the form
		\begin{equation*}
			\begin{pmatrix}
				1&0&0 \\
				0&0&0 \\
				0&0&0
			\end{pmatrix}, \;
			\phi\otimes d+
			\begin{pmatrix}
				0&(z-t_2)(z-t_3)&0 \\
				1&0&0\\
				0&z-t_2&z-t_1
			\end{pmatrix}
			\frac{dz}{h(z)}.
		\end{equation*}
		By the proof of Proposition \ref{injrk12}, it follows that  parabolic structures $l^{(1)}_{i,*}$ and $l^{(2)}_{i,*}$ satisfying the conditions $\phi_{t_i}(l^{(1)}_{i,j}) \subset l^{(2)}_{i,j}$ and $(\res_{t_i}\nabla-\nu_{i,j}\phi_{t_i})(l^{(1)}_{i,j})\subset l^{(2)}_{i,j+1}$ are uniquely determined. 
	\end{remark}
	\begin{proposition}\label{injrk12}
		Let $Y$ be the closed subscheme of $\widehat{M}$ defined by the condition $\wedge^3\phi=0$. Then the restriction morphism $\varphi \colon Y\longrightarrow \mathbb{P}(\Omega_{\pl}^1(D) \oplus \Opl)$ is injective.
	\end{proposition}
	\begin{proof}
		Take a point  $x=(E_1,E_2,\phi,\nabla, l^{(1)}_*,l^{(2)}_*,F^{(1)}_1) \in Y_{(\bm{t},\boldsymbol{\nu})}$. Then $\rank \phi$ must be one or two by Lemma \ref{zerostab}. Let $D_0$ be the section of $\mathbb{P}(\Omega_{\pl}^1(D)\oplus \Opl)$ over $\pl$ defined by the injection $\Omega_{\pl}^1(D) \hookrightarrow \Ompld \oplus \Opl$, that is, $D_0$ is the section defined by $h_2=0$, where $h_2$ is defined in subsection \ref{morsec}. Let $D_i\subset \mathbb{P}(\Omega_{\pl}^1(D) \oplus \Opl)$ be the fiber over $t_i \in \pl$. By the proof of Lemma \ref{rk2formlem} and Lemma \ref{rk1formlem}, $\varphi(x) \in \cup_{i=1}^3D_i\setminus D_0$ if and only if $\rank \phi =2$, and  $\varphi(x) \in D_0$ if and only if $\rank \phi =1$.
		
		First, we consider the case of $\rank \phi=2$. By Lemma \ref{rk2formlem}, a pair $(\phi, \nabla)$ is uniquely determined up to isomorphism by $\varphi(x)$.  By Proposition \ref{appfil}, $F^{(1)}_1$ is also uniquely determined by $(E_1,E_2,\phi,\nabla)$. Moreover, we can check that parabolic structures $l^{(1)}_*$ and $l^{(2)}_*$ are uniquely determined by $(E_1,E_2,\phi,\nabla)$. 
		
		Next we consider the case of $\rank \phi =1$. By Proposition \ref{appfil} and Lemma \ref{rk1formlem}, a triple $(\phi, \nabla, F^{(1)}_1)$ is uniquely determined up to isomorphism by the apparent singularity $q$. We can see that parabolic structures $l^{(1)}_*$ and $l^{(2)}_*$ are determined by $\phi$ and $\nabla$. So $\varphi|_{Y_{(\bm{t},\boldsymbol{\nu})}}$ is injective.
	\end{proof}
	
	\subsection{Proof of Theorem \ref{maintheorem}}\label{pfsec}
	To prove Theorem \ref{maintheorem}, we study  $\widehat{M}$ and $\overline{M}$ in more detail. Let $D_0$ be the section of $\mathbb{P}(\Omega^1_{\pl}(D)\oplus \mcO_{\pl})$ over $\mathbb{P}^1$ defined by the injection $\Omega_{\pl}^1(D)\hookrightarrow \Omega^1_{\pl}(D)\oplus \mcO_{\pl}$, and $D_i$ be the fiber of $\mathbb{P}(\Omega^1_{\pl}(D)\oplus \mcO_{\pl})$ over $t_i \in \mathbb{P}^1$. Let $b_{i,j}$ be the point of $D_i$ corresponding to $\nu_{i,j}$, and put $B:=\{b_{i,j}\mid 1\leq i\leq 3, 0\leq j\leq 2\}$. We show that $\widehat{M}$ is obtained by blowing up $\mathbb{P}(\Omega^1_{\pl}(D)\oplus \mcO_{\pl})$ at any point in $B$.
	\begin{proposition}\label{isominusblow}
		The restriction morphism
		\begin{equation}\label{restphi}
			\varphi \colon \widehat{M} \setminus \varphi^{-1}(B) \longrightarrow \mathbb{P}(\Omega^1_{\pl}(D)\oplus \mcO_{\pl}) \setminus B
		\end{equation}
		is an isomorphism.
	\end{proposition}
	\begin{proof}
		Let $z$ be a fixed inhomogeneous coordinate on $\pl=\Spec \mathbb{C}[z] \cup \{\infty\}$. Let $D_\infty$ be the fiber of $\mathbb{P}(\Omega^1_{\pl}(D)\oplus \mcO_{\pl})$ over $\infty \in \mathbb{P}^1$. Put $\mcD= \bigcup_{i=0}^3 D_i\cup D_\infty$. Then the morphism
		\[
		(\pl\setminus \{t_1,t_2,t_3, \infty\})\times \mathbb{C}\longrightarrow \mathbb{P}(\Omega^1_{\pl}(D)\oplus \mcO_{\pl})\setminus \mcD; \quad (q,p)\longmapsto \mathbb{C}(p\tfrac{dz}{h(z)}, 1) \subset  \Omega^1_{\pl}(D)|_q\oplus \mcO_{\pl}|_q
		\]
		becomes an isomorphism. By this isomorphism, we regard $(q,p)$ as a coordinate on $\mathbb{P}(\Omega^1_{\pl}(D)\oplus \mcO_{\pl})\setminus \mcD$. We define a family of $\boldsymbol{\nu}$-parabolic connections $(E,\nabla,l_*)$ on $\mathbb{P}(\Omega^1_{\pl}(D)\oplus \mcO_{\pl})\setminus \mcD\times \pl$ as follows. Let $E=p_2^*(\mcO_{\pl}\oplus \mcO_{\pl}(-1)\oplus \mcO_{\pl}(-1))$, where $p_2\colon \mathbb{P}(\Omega^1_{\pl}(D)\oplus \mcO_{\pl})\setminus \mcD\times \pl\rightarrow \pl$ be the projection. We define a relative logarithmic connection $\nabla \colon E\rightarrow E\otimes p_2^*\Omega^1_{\pl}(D)$ by
		\[
		\nabla
		:=
		d
		+
		\begin{pmatrix}
			0&a_{12}(p;z)&a_{13}(q,p;z)\\
			1&(z-t_1)(z-t_2)-p&0\\
			0&z-q&(z-t_1)(z-t_2)+p
		\end{pmatrix}
		\frac{dz}{h(z)},
		\]
		where $a_{12}(p;z), a_{13}(q,p;z)$ are the quadratic polynomials in $z$ satisfying
		\[
		a_{12}(p;t_i)=(t_i-t_1)^2(t_i-t_2)^2-p^2-h'(t_i)^{2}(\nu_{i,0}\nu_{i,1}+\nu_{i,1}\nu_{i,2}+\nu_{i,2}\nu_{i,0})
		\]
		\[
		(t_i-q)a_{13}(q,p;t_i)=\prod_{j=0}^2(h'(t_i)(\nu_{i,j}-(\res_{t_i}(\tfrac{dz}{z-t_3})))-p)
		\]
		
		for any $i=1,2,3$. Let $E|_{t_i}\supsetneq l_{i,1}\supsetneq l_{i,2} \supsetneq 0$ be a filtration by subbundles such that $(\res_{t_i}\nabla-\nu_{i,j}\id)(l_{i,j})\subset l_{i,j+1}$ for any $j=0,1,2$. Then we have
		\begin{equation}\label{rk3flag2}
			l_{i,2}=\mathbb{C}
			\begin{pmatrix}
				(p+h'(t_i)(\nu_{i,2}-\res_{t_i}(\tfrac{dz}{z-t_3})))(h'(t_i)(\nu_{i,2}-\res_{t_i}(\tfrac{dz}{z-t_3}))-p)\\
				(h'(t_i)(\nu_{i,2}-\res_{t_i}(\tfrac{dz}{z-t_3}))-p)\\
				t_i-q
			\end{pmatrix},
		\end{equation}
		\begin{equation}\label{rk3flag1}
			l_{i,1}=
			\mathbb{C}
			\begin{pmatrix}
				(p+h'(t_i)(\nu_{i,2}-\res_{t_i}(\tfrac{dz}{z-t_3})))(h'(t_i)(\nu_{i,2}-\res_{t_i}(\tfrac{dz}{z-t_3}))-p)\\
				(h'(t_i)(\nu_{i,2}-\res_{t_i}(\tfrac{dz}{z-t_3}))-p)\\
				t_i-q
			\end{pmatrix}
			+\mathbb{C}
			\begin{pmatrix}
				-h'(t_i)\nu_{i,0}\\
				1\\
				0
			\end{pmatrix}
		\end{equation}
		For any $(q,p) \in \mathbb{P}(\Omega^1_{\pl}(D)\oplus \mcO_{\pl})\setminus \mcD$, the corresponding $\boldsymbol{\nu}$-parabolic connection $(E_{(q,p)},\nabla_{(q,p)},(l_*)_{(q,p)})$ is $\boldsymbol{\alpha}$-stable. So we obtain a morphism 
		\[
		\mathbb{P}(\Omega^1_{\pl}(D)\oplus \mcO_{\pl})\setminus \mcD \longrightarrow \widehat{M} \setminus \varphi^{-1}(\mcD), 
		\]
		which is just the inverse of the morphism
		\[
		\varphi \colon \widehat{M} \setminus \varphi^{-1}(\mcD) \longrightarrow \mathbb{P}(\Omega^1_{\pl}(D)\oplus \mcO_{\pl}) \setminus \mcD.
		\]
		Hence the morphism (\ref{restphi}) is a birational morphism. By Proposition \ref{injrk12} and Zariski's main theorem, the morphism (\ref{restphi}) is an isomorphism. 
	\end{proof}
	\begin{proposition}
		$\overline{M}$ and $\widehat{M}$ is a smooth variety. 
	\end{proposition}
	\begin{proof}
		We give a proof of the smoothness of $\overline{M}$ in Appendix C.
		By Remark \ref{rk1rem}, the locus on $\overline{M}$ defined by $\rank \phi =1$ consists of one point $p_0$. Let $\PC \colon  \widehat{M}\rightarrow \overline{M}$ be the forgetful map. Then, by Proposition \ref{appfil}, the restriction map
		\[
		\PC \colon  \widehat{M}\setminus \PC^{-1}(p_0)\longrightarrow \overline{M}\setminus\{p_0\}
		\]
		becomes an isomorphism. So it is sufficient to prove that $\widehat{M}$ is smooth at any point in $\PC^{-1}(p_0)$, and it follows from Proposition \ref{isominusblow}.
	\end{proof}
	We investigate the fiber of $\varphi$ over $B$.
	\begin{proposition}\label{allneq}
		If $\nu_{i,0}\neq \nu_{i,1}\neq \nu_{i,2}\neq \nu_{i,0}$, then $\varphi
		^{-1}(b_{i,j})\cong \pl$ for any $j=0,1,2$ and these are $(-1)$-curves.
	\end{proposition}
	\begin{proof}
		Let $E_1=E_2=\Opl\oplus\Opl(-1)\oplus\Opl(-1)$, $p=h'(t_i)(\nu_{i,j}-\res_{t_i}(\frac{dz}{z-t_3}))$ and $h(z)=(z-t_1)(z-t_2)(z-t_3)$. 
		Let $a(z)$ be the quadratic polynomial satisfying 
		\[
		a(t_m)=(t_m-t_1)^2(t_m-t_2)^2-p^2- h'(t_m)^{2}(\nu_{m,0}\nu_{m,1}+\nu_{m,1}\nu_{m,2}+\nu_{m,2}\nu_{m,0})
		\]
		for $m=1,2,3$. Let $b(z)$ be the quadratic polynomial satisfying $b(t_i)=0$ and
		\[
		(t_m-t_i)b(t_m)
		=\prod_{j=0}^2 (h'(t_m)(\nu_{m,j}-\res_{t_m}(\tfrac{dz}{z-t_3}))-p)
		\]
		for $m\neq i$.  
		Put 
		\begin{equation}\label{allneqform}
			\phi_\mu=
			\begin{pmatrix}
				1& 0& 0\\
				0& \mu& 0\\
				0& 0& 1
			\end{pmatrix}
			,\;\nabla_{(\mu,\eta)}=\phi_\mu\otimes d+
			\begin{pmatrix}
				0&\mu a(z)&\mu b(z)+\eta \prod_{m\neq i}(z-t_m)\\
				1&\mu(z-t_1)(z-t_2)-\mu p&0\\
				0&z-t_i&(z-t_1)(z-t_2)+p
			\end{pmatrix}\frac{dz}{h(z)},
		\end{equation}
		where $\mu, \eta \in \mathbb{C}$. When $\mu=\eta=0$, the $\phi_\mu$-connection $(E_1,E_2, \phi_\mu, \nabla_{(\mu,\eta)})$ becomes $\boldsymbol{\alpha}$-unstable for any parabolic structures. Assume that $(\mu,\eta)\neq (0,0)$. Then parabolic structures $l^{(1)}_{i,*}$ and $l^{(2)}_{i,*}$ of $E_1$ and $E_2$, respectively, satisfying the conditions $(\phi_\mu)_{t_i}(l^{(1)}_{i,j}) \subset l^{(2)}_{i,j}$ and $(\res_{t_i}(\nabla_{(\mu,\eta)})-\nu_{i,j}(\phi_\mu)_{t_i})(l^{(1)}_{i,j})\subset l^{(2)}_{i,j+1}$ are uniquely determined.  We can see that $(E_1,E_2,\phi_\mu,\nabla_{(\mu,\eta)}, l^{(1)}_*,l^{(2)}_*)$ is $\boldsymbol{\alpha}$-stable if and only if $(\mu,\eta)\neq (0,0)$. We can also see that $(E_1, E_2,\phi_{\mu_1},\nabla_{(\mu_1,\eta_1)})$ and $(E_1, E_2,\phi_{\mu_2},\nabla_{(\mu_2,\eta_2)} )$ are isomorphic to each other if and only if there exists $c \in \mathbb{C}^\times $ such that $(\mu_1,\eta_1)=c(\mu_2,\eta_2)$. So we obtain the morphism
		\[
		\mathbb{P}^1\longrightarrow \varphi^{-1}(b_{i,j}); \; (\mu:\eta)\longmapsto (E_1,E_2,\phi_{\mu_1},\nabla_{(\mu_1,\eta_1)},l^{(1)}_*, l^{(2)}_*), 
		\]
		which is an isomorphism by Lemma \ref{rk3formlem} and Lemma \ref{rk2formlem}. Since $\widehat{M}$ and $\mathbb{P}(\Omega_{\pl}^1(D)\oplus\Opl)$ are smooth, $\varphi^{-1}(b_{i,j})$ is a $(-1)$-curve.
	\end{proof}
	Let $N_3(\bm{t},\boldsymbol{\nu})$ be the moduli space of rank 3 stable $\boldsymbol{\nu}$-logarithmic connections over $(\mathbb{P}, \bm{t})$. A connection $(E,\nabla)$ is said to be stable if for any nonzero subbundle $F\subsetneq E$ preserved by $\nabla$, the inequality 
	\[
	\frac{\deg F}{\rank F}<\frac{\deg E}{\rank E}
	\]
	holds. Under the assumption in this section, a $\boldsymbol{\nu}$-parabolic connection $(E,\nabla,l_*)$ is $\boldsymbol{\alpha}$-stable if and only if $(E,\nabla)$ is stable. So we have the surjective morphism $M^{\boldsymbol{\alpha}}_3(\bm{t},\boldsymbol{\nu})\rightarrow N_3(\bm{t},\boldsymbol{\nu})$ by forgetting parabolic structures.
	
	\begin{proposition}\label{twoeq} Let $j_0,j_1$ and $j_2$ be distinct elements of $\{0,1,2\}$. Assume that $\nu_{i,j_0}=\nu_{i,j_1}\neq\nu_{i,j_2}$. Then $\varphi^{-1}(b_{i,j_0})$ is the union of two projective lines $C_1$, $C_2$ such that $ Y\cap C_1$ and  $C_1\cap C_2$ consist of one point, respectively, and $Y\cap C_2=\emptyset$. Moreover, self-intersection numbers of $C_1$ and  $C_2$ are $-1$ and $-2$, respectively. 
	\end{proposition}
	\begin{proof}
		Assume that $j_0=0, j_1=1, j_2=2$.
		Put $\nu_i:=\nu_{i,0}=\nu_{i,1}, \nu'_i:=\nu_{i,2}$ and $p:=h'(t_i)(\nu_i-\res_{t_i}(\frac{dz}{z-t_3}))$. Let $a(z), b(z), h(z)$ be the polynomials defined in the proof of Proposition \ref{allneq}. Then we can see that any element $(E_1,E_2,\phi,\nabla,l^{(1)}_*, l^{(2)}_*) \in \varphi^{-1}(b_{i,0})$ have the forms
		\[
		\phi=
		\begin{pmatrix}
			1& 0& 0\\
			0& \mu& 0\\
			0& 0& 1
		\end{pmatrix}
		,\;\nabla=\phi\otimes d+
		\begin{pmatrix}
			0&\mu a(z)&\mu b(z)+\eta \prod_{m\neq i}(z-t_m)\\
			1&\mu(z-t_1)(z-t_2)-\mu p&0\\
			0&z-t_i&(z-t_1)(z-t_2)+p
		\end{pmatrix}\frac{dz}{h(z)},
		\]
		where $(\mu: \eta) \in \pl$. So we have
		\[
		\res_{t_i}\nabla-\nu_i\phi_{t_i}=\frac{1}{h'(t_i)}
		\begin{pmatrix}
			- h'(t_i)\nu_i&\mu a(t_i)&\eta \prod_{m\neq i}(t_i-t_m)\\
			1&- \mu h'(t_i)\nu'_i&0\\
			0&0&0
		\end{pmatrix}
		\]
		and
		\[
		\res_{t_i}\nabla-\nu'_i\phi_{t_i}=\frac{1}{h'(t_i)}
		\begin{pmatrix}
			- h'(t_i)\nu'_i&\mu a(t_i)&\eta \prod_{m\neq i}(t_i-t_m)\\
			1&- \mu h'(t_i)\nu_i&0\\
			0&0&h'(t_i)(\nu_i-\nu'_i)
		\end{pmatrix}.
		\]
		By definition, we have $a(t_i)=-h'(t_i)^2\nu_i\nu'_i$. If $\eta=0$, then $l^{(1)}_{i,*}$ and $ l^{(2)}_{i,*}$ are of the form
		\[
		l^{(1)}_{i,2}=\begin{pmatrix}
			-h'(t_i)\nu_i\mu\\1\\0
		\end{pmatrix},\;
		l^{(1)}_{i,1}=\mathbb{C}
		\begin{pmatrix}
			-h'(t_i)\nu_i\mu\\1\\0
		\end{pmatrix}
		+\mathbb{C}
		\begin{pmatrix}
			s\\0\\t
		\end{pmatrix},\;
		\]
		\[
		l^{(2)}_{i,2}=\mathbb{C}
		\begin{pmatrix}
			-h'(t_i)\nu_i\\1\\0
		\end{pmatrix},\;
		l^{(2)}_{i,1}=\mathbb{C}
		\begin{pmatrix}
			-h'(t_i)\nu_i\\1\\0
		\end{pmatrix}
		+\mathbb{C}
		\begin{pmatrix}
			s\\0\\t
		\end{pmatrix},
		\]
		where $(s:t) \in \pl$. 
		If $\eta \neq 0$, then 
		\[
		l^{(1)}_{i,2}=\mathbb{C}
		\begin{pmatrix}
			- h'(t_i)\nu_i\mu\\1\\0
		\end{pmatrix}, 
		l^{(1)}_{i,1}=\mathbb{C}
		\begin{pmatrix}
			1\\0\\0
		\end{pmatrix}+\mathbb{C}
		\begin{pmatrix}
			0\\1\\0
		\end{pmatrix}, \;
		l^{(2)}_{i,2}=\mathbb{C}
		\begin{pmatrix}
			-h'(t_i)\nu_i\\1\\0
		\end{pmatrix}, 
		l^{(2)}_{i,1}=\mathbb{C}
		\begin{pmatrix}
			1\\0\\0
		\end{pmatrix}+\mathbb{C}
		\begin{pmatrix}
			0\\1\\0
		\end{pmatrix}.
		\]
		By the above argument, we have
		\[
		C_1:=\overline{\{\eta\neq 0\}\cap \varphi^{-1}(b_{i,j_0})}\cong \pl, C_2:=\{\eta= 0\}\cong \pl, \varphi^{-1}(b_{i,j_0})=C_1\cup C_2
		\]
		and we find that $C_1\cap Y$ and $C_1\cap C_2$ consist of one point, respectively. 
		
		Next we consider self-intersection numbers.
		Let $a_{12}(p;z)$ be the quadratic polynomial satisfying 
		\[
		a_{12}(p;t_m)=(t_m-t_1)^2(t_m-t_2)^2-p^2- h'(t_m)^{2}(\nu_{m,0}\nu_{m,1}+\nu_{m,1}\nu_{m,2}+\nu_{m,2}\nu_{m,0})
		\]
		for $m=1,2,3$.
		Let $a_{13}(q,p,\eta; z)$ be the quadratic polynomial satisfying $a_{13}(q,p,\eta;t_i)=\eta$ and
		\[
		(t_m-q)a_{13}(q,p,\eta;t_m)
		=
		\prod_{j=0}^2 (h'(t_m)(\nu_{m,j}-\res_{t_m}(\tfrac{dz}{z-t_3}))-p)
		\]
		for $m\neq i$.  Put $E=\Opl\oplus\Opl(-1)\oplus\Opl(-1)$, 
		\[
		\nabla_{(q,p,\eta)}=d+
		\begin{pmatrix}
			0&a_{12}(p;z)&a_{13}(q,p,\eta;z)\\
			1&(z-t_1)(z-t_2)-p&0\\
			0&z-q&(z-t_1)(z-t_2)+p
		\end{pmatrix}\frac{dz}{h(z)},
		\]
		\[
		f(q,p,\eta)=(t_i-q)\eta-\prod_{j=0}^2 (h'(t_i)(\nu_{i,j}-\res_{t_i}(\tfrac{dz}{z-t_3}))-p),
		\]
		and
		\[
		X=\{f(q,p,\eta)=0\} \subset  (\mathbb{C}\setminus \{t_m\}_{m \neq i})\times \mathbb{C}\times \mathbb{C}.
		\]
		Then $(E,\nabla_{(q,p,\eta)})$ is a stable $\boldsymbol{\nu}$-connection, which induces the morphism $X \rightarrow N_3(\bm{t},\boldsymbol{\nu})$. We can see that this morphism is an open immersion, which implies that the point in $N_3(\bm{t},\boldsymbol{\nu})$ corresponding to $(q,p,\eta)=(t_i,h'(t_i)(\nu_{i}-\res_{t_i}(\tfrac{dz}{z-t_3})),0)$ is an $A_1$-singularity. Since $C_2$ is the fiber of the map $M^{\boldsymbol{\alpha}}_3(\bm{t},\boldsymbol{\nu})\rightarrow N_3(\bm{t},\boldsymbol{\nu})$ over $(t_i,h'(t_i)(\nu_{i}-\res_{t_i}(\tfrac{dz}{z-t_3})),0)$, we have $C_2^2=-2$. The morphism $\varphi$ can be factored into a composition of blow-ups, so $C_1$ must be a $(-1)$-curve.
		
		We can also prove the case of $j_2=0,1$ in the same manner.
	\end{proof}
	\begin{figure}
		\centering
		\includegraphics[width=11cm]{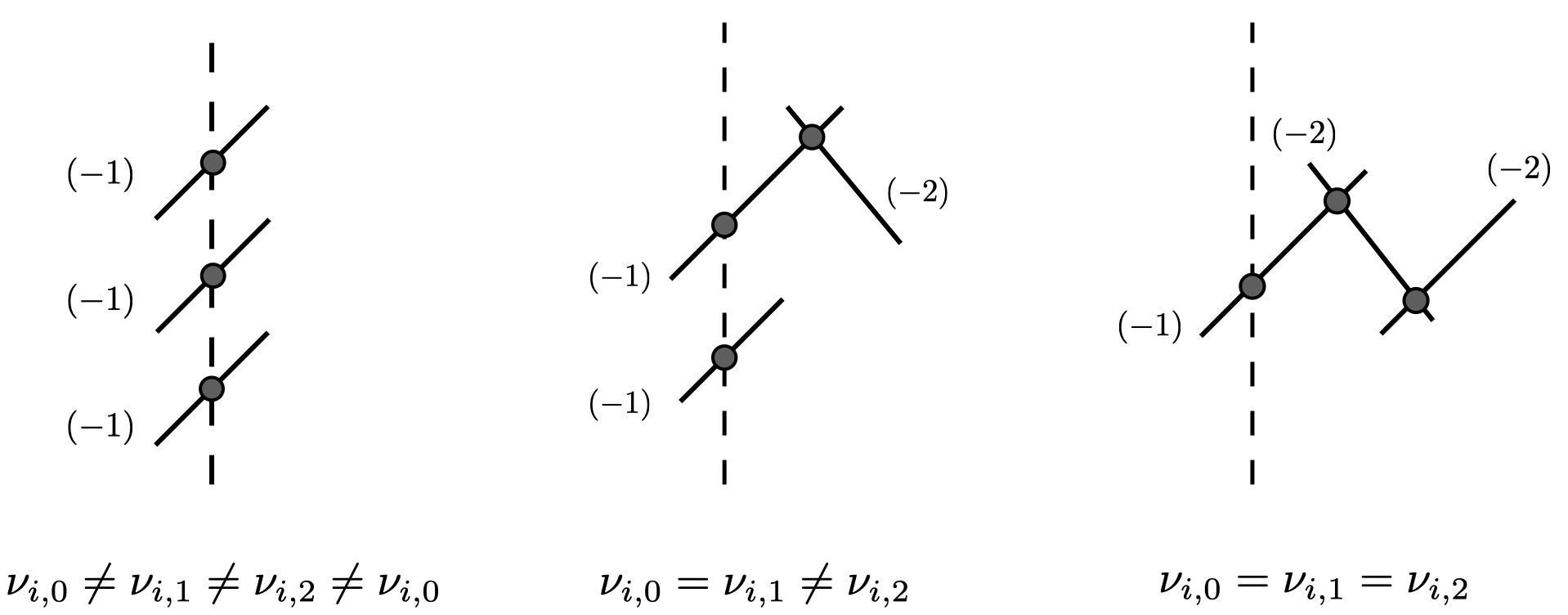}
	\end{figure}
	The following is shown in the same way of the Proposition \ref{twoeq}.
	\begin{proposition}\label{alleq}
		Assume that $\nu_{i,0}=\nu_{i,1}=\nu_{i,2}$. Then $\varphi^{-1}(b_{i,j})$ is the union of three projective lines $C_1, C_2, C_3$ such that $C_1\cap Y$, $C_1\cap C_2$, and $C_2\cap C_3$ consist of one point, $C_1\cap C_3=\emptyset$, and self-intersection numbers of $C_1$, $C_2$ and $C_3$ are $-1$, $-2$, and $-2$, respectively. 
	\end{proposition}
	\begin{proof}[Proof of Theorem 2.1] We prove (2) first. The morphism (\ref{varphidef}) extends to the morphism
		\[
		\varphi\colon\widehat{M^{\boldsymbol{\alpha}}_3}(0,0,2)\longrightarrow \mathbb{P}(\mcE).
		\]
		Let $\tilde{\mcB}$ be the reduced induced structure on $\tilde{\mcB_0}\cup\tilde{\mcB_1}\cup\tilde{\mcB_2}$. Then we can see that the restriction morphism 
		\[
		\varphi\colon\widehat{M^{\boldsymbol{\alpha}}_3}(0,0,2)\setminus \varphi^{-1}(\tilde{\mcB})\longrightarrow \mathbb{P}(\mcE)\setminus \tilde{\mcB}
		\]
		is an isomorphism by Proposition \ref{isominusblow}.  Any irreducible component of the inverse image $\varphi^{-1}(\tilde{\mcB})$ has codimension one by Zariski's main theorem. In particular, the inverse image $\varphi^{-1}(\tilde{\mcB}_2)$ is a Cartier divisor on $\widehat{M^{\boldsymbol{\alpha}}_3}(0,0,2)$, so $\varphi$ induces the morphism
		\[
		f_2\colon \widehat{M^{\boldsymbol{\alpha}}_3}(0,0,2) \longrightarrow Z_2,
		\]
		where $Z_2$ is the blow-up of $\mathbb{P}(\mcE)$ along $\tilde{\mcB}_2$. Let $Z_1$ be the blow-up of $Z_2$ along the strict transform of $\tilde{\mcB}_1$. In the same way, we obtain the morphisms $f_1\colon \widehat{M^{\boldsymbol{\alpha}}_3}(0,0,2) \rightarrow Z_1$ and $f\colon \widehat{M^{\boldsymbol{\alpha}}_3}(0,0,2) \rightarrow Z$.
		By Proposition \ref{isominusblow}, \ref{allneq}, \ref{twoeq}, and \ref{alleq}, the morphism $f_{(\bm{t},\boldsymbol{\nu})}\colon \widehat{M^{\boldsymbol{\alpha}}_3}(\bm{t},\boldsymbol{\nu}) \rightarrow Z_{(\bm{t},\boldsymbol{\nu})}$ is an isomorphism for any $(\bm{t},\boldsymbol{\nu})\in T_3\times \mcN$. So $f$ is an isomorphism. Let  $(Y_{\leq 1})_{\text{red}}$ be the reduction of $Y_{\leq 1}$. Then the composite
		\[
		Bl_W\circ f \circ \PC^{-1}\colon \overline{M^{\boldsymbol{\alpha}}_3}(0,0,2)\setminus (Y_{\leq 1})_{\text{red}}  \longrightarrow S\setminus W
		\]
		is an isomorphism, where $Bl_W\colon Z\rightarrow S$ is the blow-up along $W$. By Hartogs' theorem, the above morphism extends to the morphism $f'\colon \overline{M^{\boldsymbol{\alpha}}_3}(0,0,2)\rightarrow S$ and it becomes an isomorphism by Zariski's main theorem. By the construction of $f'$, the diagram
		\[
		\begin{tikzcd}
			\widehat{M^{\boldsymbol{\alpha}}_3}(0,0,2)\ar[r, "f"]\ar[d, "\PC"']&Z\ar[d, "Bl_W"]\\
			\overline{M^{\boldsymbol{\alpha}}_3}(0,0,2)\ar[r, "f'"]&S
		\end{tikzcd}
		\]
		becomes commutative.
		
		To prove (1), it is sufficient to show that $Y_{\leq 1}$ is reduced. Let us fix $\bm{t}=(t_i)_{1\leq i\leq 3}\in T_3$. Take a Zariski open subset  $U\subset \pl$ such that $U\cong \Spec \mathbb{C}[z]$ and $t_1, t_2, t_3 \in U\setminus \{0\}\cong \Spec \mathbb{C}[z, \frac{1}{z}]$. Let $a_{12}(u; z)$ and $a_{13}(u,v; z)$ be the quadratic polynomials in $z$ satisfying 
		\[
		a_{12}(u;t_i)=u^2(t_i-t_1)^2(t_i-t_2)^2-1- u^2h'(t_i)^{2}(\nu_{i,0}\nu_{i,1}+\nu_{i,1}\nu_{i,2}+\nu_{i,2}\nu_{i,0})
		\]
		\[
		a_{13}(u,v;t_i)= \prod_{j=0}^{2}((\nu_{i,j}-\res_{t_i}(\tfrac{dz}{z-t_3}))h'(t_i)u-1)\prod_{m\neq i}(t_mv-u)
		\]
		for $i=1,2,3$. Put $E_1=E_2=\Opl\oplus \Opl(-1)\oplus \Opl(-1)$, $\mu(u,v)=(t_1v-u)(t_2v-u)(t_3v-u)$
		\[
		\phi_{(u,v)}=
		\begin{pmatrix}
			1&0&0\\
			0&u^2\mu(u,v)&0\\
			0&0&u
		\end{pmatrix}, \ 
		\nabla_{(u,v)}=
		\begin{pmatrix}
			0&\mu(u,v) a_{12}(u;z)&a_{13}(u,v;z)\\
			1&u^2\mu(u,v)(z-t_1)(z-t_2)-u\mu(u,v)&0\\
			0&vz-u&u(z-t_1)(z-t_2)+1
		\end{pmatrix}
		\]
		and
		\[
		X=\left\{(u,v,\bm{t},\boldsymbol{\nu})\in \mathbb{C}^2\times T_3\times\mcN \vb \begin{minipage}{8cm}
			$(\nu_{i,j}-\res_{t_i}(\tfrac{dz}{z-t_3}))h'(t_i)u-1\neq 0$ for any $1\leq i\leq 3$ and $0\leq j\leq 2$ and $\bm{t}\in (U\setminus \{0\})^3$
		\end{minipage}\right\}.
		\]
		Then we can see that parabolic structures of $(l^{(1)}_*)_{(u,v)}$ and $(l^{(2)}_*)_{(u,v)}$ of $E_1$ and $E_2$, respectively, satisfying $\phi_{(u,v)}((l^{(1)}_{i,j})_{(u,v)})\subset (l^{(2)}_{i,j})_{(u,v)}$ and $(\res_{t_i}\nabla_{(u,v)}-\nu_{i,j}\phi_{(u,v)})((l^{(1)}_{i,j})_{(u,v)})\subset (l^{(2)}_{i,j+1})_{(u,v)}$ are unique. So we obtain an open immersion $X \hookrightarrow \overline{M^{\boldsymbol{\alpha}}_3}(0,0,2)$.
		Since $Y_{\leq 1}$ is defined by $u=0$, $Y_{\leq 1}$ is reduced.
		
		Finally, we prove (3). Let $\rho\colon \mathbb{P}(\Ompld\oplus \Opl)\rightarrow \mathbb{P}^2$ be the blow-down of $D_0$ and $H_i=\rho(D_i)$. Then there is a unique morphism $\varphi'\colon \overline{M^{\boldsymbol{\alpha}}_3}(\bm{t},\boldsymbol{\nu})\rightarrow \mathbb{P}^2$ such that the diagram
		\begin{equation}\label{diagmtpf}
			\begin{tikzcd}
				\widehat{M^{\boldsymbol{\alpha}}_3}(\bm{t},\boldsymbol{\nu})\ar[r, "\varphi"]\ar[d, "\PC"']&\mathbb{P}(\Ompld\oplus \Opl)\ar[d, "\rho"]\\
				\overline{M^{\boldsymbol{\alpha}}_3}(\bm{t},\boldsymbol{\nu})\ar[r, "\varphi'"]&\mathbb{P}^2
			\end{tikzcd}
		\end{equation}
		commutes. The morphism $\varphi'$ can be factored into a composition of blow-ups at a point. Let $\hat{H}_i$ be the strict transform of $H_i$ under $\varphi'$, respectively. Then we have $-K_{\overline{M^{\boldsymbol{\alpha}}_3}(\bm{t},\boldsymbol{\nu})}=\hat{H}_1+\hat{H}_2+\hat{H}_3$. So it is sufficient to show that $Y_{(\bm{t},\boldsymbol{\nu})}$ on $\overline{M^{\boldsymbol{\alpha}}_3}(\bm{t},\boldsymbol{\nu})$ has multiplicity one along $\hat{H}_i$ for each $i=1,2,3$, which is equivalent to that the strict transform $\hat{Y}_{(\bm{t},\boldsymbol{\nu})}$ of $Y_{(\bm{t},\boldsymbol{\nu})}$ under $\PC$ on $\widehat{M^{\boldsymbol{\alpha}}_3}(\bm{t},\boldsymbol{\nu})$ has multiplicity one along $\hat{D}_i$ for $i=1,2,3$, where $\hat{D}_i$ is that the strict transform of $D_i$ under $\varphi$. Let $b_{12}(p;z)$ be the quadratic polynomial in $z$ satisfying 
		\[
		b_{12}(p;t_m)=(t_m-t_1)^2(t_m-t_2)^2-p^2- h'(t_m)^{2}(\nu_{i,0}\nu_{i,1}+\nu_{i,1}\nu_{i,2}+\nu_{i,2}\nu_{i,0})
		\]
		for $m=1,2,3$. Let $b_{13}(q,p; z)$ be the quadratic polynomial in $z$ satisfying $b_{13}(q,p;t_i)=0$ and
		\[
		(t_m-q)b_{13}(q,p;t_m)
		= \prod_{j=0}^2 (h'(t_m)(\nu_{m,j}-\res_{t_m}(\tfrac{dz}{z-t_3}))-p)
		\]
		for $m\neq i$.  Put
		\[
		f(q,p,\mu)=h'(t_i)(t_i-q)-\mu \prod_{j=0}^2 (h'(t_i)(\nu_{i,j}-\res_{t_i}(\tfrac{dz}{z-t_3}))-p)
		\]
		and 
		\[
		X=\{f(q,p,\mu)=0\}\subset (\mathbb{C}\setminus\{t_m\}_{m\neq i})\times (\mathbb{C}\setminus \{h'(t_i)(\nu_{i,j}-\res_{t_i}(\tfrac{dz}{z-t_3}))\}_{0\leq j\leq 2})\times \mathbb{C}.
		\]
		Then the family of parabolic $\phi$-connections defined by 
		\begin{equation}\label{rk2coord}
			\phi_\mu=
			\begin{pmatrix}
				1& 0& 0\\
				0& \mu& 0\\
				0& 0& 1
			\end{pmatrix}
			,\;\nabla_{(q,p,\mu)}=\phi_\mu\otimes d+
			\begin{pmatrix}
				0&\mu b_{12}(p;z)&\mu b_{13}(q,p;z)+\prod_{m\neq i}(z-t_m)\\
				1&\mu(z-t_1)(z-t_2)-\mu p&0\\
				0&z-q&(z-t_1)(z-t_2)+p
			\end{pmatrix}\frac{dz}{h(z)}
		\end{equation}
		gives an open immersion $\iota \colon X \hookrightarrow \widehat{M^{\boldsymbol{\alpha}}_3}(\bm{t},\boldsymbol{\nu})$.
		In particular, $\iota^*\hat{Y}_{(\bm{t},\boldsymbol{\nu})}$ is defined by $\mu=0$. So $\hat{Y}_{(\bm{t},\boldsymbol{\nu})}$ on  $\widehat{M^{\boldsymbol{\alpha}}_3}(\bm{t},\boldsymbol{\nu})$ has multiplicity one along $D_i$.
	\end{proof}
	\subsection{Explicit correspondence between points on the $A^{(1)*}_2$-surface and parabolic $\phi$-connections}\label{expcorr}
	Let $(c_0:c_1)$ be a homogeneous coordinate of $\pl$ and $(z_0: z_1: z_2)$ be a homogeneous coordinate of $\mathbb{P}^2$. 
	We consider the case $t_1=(1:0), t_2=(1:1), t_3=(0:1)$ and $\nu_{i,0}\neq \nu_{i,1}\neq \nu_{i,2}\neq \nu_{i,0}$ for each $i=1,2,3$.
	Then we can see that the morphism $\varphi'\colon \overline{M^{\boldsymbol{\alpha}}_3}(\bm{t},\boldsymbol{\nu})\rightarrow \mathbb{P}^2$ defined in (\ref{diagmtpf}) is the blow-up at the following nine points;
	\[
	\begin{tabular}{lll}
		$p_1: (1: \nu_{1,0}: 1)$,&$p_2: (1: \nu_{1,1}: 1)$, & $p_3: (1: \nu_{1,2}: 1)$, \\
		$p_4: (0:-\nu_{0,0}: 1)$,&$p_5: (0:-\nu_{0,1}: 1)$,& $p_9: (0:-\nu_{0,2}: 1)$, \\
		$p_6: (1: -\nu_{\infty,0}+1: 0)$,&$p_7: (1: -\nu_{\infty,1}+1: 0)$,&  $p_8: (1: -\nu_{\infty,2}+1: 0)$.
	\end{tabular}
	\]
	The numbering follows Sakai's paper (see Appendix B in \cite{Sa1}).
	Put
	\[
	U_0:=\{c_0\neq 0\}\subset \pl, \ U_\infty:=\{c_1\neq 0\} \subset \pl, \ z:=c_1/c
	_0, \ w:=c_0/c_1.
	\]
	We take a local basis $e^{(0)}_0, e^{(0)}_1, e^{(0)}_2$ (resp. $e^{(\infty)}_0, e^{(\infty)}_1, e^{(\infty)}_2$) of $E\cong \Opl\oplus\Opl(-1)\oplus\Opl(-1)$ on $U_0$ (resp. on $U_\infty$) satisfying 
	$e^{(0)}_0=e^{(\infty)}_0$, $e^{(0)}_1=\frac{1}{w}e^{(\infty)}_1$,$e^{(0)}_2=\frac{1}{w}e^{(\infty)}_2$. For simplicity of notation, we write $\phi=A$ on $U_0$ (resp. on $U_\infty$) instead of $\phi(e^{(0)}_0, e^{(0)}_1, e^{(0)}_2)=(e^{(0)}_0, e^{(0)}_1, e^{(0)}_2)A$ (resp. $\phi(e^{(\infty)}_0, e^{(\infty)}_1, e^{(\infty)}_2)=(e^{(\infty)}_0, e^{(\infty)}_1, e^{(\infty)}_2)A$). We use a similar expression for a $\phi$-connection $\nabla$.

	The correspondence between points on $\mathbb{P}^2\setminus \{p_i\}_{1\leq i\leq 9}$ and parabolic $\phi$-connections is as follows;
	\begin{itemize}
		\item $z_0, z_2\neq 0$
		\[
		(q:p:1)\longleftrightarrow 
		\phi=\begin{pmatrix}
			1& 0& 0\\
			0& 1& 0\\
			0& 0& 1
		\end{pmatrix}, \ \nabla=d+
		\begin{pmatrix}
			0& a_{12}(z)& a_{13}(z)\\
			1& -p& 0\\
			0& z-q& p
		\end{pmatrix}
		\frac{dz}{z(z-1)}\quad \text{on $U_0$}
		\]
		or
		\[
		(1:p':q')\longleftrightarrow 
		\phi=\begin{pmatrix}
			1& 0& 0\\
			0& 1& 0\\
			0& 0& 1
		\end{pmatrix}, \ \nabla=d+
		\begin{pmatrix}
			0& b_{12}(w)& b_{13}(w)\\
			1& w-1-p'& 0\\
			0& w-q'& w-1+p'
		\end{pmatrix}
		\frac{dw}{w(w-1)} \quad \text{on $U_\infty$}
		\]
		
		\item$z_0=0, z_2\neq 0$
		\[
		(0:p:1)\longleftrightarrow 
		\phi=
		\begin{pmatrix}
			1& 0& 0\\
			0& 0& 0\\
			0& 0& 1
		\end{pmatrix}
		, \ \nabla=\phi \otimes d+
		\begin{pmatrix}
			0& 0& z-1\\
			1& 0& 0\\
			0& z& p
		\end{pmatrix}
		\frac{dz}{z(z-1)}\quad \text{on $U_0$}
		\]
		
		\item$z_0=z_2\neq 0$
		\[
		(1:p:1)\longleftrightarrow 
		\phi=
		\begin{pmatrix}
			1& 0& 0\\
			0& 0& 0\\
			0& 0& 1
		\end{pmatrix}
		, \ \nabla=\phi \otimes d+
		\begin{pmatrix}
			0& 0& z\\
			1& 0& 0\\
			0& z-1& p
		\end{pmatrix}
		\frac{dz}{z(z-1)}\quad \text{on $U_0$}
		\]
		
		\item$z_0\neq0, z_2=0$
		\[
		(1:p':0)\longleftrightarrow 
		\phi=\begin{pmatrix}
			1& 0& 0\\
			0& 0& 0\\
			0& 0& 1
		\end{pmatrix}
		, \ \nabla=\phi \otimes d+
		\begin{pmatrix}
			0& 0& w-1\\
			1& 0& 0\\
			0& w& p'
		\end{pmatrix}
		\frac{dw}{w(w-1)} \quad \text{on $U_\infty$}
		\]
		
		\item$z_0=z_2=0$
		\[
		(0:1:0)\longleftrightarrow 
		\phi=\begin{pmatrix}
			1& 0& 0\\
			0& 0& 0\\
			0& 0& 0
		\end{pmatrix}, \ \nabla=\phi \otimes d+
		\begin{pmatrix}
			0& z& 0\\
			1& 0& 0\\
			0& z& z-1
		\end{pmatrix}
		\frac{dz}{z(z-1)} \quad \text{on $U_0$}
		\]
	\end{itemize}
	Here $a_{12}(z), a_{13}(z), b_{12}(w), b_{13}(w)$ are the quadratic polynomials satisfying
	\begin{align*}
		a_{12}(0)=&-p^2-(\nu_{0,0}\nu_{0,1}+\nu_{0,1}\nu_{0,2}+\nu_{0,2}\nu_{0,0}),\\
		a_{12}(1)=&-p^2-(\nu_{1,0}\nu_{1,1}+\nu_{1,1}\nu_{1,2}+\nu_{1,2}\nu_{1,0}),\\
		\lim_{z\rightarrow \infty}a_{12}(z)/z^2=&1-(\nu_{\infty,0}\nu_{\infty,1}+\nu_{\infty,1}\nu_{\infty,2}+\nu_{\infty,2}\nu_{\infty,0}),\\
		a_{13}(0)=&(p+\nu_{0,0})(p+\nu_{0,1})(p+\nu_{0,2})/q,\\
		a_{13}(1)=&(p-\nu_{1,0})(p-\nu_{1,1})(p-\nu_{1,2})/(q-1),\\
		\lim_{z\rightarrow \infty}a_{13}(z)/z^2=&(1-\nu_{\infty,0})(1-\nu_{\infty,1})(1-\nu_{\infty,2}),\\
		b_{12}(0)=&1-p'^2-(\nu_{\infty,0}\nu_{\infty,1}+\nu_{\infty,1}\nu_{\infty,2}+\nu_{\infty,2}\nu_{\infty,0}),\\
		b_{12}(1)=&-p'^2-(\nu_{1,0}\nu_{1,1}+\nu_{1,1}\nu_{1,2}+\nu_{1,2}\nu_{1,0}),\\
		\lim_{w\rightarrow \infty}b_{12}(w)/w^2=&-(\nu_{0,0}\nu_{0,1}+\nu_{0,1}\nu_{0,2}+\nu_{0,2}\nu_{0,0}),\\
		b_{13}(0)=&(p'-1+\nu_{\infty,0})(p'-1+\nu_{\infty,1})(p'-1+\nu_{\infty,2})/q',\\
		b_{13}(1)=&(p'-\nu_{1,0})(p'-\nu_{1,1})(p'-\nu_{1,2})/(q'-1),\\
		\lim_{w\rightarrow \infty}b_{13}(w)/w^2=&\nu_{0,0}\nu_{0,1}\nu_{0,2}.
	\end{align*}
	We can see that the parabolic connections corresponding to $(q:p:1)$ and $(1:p':q')$ are isomorphic to each other when $q'=q^{-1}$ and $p'=pq^{-1}$.
	
	The correspondence between points on the strict transform at $p_i$ and parabolic $\phi$-connections is as follows;
	\begin{itemize}
		\item The strict transform at $(0: -\nu_{0,j}: 1)$
		\[
		(\mu:\eta)\longleftrightarrow 
		\phi=\begin{pmatrix}
			1& 0& 0\\
			0& \mu& 0\\
			0& 0& 1
		\end{pmatrix}
		, \ \nabla=\phi \otimes d+
		\begin{pmatrix}
			0& \mu a_{12}(p;z)& \mu c_0(z)+\eta (z-1)\\
			1& \mu \nu_{0,j}& 0\\
			0& z& -\nu_{0,j}
		\end{pmatrix}
		\frac{dz}{z(z-1)} \quad \text{on $U_0$}
		\]
		\item The strict transform at $(1: \nu_{1,0}: 1)$
		\[
		(\mu:\eta)\longleftrightarrow 
		\phi=\begin{pmatrix}
			1& 0& 0\\
			0& \mu& 0\\
			0& 0& 1
		\end{pmatrix}
		, \ \nabla=\phi \otimes d+
		\begin{pmatrix}
			0& \mu a_{12}(p;z)& \mu c_1(z)+\eta z\\
			1& -\mu \nu_{1,j}& 0\\
			0& z-1& \nu_{1,j}
		\end{pmatrix}
		\frac{dz}{z(z-1)} \quad \text{on $U_0$}
		\]
		\item The strict transform at $(1, -\nu_{\infty,j}+ 1: 0)$
		\[
		(\mu:\eta)\longleftrightarrow 
		\phi=\begin{pmatrix}
			1& 0& 0\\
			0& \mu& 0\\
			0& 0& 1
		\end{pmatrix}
		, \nabla=\phi \otimes d+
		\begin{pmatrix}
			0& \mu b_{12}(w)& \mu c_\infty(w)+\eta (w-1)\\
			1& \mu (w-1+\nu_{0,j})& 0\\
			0& w& w-1-\nu_{0,j}
		\end{pmatrix}
		\hspace{-2pt}\frac{dw}{w(w-1)} \ \text{on $U_\infty$}
		\]
	\end{itemize}
	Here
	\begin{align*}
		c_0(z)=&(1-\nu_{\infty,0})(1-\nu_{\infty,1})(1-\nu_{\infty,2})z(z-1)+(\nu_{0,j}+\nu_{1,0})(\nu_{0,j}+\nu_{1,1})(\nu_{0,j}+\nu_{1,2})z,\\
		c_1(z)=&(1-\nu_{\infty,0})(1-\nu_{\infty,1})(1-\nu_{\infty,2})z(z-1)-(\nu_{1,j}+\nu_{0,0})(\nu_{1,j}+\nu_{0,1})(\nu_{1,j}+\nu_{0,2})(z-1),\\
		c_\infty(w)=&\nu_{0,0}\nu_{0,1}\nu_{0,2}w(w-1)-(1-\nu_{\infty,j}-\nu_{1,0})(1-\nu_{\infty,j}-\nu_{1,1})(1-\nu_{\infty,j}-\nu_{1,2})w.
	\end{align*}
	\section{Moduli space of parabolic bundles and parabolic connections}
	\subsection{Moduli space of $w$-stable parabolic bundles}
	In this subsection, we determine $w$-stable parabolic bundles with degree $-2$ and investigate the moduli space and the wall-crossing behavior. Let us fix $\bm{t}\in T_3$.
	\begin{definition}
		A rank 3 parabolic bundle $(E,l_*)$ over $(\pl,\bm{t})$ is said to be $\boldsymbol{\alpha}$-stable if for any nonzero subbundle $F \subsetneq E$, the inequality 
		\begin{equation}\label{pbstabineq}
			\frac{\deg F + \sum_{i=1}^{3} \sum_{j=1}^{3} \alpha_{i,j} d_{i,j}(F)}{\rank F}<\frac{\deg E + \sum_{i=1}^{3} \sum_{j=1}^{3} \alpha_{i,j}}{\rank E}
		\end{equation}
		holds, where $d_{i,j}(F)=\dim (F|_{t_i}\cap l_{i,j-1})/(F|_{t_i}\cap l_{i,j})$.
	\end{definition}
	We assume that
	\[
	\alpha_{1,3}-\alpha_{1,2}=\alpha_{1,2}-\alpha_{1,1}=\alpha_{2,3}-\alpha_{2,2}=\alpha_{2,2}-\alpha_{2,1}=\alpha_{3,3}-\alpha_{3,2}=\alpha_{3,2}-\alpha_{3,1}=:w. 
	\]
	Then we have $0<w<1/2$. We consider the case of $\deg E=-2$. Take a nonzero subbundle $F\subsetneq E$. If $\rank F=2$, then the inequality (\ref{pbstabineq}) is equivalent to 
	\begin{equation}
		-4-3\deg F+\sum_{i=1}^3\sum_{j=1}^3\alpha_{i,j}(2-3d_{i,j}(F))>0,
	\end{equation}
	and we have
	\[
	\sum_{j=1}^3\alpha_{i,j}(2-3d_{i,j}(F))=\left\{
	\begin{array}{lll}
		-3w &\quad F|_{t_i}= l_{i,1} \\
		0&\quad   F|_{t_i}\neq l_{i,1},  F|_{t_i}\supset l_{i,2}\\
		3w&\quad F|_{t_i}\nsupseteq l_{i,2}.
	\end{array}
	\right.
	\]
	In the case of $\rank F=1$, (\ref{pbstabineq}) is equivalent to 
	\begin{equation}
		-2-3\deg F+\sum_{i=1}^3\sum_{j=1}^3\alpha_{i,j}(1-3d_{i,j}(F))>0,
	\end{equation}
	and we have
	\[
	\sum_{j=1}^3\alpha_{i,j}(1-3d_{i,j}(F))=\left\{
	\begin{array}{lll}
		3w &\quad F|_{t_i}\nsubseteq l_{i,1} \\
		0&\quad  F|_{t_i}\subset l_{i,1},  F|_{t_i}\neq l_{i,2}\\
		-3w&\quad F|_{t_i}= l_{i,2}.
	\end{array}
	\right.
	\]
	The stability condition is determined by $w$ under the assumption, so we call the special case of the $\boldsymbol{\alpha}$-stability the $w$-stability.
	
	Let $(E,l_*)$ be a $w$-stable parabolic bundle with $\deg E=-2$. The vector bundle $E$ can be written by the form $\Opl(m_1)\oplus \Opl(m_2)\oplus \Opl(m_3)$, where $m_1\geq m_2\geq m_3$ and $m_1+m_2+m_3=-2$. Suppose that $m_1\geq 1$. Then we can see that $\Opl(m_1)$ breaks the stability. Hence $E$ is isomorphic to $\Opl\oplus \Opl(-1)\oplus \Opl(-1)$. Suppose that $\Opl|_{t_i}= l_{i,2}$ for some $i$. Then $\Opl$ breaks the stability.
	So $\Opl|_{t_i}\neq  l_{i,2}$ for any $i$. Let $l'_i$ be the image of $l_{i,2}$ by the quotient $E|_{t_i}\rightarrow (E/\Opl)|_{t_i}$. Since $\Opl|_{t_i}\neq  l_{i,2}$, $l'_i$ is not zero for any $i$. For a parabolic structure $l'_*=\{l'_i\}_{1\leq i\leq 3}$ on $\Opl(-1)^{\oplus 2}$, put
	\[
	n(l'_*):=\max_{\Opl(-1)\cong F\subset \Opl(-1)^{\oplus 2}}\#\{i\mid F|_{t_i}=l'_i\}.
	\]
	A parabolic bundle $(\Opl(-1)^{\oplus 2}, l'_*)$ with $n(l'_*)=1$ and $3$ is unique up to isomorphism, respectively. When $n(l'_*)=2$, there are three isomorphism classes of such parabolic bundles, that is, those isomorphism classes are determined by the pair of numbers $1\leq i< j\leq 3$. Let $(*)$ be the following condition;
	\begin{itemize}
		\item[$(*)$] There is no  subbundle $F\subset E$ such that  $F\cong \Opl(-1)^{\oplus 2}, l_{i,2}\subset F|_{t_i}$ and $F|_{t_j}=l_{j,1}$ for some $i$  and any $j\neq i$. 
	\end{itemize}
	\begin{proposition}\label{pbmod}
		Let $P^w(-2)$ be the moduli space of $w$-stable parabolic bundles over $(\pl,\bm{t})$ of rank 3 and degree $-2$.
		\begin{itemize}
			\item[(1)] If $0<w<2/9, 4/9<w<1/2$, then $P^w(-2)=\emptyset$.
			\item[(2)] If $2/9<w<1/3$, then a $w$-stable parabolic bundle $(E,l_*)$ fits into a nonsplit exact sequence 
			\begin{equation}\label{pbextprop}
				0\longrightarrow (\Opl, \emptyset)\longrightarrow (E,l_*)\longrightarrow (\Opl(-1)^{\oplus 2}, l'_*)\longrightarrow 0,
			\end{equation}
			where $n(l'_*)=1$. In particular, $P^w(-2)$ is isomorphic to $\pl$.
			\item[(3)] If $1/3<w<4/9$, then a $w$-stable parabolic bundle $(E,l_*)$ is either type of the following:
			\begin{itemize}
				\item[(i)] $E\cong \Opl \oplus \Opl(-1)\oplus \Opl(-1)$, $\#\{i\mid \Opl|_{t_i}\subset l_{i,1}\}=0$, $n(l'_*)=1$, and the condition $(*)$ holds.
				\item[(ii)] $E\cong \Opl \oplus \Opl(-1)\oplus \Opl(-1)$, $\#\{i\mid \Opl|_{t_i}\subset l_{i,1}\}= 1$, $n(l'_*)=1$, and the condition $(*)$ holds. 
			\end{itemize}
			In particular, $P^w(-2)$ is isomorphic to $\pl$.
		\end{itemize}
	\end{proposition}
	\begin{proof}
		
		Assume that $w<2/9$. Then $\Opl$ breaks the stability. In particular, we have $P^{\boldsymbol{\alpha}}(-2)=\emptyset$. 
		
		Assume that $2/9<w<1/3$. If $\Opl|_{t_i}\subset l_{i,1}$ for some $i$, then $\Opl$ breaks the stability. 
		So $\Opl|_{t_i}\nsubseteq l_{i,1}$ for any $i$. Hence $(E,l_*)$ fits into an exact sequence 
		\begin{equation}\label{pbext}
			0\longrightarrow (\Opl, \emptyset)\longrightarrow (E,l_*)\longrightarrow (\Opl(-1)^{\oplus 2}, l'_*=\{l'_i\}_{1\leq i\leq 3})\longrightarrow 0.
		\end{equation}
		If (\ref{pbext}) splits, that is, there exists a subbundle $F$ such that $F\cong \Opl(-1)^{\oplus 2}$ and $F|_{t_i}=l_{i,1}$ for all $i$, then $F$ breaks the stability. 
		So (\ref{pbext}) does not split. When $n(l'_*)\geq 2$, we can take a subbundle $F\subset E$ satisfying $F\cong \Opl(-1)$ and $F|_{t_i}=l_{i,2}, F|_{t_j}=l_{j,2}$ for some $1\leq i<j\leq 3$. Then $F$ breaks the stability.
		Hence $n(l'_*)=1$ and we have
		\[
		P^w(-2)\cong \mathbb{P}\Ext^1((\Opl(-1)^{\oplus 2}, l'_*), (\Opl, \emptyset))\cong \mathbb{P}H^1((\Opl(1)^{\oplus 2})(-D))\cong \pl.
		\]
		Assume that $1/3<w<1/2$.  When $n(l'_*)\geq 2$, we can take a subbundle $F\subset E$ satisfying $F\cong \Opl(-1)$ and $F|_{t_i}=l_{i,2}, F|_{t_j}=l_{j,2}$ for some $1\leq i<j\leq 3$. Then $F$ breaks stability. Hence $n(l'_*)= 1$. In this case, we can take a unique subbundle $F\subset E$ such that $F\cong \Opl(-2)$ and $F|_{t_i}=l_{i,2}$ for any $i$, and we have
		\[
		-2-3\deg F+\sum_{i=1}^3\sum_{j=1}^3\alpha_{i,j}(1-3d_{i,j}(F))=4-9w.
		\]
		So $P^w(-2)=\emptyset$ if $w>4/9$. Assume that $1/3<w<4/9$. When $\#\{i\mid \Opl|_{t_i}\subset l_{i,1}\}\geq 2$, $\Opl$ breaks the stability. 
		Hence $\#\{i\mid \Opl|_{t_i}\subset l_{i,1}\}\leq 1$. We consider the case $\Opl|_{t_i}\nsubseteq l_{i,1}$ for any $i$.  Then we can take a unique subbundle $F_{ij}\subset E$ such that $F_{ij}\cong \Opl(-1)^{\oplus 2}, F_{ij}|_{t_i}=l_{i,1}$ and $F_{ij}|_{t_j}=l_{j,1}$ for each $1\leq i<j\leq 3$. If $l_{m,2}\subset F_{ij}|_{t_m}$ for $m\neq i,j$, then $F$ breaks the stability.
		So such a parabolic bundle becomes $w$-unstable, which is a contradiction. We can see  that such a parabolic bundle $p_{ij}\in \mathbb{P}\Ext^1((\Opl(-1)^{\oplus 2}, l'_*), (\Opl, \emptyset))$ is unique for each $1\leq i<j\leq 3$. Next we consider the case $\Opl|_{t_i}\nsubseteq l_{m,1}$ for some $m$. Let $i,j$ be different elements of $\{1,2,3\}\setminus \{m\}$. Then we can take a unique subbundle $F_{ij}\subset E$ such that $F_{ij}\cong \Opl(-1)^{\oplus 2}, F_{ij}|_{t_i}=l_{i,1}$ and $F_{ij}|_{t_j}=l_{j,1}$. For the same reason as the above, we have $l_{m,2}\nsubseteq F|_{t_m}$. We can see that such a parabolic bundle $p_m$ is unique up to isomorphism. Therefore we have
		\[
		P^w(-2)\cong (\mathbb{P}\Ext^1((\Opl(-1)^{\oplus 2}, l'_*), (\Opl, \emptyset))\setminus \{p_{12},p_{13},p_{23}\})\sqcup \{p_1,p_2,p_3\}\cong \pl.
		\]
	\end{proof}
	
	As the above proof shows, $p_{12}, p_{13}, p_{23}$ become $w$-unstable and $p_1, p_2, p_3$ become $w$-stable when $w$ is across $1/3$.
	Let us investigate this in detail. Assume that $2/9<w<1/3$. In this case, a $w$-stable parabolic bundle $(E,l_*)$ fits into a nonsplit exact sequence (\ref{pbextprop}). Then we can take nonzero homomorphisms $s_1,s_2\colon \Opl(-1)\rightarrow E$ satisfying $l_{1,2}=(\Image s_1)|_{t_1}$, $l_{2,2}=(\Image s_2)|_{t_2}$, $0\neq (\Image s_1)|_{t_2}\subset l_{2,1}$, $0\neq (\Image s_2)|_{t_1}\subset l_{1,1}$. Let $e_1,e_2$ be local basis corresponding to $s_1,s_2$, respectively, and $e_0$ be the nonzero section of $\Opl\subset E$. Let us denote $ae_0+be_1+ce_2$ by the matrix $^t(a\;b\;c)$. Since $n(l'_*)=1$,  we can wright $l_*$  by the form
	\[
	l_{1,2}=\mathbb{C}
	\begin{pmatrix}
		0\\1\\0
	\end{pmatrix}, \; 
	l_{1,1}=\mathbb{C}
	\begin{pmatrix}
		0\\1\\0
	\end{pmatrix}+
	\mathbb{C}
	\begin{pmatrix}
		0\\0\\1
	\end{pmatrix}, \quad
	l_{2,2}=\mathbb{C}
	\begin{pmatrix}
		0\\0\\1
	\end{pmatrix}, \; 
	l_{2,1}=\mathbb{C}
	\begin{pmatrix}
		0\\1\\0
	\end{pmatrix}+
	\mathbb{C}
	\begin{pmatrix}
		0\\0\\1
	\end{pmatrix}
	\]
	\[
	l_{3,2}=\mathbb{C}
	\begin{pmatrix}
		a+b\\1\\1
	\end{pmatrix}, \;
	l_{3,1}=\mathbb{C}
	\begin{pmatrix}
		a\\1\\0
	\end{pmatrix}+
	\mathbb{C}
	\begin{pmatrix}
		b\\0\\1
	\end{pmatrix}, 
	\]
	where $a,b\in \mathbb{C}$. The exact sequence (\ref{pbextprop}) splits if and only if $(a,b)=(0,0)$, and parabolic bundles defined by $(a,b), (a',b')$ are isomorphic to each other if and only if $(a,b), (a',b')$ are the same up to scalar multiplicities. In this way, we also prove that $P^w(-2)\cong \pl$. The parabolic bundles $p_{12}, p_{13}, p_{23}$ in the proof of Proposition \ref{pbmod} correspond to the case $a+b=0, b=0,a=0$, respectively. Let us fix $a\neq 0$ and put $\mu=a+b$. Let $\tilde{l}_*$ be the parabolic structure defined by
	\[
	\tilde{l}_{1,2}=\mathbb{C}
	\begin{pmatrix}
		0\\1\\0
	\end{pmatrix}, \; 
	\tilde{l}_{1,1}=\mathbb{C}
	\begin{pmatrix}
		0\\1\\0
	\end{pmatrix}+
	\mathbb{C}
	\begin{pmatrix}
		0\\0\\1
	\end{pmatrix}, \quad
	\tilde{l}_{2,2}=\mathbb{C}
	\begin{pmatrix}
		0\\0\\1
	\end{pmatrix}, \; 
	\tilde{l}_{2,1}=\mathbb{C}
	\begin{pmatrix}
		0\\1\\0
	\end{pmatrix}+
	\mathbb{C}
	\begin{pmatrix}
		0\\0\\1
	\end{pmatrix}
	\]
	\[
	\tilde{l}_{3,2}=\mathbb{C}
	\begin{pmatrix}
		1\\1\\1
	\end{pmatrix}, \;
	\tilde{l}_{3,1}=\mathbb{C}
	\begin{pmatrix}
		1\\ \frac{\mu}{a}\\0
	\end{pmatrix}+
	\mathbb{C}
	\begin{pmatrix}
		1\\1\\ 1
	\end{pmatrix}.
	\]
	When $\mu\neq 0$, the homomorphism defined by the matrix $\textrm{diag}(\mu,1,1)$ is an isomorphism from $(E,\tilde{l}_*)$ to $(E,l_*)$. When $\mu=0$, $(E,\tilde{l}_*)$ and $(E,l_*)$ are parabolic bundles corresponding to $p_3$ and $p_{12}$ in the proof of Proposition \ref{pbmod}, respectively. So $p_3$ and $p_{12}$ are infinitesimally close to each other. In the same way, we can see that $p_1, p_2$ are infinitesimally close to $p_{23}, p_{13}$, respectively. 
	
	\subsection{Moduli space of $\lambda$-connections}
	In this subsection, we consider the compactification of the moduli space of parabolic connections by using $\lambda$-connections.
	\begin{definition}
		A $\boldsymbol{\nu}$-parabolic $\lambda$-connection is a collection $(\lambda, E,\nabla, l_*=\{l_{i,*}\}_{1\leq i\leq 3})$ consisting the following data:
		\begin{itemize}
			\setlength{\itemsep}{0cm}
			\item[(1)] $E$ is a vector bundle on $\pl$ of rank 3 and degree $-2$,
			\item[(2)]  $\nabla \colon E \rightarrow E \otimes \Ompld$ is a $\lambda$-twisted logarithmic connection, i.e.  $\nabla(fa)=a\otimes \lambda df + f \nabla(a)$ for any $f \in \Opl, a \in E_1$, and 
			\item[(3)] $l_{i,*}$ is a filtration $E|_{t_i}=l_{i,0} \supsetneq l_{i,1} \supsetneq  l_{i,2} \supsetneq l_{i,3}=0$ satisfying  $(\res_{t_i}(\nabla)-\nu_{i,j}\id )({l_{i,j}})\subset l_{i,j+1}$ for $i=1,2,3$ and $j=0,1,2$.
		\end{itemize}
	\end{definition}
	A $\boldsymbol{\nu}$-parabolic 0-connection is a parabolic Higgs bundle, and a $\boldsymbol{\nu}$-parabolic 1-connection is a $\boldsymbol{\nu}$-parabolic connection. 
	Let $\overline{M^w_3(\bm{t},\boldsymbol{\nu})^0}$ be the moduli space of $\lambda\boldsymbol{\nu}$-parabolic $\lambda$-connections over $(\pl,\bm{t})$ whose underlying parabolic bundle is $w$-stable, that is,
	\[
	\overline{M^w_3(\bm{t},\boldsymbol{\nu})^0}:=\left\{(\lambda, E,\nabla, l_*) \vb (E,l_*)\in P^w(-2)\right\}/\sim.
	\]
	Here two objects $(\lambda_1, E_1,\nabla_1,(l_1)_*), (\lambda_2, E_2,\nabla_2,(l_2)_*)$ are equivalent if there exists an isomorphism $\sigma \colon (E_1,(l_1)_*)\rightarrow (E_2,(l_2)_*)$ and $\mu \in \mathbb{C}^*$ such that the diagram
	\[
	\begin{tikzcd}
		E_1 \arrow[r,"\nabla_1"] \arrow[d, "\sigma"']&E_1 \otimes \Ompld \arrow[d, "\sigma \otimes \id "] \\
		E_2 \arrow[r,"\mu\nabla_2"]&E_2 \otimes \Ompld 
	\end{tikzcd}
	\]
	commutes. Take $(E,l_*) \in P^w(-2)$ and a $\boldsymbol{\nu}$-logarithmic connection $\nabla$ over $(E,l_*)$. All $\lambda\boldsymbol{\nu}$-logarithmic $\lambda$-connections over $(E,l_*)$ are of the form $\lambda\nabla+\Phi$, where $\Phi$ is a parabolic Higgs field over $(E,l_*)$. 
	The space of all isomorphism classes of $\lambda\boldsymbol{\nu}$-logarithmic $\lambda$-connections over $(E,l_*)$ is $\mathbb{P}(\mathbb{C}\nabla\oplus H)$
	and it can be regarded as a compactification of the space of all $\boldsymbol{\nu}$-logarithmic connections over $(E,l_*)$. Here $H$ is the space of all parabolic Higgs fields over $(E,l_*)$. In particular, $\overline{M^w_3(\bm{t},\boldsymbol{\nu})^0}$ is a compactification of a Zariski open subset 
	\[
	M^w_3(\bm{t},\boldsymbol{\nu})^0:=\left\{(E,\nabla, l_*) \vb (E,l_*)\in P^w(-2)\right\}/\sim
	\]
	of $M^w_3(\bm{t},\boldsymbol{\nu})$. The boundary is the locus defined by $\lambda=0$ on $\overline{M^w_3(\bm{t},\boldsymbol{\nu})^0}$ and is isomorphic to the projectivization $\mathbb{P}T^*P^w(-2)$ of the cotangent bundle of $P^w(-2)$ because $T^*P^w(-2)$ is the moduli space of parabolic Higgs bundles whose underlying parabolic bundles are $w$-stable.
	The following result when $\nu_{1,0}+\nu_{2,0}+\nu_{3,0}= 0$ is a version of Proposition 4.6 in \cite{LS} in the present setting. 
	\begin{theorem}\label{pbthm}
		Assume that $2/9<w<1/3$. Then we have 
		\[
		\overline{M^w_3(\bm{t},\boldsymbol{\nu})^0}\cong 
		\left\{
		\begin{array}{lll}
			\pl\times \pl &\nu_{1,0}+\nu_{2,0}+\nu_{3,0}\neq 0\\
			\mathbb{P}(\Opl\oplus \Opl(-2)) &\nu_{1,0}+\nu_{2,0}+\nu_{3,0}= 0.
		\end{array}
		\right.
		\]
	\end{theorem}
	\begin{proof}
		Let $U_0:=\mathbb{C}$ and $U_\infty:=\mathbb{C}$. For $a\in U_0$ and $b\in U_\infty$, let us define a parabolic structure $(l_a)_*$ and $(l_b)_*$ on $\Opl\oplus \Opl(-1)\oplus \Opl(-1)$ by
		\[
		(l_a)_{1,2}=(l_b)_{1,2}=\mathbb{C}
		\begin{pmatrix}
			0\\1\\0
		\end{pmatrix}, \; 
		(l_a)_{1,1}=(l_b)_{1,1}=\mathbb{C}
		\begin{pmatrix}
			0\\1\\0
		\end{pmatrix}+
		\mathbb{C}
		\begin{pmatrix}
			0\\0\\1
		\end{pmatrix},
		\]
		\[
		(l_a)_{2,2}=(l_b)_{2,2}=\mathbb{C}
		\begin{pmatrix}
			0\\0\\1
		\end{pmatrix}, \; 
		(l_a)_{2,1}=(l_b)_{2,1}=\mathbb{C}
		\begin{pmatrix}
			0\\1\\0
		\end{pmatrix}+
		\mathbb{C}
		\begin{pmatrix}
			0\\0\\1
		\end{pmatrix},
		\]
		\[
		(l_a)_{3,2}=\mathbb{C}
		\begin{pmatrix}
			a+1\\1\\1
		\end{pmatrix}, \;
		(l_a)_{3,1}=\mathbb{C}
		\begin{pmatrix}
			a\\1\\0
		\end{pmatrix}+
		\mathbb{C}
		\begin{pmatrix}
			1\\0\\1
		\end{pmatrix}, 
		(l_b)_{3,2}=\mathbb{C}
		\begin{pmatrix}
			1+b\\1\\1
		\end{pmatrix}, \;
		(l_b)_{3,1}=\mathbb{C}
		\begin{pmatrix}
			1\\1\\0
		\end{pmatrix}+
		\mathbb{C}
		\begin{pmatrix}
			b\\0\\1
		\end{pmatrix}. 
		\]
		Then $(U_0, a)$ and $(U_\infty, b)$ define coordinates on $P^w(-2)$, and we have $a=1/b$ when $a,b\neq 0$. Put
		\[
		\nabla_0(a):=d+
		\begin{pmatrix}
			c_{11}(z)&c^0_{12}(a)(z-t_1)(z-t_2)&c^0_{13}(a)(z-t_1)(z-t_2)\\ 
			0&(z-t_1)(z-t_2)+c_{22}(z)&c^0_{23}(t_3-t_1)(z-t_2)\\ 
			c^0_{31}h'(t_3)&c^0_{32}(a)(t_3-t_2)(z-t_1)&(z-t_1)(z-t_2)+c_{33}(z)
		\end{pmatrix}\frac{dz}{h(z)}.
		\]
		\[
		\Phi_0(a):=
		\begin{pmatrix}
			0&a(a+1)(z-t_1)(z-t_2)&-a(a+1)(z-t_1)(z-t_2) \\
			h'(t_3)&0&-(a+1)(t_3-t_1)(z-t_2) \\
			-ah'(t_3)&a(a+1)(t_3-t_2)(z-t_1)&0
		\end{pmatrix}
		\frac{dz}{h(z)},
		\]
		\[
		\nabla_\infty(b):=d+
		\begin{pmatrix}
			c_{11}(z)&c^\infty_{12}(b)(z-t_1)(z-t_2)&c^\infty_{13}(b)(z-t_1)(z-t_2)\\ 
			c^\infty_{21}h'(t_3)&(z-t_1)(z-t_2)+c_{22}(z)&c^\infty_{23}(b)(t_3-t_1)(z-t_2)\\ 
			0&c^\infty_{32}(t_3-t_2)(z-t_1)&(z-t_1)(z-t_2)+c_{33}(z)
		\end{pmatrix}\frac{dz}{h(z)},
		\]
		\[
		\Phi_\infty(b):=
		\begin{pmatrix}
			0&b(1+b)(z-t_1)(z-t_2)&-b(1+b)(z-t_1)(z-t_2) \\
			bh'(t_3)&0&-b(1+b)(t_3-t_1)(z-t_2) \\
			-h'(t_3)&(1+b)(t_3-t_2)(z-t_1)&0
		\end{pmatrix}
		\frac{dz}{h(z)}, 
		\]
		where 
		\begin{align*}
			c_{11}(z)&=\nu_{2,0}(t_2-t_3)(z-t_1)+\nu_{1,0}(t_1-t_3)(z-t_2),\\
			c_{22}(z)&=\nu_{2,1}(t_2-t_3)(z-t_1)+\nu_{1,2}(t_1-t_3)(z-t_2),\\ c_{33}(z)&=\nu_{2,2}(t_2-t_3)(z-t_1)+\nu_{1,1}(t_1-t_3)(z-t_2), \\
			c^0_{12}(a)&=a(1+\nu_{1,0}+\nu_{2,0}-\nu_{1,2}-\nu_{2,1})+(1-(\nu_{1,2}+\nu_{2,1}+\nu_{3,1})), \\
			c^0_{13}(a)&=a((\nu_{1,2}+\nu_{2,1}+\nu_{3,2})-1)+(1-(\nu_{1,1}+\nu_{2,2}+\nu_{3,0})), \\
			c^0_{32}(a)&=(\nu_{1,1}+\nu_{2,2}+\nu_{3,2})-1+(a+1)(\nu_{1,0}+\nu_{2,0}+\nu_{3,0}), \\
			c^\infty_{21}&=-(\nu_{1,0}+\nu_{2,0}+\nu_{3,0}),  c^0_{23}=(\nu_{1,2}+\nu_{2,1}+\nu_{3,2})-1, \\
			c^\infty_{12}(b)&=(1-\nu_{1,2}-\nu_{2,1}-\nu_{3,0})+b((\nu_{1,1}+\nu_{2,2}+\nu_{3,2})-1), \\
			c^\infty_{13}(b)&=(1-\nu_{1,1}-\nu_{2,2}-\nu_{3,1})+b(1+\nu_{1,0}+\nu_{2,0}-\nu_{1,1}-\nu_{2,2}), \\
			c^\infty_{23}(b)&=(\nu_{1,2}+\nu_{2,1}+\nu_{3,2})-1+(1+b)(\nu_{1,0}+\nu_{2,0}+\nu_{3,0}), \\
			c^0_{31}&=-(\nu_{1,0}+\nu_{2,0}+\nu_{3,0}),   c^\infty_{32}=(\nu_{1,1}+\nu_{2,2}+\nu_{3,2})-1.
		\end{align*}
		Then we have
		\[
		\Bun^{-1}(U_0)\cong \mathbb{P}(\mathbb{C}\nabla_0\oplus\mathbb{C}\Phi_0), \; \Bun^{-1}(U_\infty)\cong \mathbb{P}(\mathbb{C}\nabla_\infty\oplus\mathbb{C}\Phi_\infty),
		\]
		where $\Bun\colon \overline{M^w_3(\bm{t},\boldsymbol{\nu})^0}\rightarrow  P^w(-2)$ is the forgetful map. We can see that 
		\[
		\nabla_\infty=P^{-1}
		(\nabla_0-(\nu_{1,0}+\nu_{2,0}+\nu_{3,0})a^{-1}\Phi_0)P, \;
		\Phi_\infty=
		P^{-1}
		(a^{-2}\Phi_0)P, 
		\]
		where $P=\textrm{diag}(a,1,1)$, and so we have
		\[
		(\nabla_\infty, \Phi_\infty)\cong (\nabla_0, \Phi_0)
		\begin{pmatrix}
			1&0\\
			-(\nu_{1,0}+\nu_{2,0}+\nu_{3,0})a^{-1}&a^{-2}
		\end{pmatrix}.
		\]
		Hence we obtain the theorem.
	\end{proof}
	
	\subsection{Comparing two compactifications}
	Let us consider the relation between the moduli space of $\boldsymbol{\nu}$-parabolic $\phi$-connections $\overline{M^{\boldsymbol{\alpha}}_3}(\bm{t},\boldsymbol{\nu})$ and the moduli space of $\lambda\boldsymbol{\nu}$-parabolic $\lambda$-connections $\overline{M^w_3(\bm{t},\boldsymbol{\nu})^0}$. We assume that $\nu_{i,0}\neq \nu_{i,1}\neq \nu_{i,2}\neq \nu_{i,0}$ for each $i$ for simplicity. Let $\varphi \colon \widehat{M^{\boldsymbol{\alpha}}_3}(\bm{t},\boldsymbol{\nu})\rightarrow \mathbb{P}(\Omega_{\pl}^1(D(\bm{t}))\oplus\Opl), \varphi'\colon  \overline{M^{\boldsymbol{\alpha}}_3}(\bm{t},\boldsymbol{\nu})\rightarrow \mathbb{P}^2$ and $\rho\colon \mathbb{P}(\Omega_{\pl}^1(D(\bm{t}))\oplus\Opl)\rightarrow \mathbb{P}^2$ be the morphism defined in Section \ref{secdes} (see the diagram (\ref{diagmtpf}) in the proof of Theorem \ref{maintheorem}). Let $D_i\subset \mathbb{P}(\Omega_{\pl}^1(D(\bm{t}))\oplus\Opl)$ be the fiber over $t_i$ and $\hat{D}_i$ be the strict transform of $D_i$ under $\varphi$. Let $H_i=\rho(D_i)$ and $\hat{H}_i$ be the strict transform of $H_i$ under $\varphi'$. Let $D_0$ be the section of $\mathbb{P}(\Omega^1_{\pl}(D(\bm{t}))\oplus \mcO_{\pl})$ over $\mathbb{P}^1$ defined by the injection $\Omega_{\pl}^1(D(\bm{t}))\hookrightarrow \Omega^1_{\pl}(D(\bm{t}))\oplus \mcO_{\pl}$. Let $b_{i,j} \in \mathbb{P}(\Omega_{\pl}^1(D(\bm{t}))\oplus\Opl)$ be the point defined in the subsection \ref{pfsec} and put $c_{i,j}=\rho(b_{i,j})\in \mathbb{P}^2$. We can see that three points $c_{1,i}, c_{2,j}, c_{3,k}$ are on the same line if and only if $\nu_{1,i}+\nu_{2,j}+\nu_{3,k}=1$, and six points $c_{1,i_1}, c_{1,i_2}, c_{2,j_1}, c_{2,j_2}, c_{3,k_1}, c_{3,k_2}$ are on the same conic if and only if $\nu_{1,i_1}+\nu_{1,i_2}+\nu_{2,j_1}+\nu_{2,j_2}+\nu_{3,k_1}+\nu_{3,k_2}=2$. 
	
	The following proposition follows from the proof of Proposition \ref{isominusblow} and Proposition \ref{allneq}.
	\begin{proposition}
		Assume that $0<\alpha_{i,j}\ll 1$ and $\nu_{i,0}\neq \nu_{i,1}\neq \nu_{i,2}\neq \nu_{i,0}$ for each $i$. Take $(E,\nabla,l_*) \in M^{\boldsymbol{\alpha}}_3(\bm{t},\boldsymbol{\nu})$. Then the type of $(E,l_*)$ is one of the following:
		\begin{itemize}
			\item[(i)] $E\cong \Opl \oplus \Opl(-1)\oplus \Opl(-1)$, $\#\{i\mid \Opl|_{t_i}\subset l^{(i)}_1\}=0$, $n(l'_*)=1$, and the condition $(*)$ holds.
			\item[(i)$'$]  $E\cong \Opl \oplus \Opl(-1)\oplus \Opl(-1)$, $\#\{i\mid \Opl|_{t_i}\subset l^{(i)}_1\}=0$, $n(l'_*)=1$, and the condition $(*)$ does not hold.
			\item[(ii)]  $E\cong \Opl \oplus \Opl(-1)\oplus \Opl(-1)$, $\#\{i\mid \Opl|_{t_i}\subset l^{(i)}_1\}=1$, $n(l'_*)=1$, and the condition $(*)$ holds.
			\item[(iii)] 
			$E\cong \Opl \oplus \Opl(-1)\oplus \Opl(-1)$, $\#\{i\mid \Opl|_{t_i}\subset l^{(i)}_1\}=0$, $n(l'_*)\geq 2$, and the condition $(*)$ holds.
		\end{itemize}
		For $(E,l_*)$ whose type is (iii), $n(l'_*)= 3$ when $\nu_{1,2}+\nu_{2,2}+\nu_{3,2}=1$ and $n(l'_*)= 2$ when $\nu_{1,2}+\nu_{2,2}+\nu_{3,2}\neq 1$
	\end{proposition}
	
	Assume that $\boldsymbol{\nu}$ satisfies the condition 
	\begin{equation}\label{nucond2}
		\nu_{1,2}+\nu_{2,2}+\nu_{3,2}\neq 1
	\end{equation}
	and
	\begin{equation}\label{nucond01}
		\nu_{1,j_1}+\nu_{2,2}+\nu_{3,2}\neq 1, \;  \nu_{1,2}+\nu_{2,j_2}+\nu_{3,2}\neq 1, \;  \nu_{1,2}+\nu_{2,2}+\nu_{3,j_3}\neq 1
	\end{equation}
	for any $j_1, j_2, j_3=0, 1$. 
	When $2/9<w<1/3$, $P^w(-2)$ consists of parabolic bundles of the type (i) and (i)$'$. We can obtain $\overline{M^w_3(\bm{t},\boldsymbol{\nu})^0}$ from $\widehat{M^{\boldsymbol{\alpha}}_3}(\bm{t},\boldsymbol{\nu})$ by the following three steps.
	
	\begin{figure}
		\centering
		\includegraphics[width=11cm]{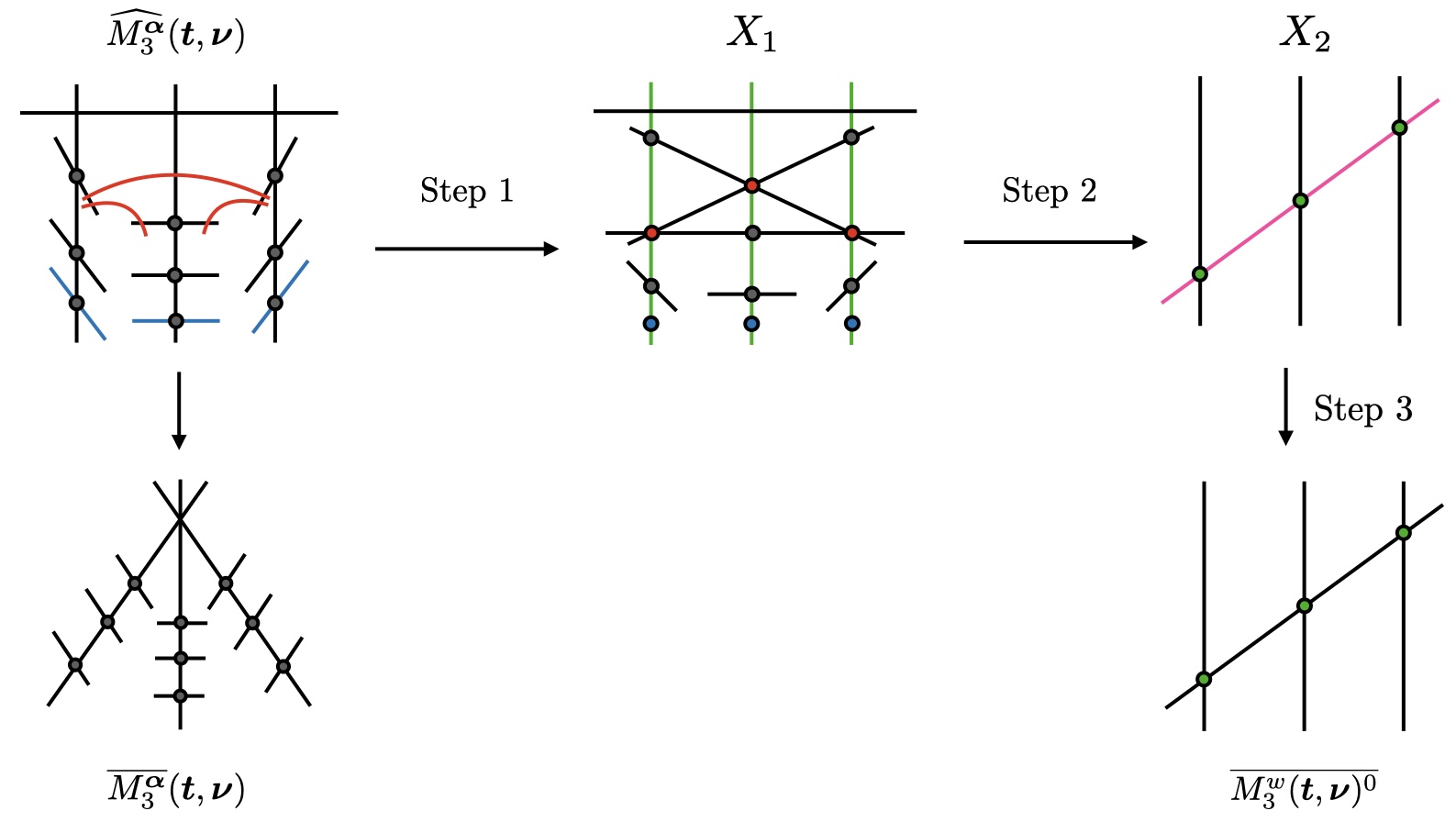}
	\end{figure}
	Step 1: contract the locus consisting of the type (ii) and (iii). We have
	\[
	\{(E,\nabla,l_*)\in M^{\boldsymbol{\alpha}}_3(\bm{t},\boldsymbol{\nu})\mid \text{the type of $(E,l_*)$ is (ii)}\}=(\varphi^{-1}(b_{1,0})\setminus D_1)\cup(\varphi^{-1}(b_{2,0})\setminus D_2)\cup(\varphi^{-1}(b_{3,0})\setminus D_3). 
	\]
	By Proposition \ref{allneq}, $\varphi^{-1}(b_{i,j})$ is a $(-1)$-curve. 
	From (\ref{rk3flag2}), the closure of the set
	\[
	\{(E,\nabla,l_*)\in M^{\boldsymbol{\alpha}}_3(\bm{t},\boldsymbol{\nu})\mid \text{$l'_i$ and $l'_j$ lie on  some subbundle $\Opl(-1)\cong F'\subset \Opl(-1)\oplus\Opl(-1)$}\}
	\]
	on $\overline{M^{\boldsymbol{\alpha}}_3}(\bm{t},\boldsymbol{\nu})$ is the closure of the locus defined by  
	\[
	(h'(t_i)(\nu_{i,2}-\res_{t_i}(\tfrac{dz}{z-t_3}))-p)
	(t_j-q)-(h'(t_j)(\nu_{j,2}-\res_{t_j}(\tfrac{dz}{z-t_3}))-p)
	(t_i-q)=0,
	\]
	where $(q,p)$ is the coordinate defined in the proof of Proposition \ref{isominusblow}, which is just the strict transform $\hat{L}_{ij}\subset \overline{M^{\boldsymbol{\alpha}}_3}(\bm{t},\boldsymbol{\nu})$ of the line $L_{ij}\subset \mathbb{P}^2$ passing through $c_{i,2}$ and $c_{j,2}$ under $\varphi'$. Since any $c_{m,n}$ for $(m,n)\neq (i,2), (j,2)$ is not on $L_{i,j}$ from the condition (\ref{nucond2}) and (\ref{nucond01}), the intersection number of $\hat{L}_{ij}$ is $-1$.  By contracting $\varphi^{-1}(b_{1,0}), \varphi^{-1}(b_{2,0}), \varphi^{-1}(b_{3,0})$ and the inverse images of 
	$\hat{L}_{12}, \hat{L}_{23}, \hat{L}_{13}$ under $\PC$, we obtain a morphism $\rho_1\colon \widehat{M^{\boldsymbol{\alpha}}_3}(\bm{t},\boldsymbol{\nu}) \rightarrow X_1$,  where $X_1$ is a smooth projective surface. 
	
	Step 2: contract the locus defined by $\rank \phi =2$. Since $\varphi\colon \widehat{M^{\boldsymbol{\alpha}}_3}(\bm{t},\boldsymbol{\nu}) \rightarrow \mathbb{P}(\Omega_{\pl}^1(D(\bm{t}))\oplus \Opl)$ is the blow-up at 9 points $\{b_{i,j}\}^{1\leq i\leq 3}_{0\leq j\leq 2}$, $\hat{D}_i$ is a $(-3)$-curve for each $i$. $\hat{H}_i$ intersects with $\varphi^{-1}(c_{i,0})$ and $\hat{L}_{jm}\;(j,m\neq i)$ at one point, respectively. So the image $\rho_1(\hat{D}_i)\subset X_1$ is a $(-1)$-curve. Contracting 
	$\hat{D}_1, \hat{D}_2, \hat{D}_3$, we obtain a morphism $\rho_2\colon X_1 \rightarrow X_2$. When $\nu_{1,0}+\nu_{2,0}+\nu_{3,0}=0$, there exists a conic $C\subset \mathbb{P}^2$ passing through six points $c_{1,1}, c_{1,2}, c_{2,1}, c_{2,2}, c_{3,1}, c_{3,2}$. Let $\hat{C}\subset \widehat{M^{\boldsymbol{\alpha}}_3}(\bm{t},\boldsymbol{\nu})$ be the strict transform of $C$ under $\rho\circ \varphi=\varphi'\circ \PC$. Then $\rho_1(\hat{C})\cong \rho_2(\rho_1(\hat{C}))$ is a projective line and intersects with $\rho_2(\rho_1(\varphi^{-1}(b_{i,1})))$ for each $i=1,2,3$. So $X_2$ is isomorphic to $\mathbb{P}(\Opl \oplus \Opl(-2))$. Since $C$ does not intersect with  $\varphi'^{-1}(c_{i,0})$, and $C$ intersects with each $\hat{H}_i$ and $\hat{L}_{mn}$ at two points, we have $\rho_2(\rho_1(\hat{C}))^2=\rho_1(\hat{C})^2=\hat{C}^2=-2$. $\rho_2(\rho_1(\hat{C}))$ is the unique section whose intersection number is $-2$. When $\nu_{1,0}+\nu_{2,0}+\nu_{3,0}\neq 0$, there is no projective line contained in $X_2$ which intersects with $\rho_2(\rho_1(\varphi^{-1}(b_{i,1})))$ for each $i=1,2,3$. So $X_2$ is isomorphic to $\pl \times \pl$.
	
	Step 3: change $D_0$ to $\mathbb{P}T^*P^w(-2)$. $D_0$ and $\mathbb{P}T^*P^w(-2)$ are infinitesimally close to each other. A $\boldsymbol{\nu}$-parabolic connection 
	\[
	\phi=
	\begin{pmatrix}
		1&0&0\\
		0&1&0\\
		0&0&1
	\end{pmatrix}, \;
	\nabla=d+
	\begin{pmatrix}
		0& a_{12}(z)& a_{13}(z)\\
		1& (z-t_1)(z-t_2)-p& 0\\
		0& z-q& (z-t_1)(z-t_2)+p
	\end{pmatrix}
	\frac{dz}{h(z)}
	\]
	whose apparent singularity $q$ is not $t_1,t_2$ and $t_3$
	has the limits
	\begin{equation}
		\begin{tikzcd}
			\small
			\begingroup 
			\setlength\arraycolsep{2pt}
			\begin{pmatrix}
				p^{-2}&0&0\\
				0&p^{-2}&0\\
				0&0&1
			\end{pmatrix}
			(\phi,\nabla)
			\begin{pmatrix}
				p^2&0&0\\
				0&1&0\\
				0&0&p^{-1}
			\end{pmatrix}
			\endgroup
			\ar[r,"p\rightarrow \infty"]&
			\footnotesize\left(
			\begingroup 
			\setlength\arraycolsep{2pt}
			\begin{pmatrix}
				1&0&0\\
				0&0&0\\
				0&0&0
			\end{pmatrix}, \;
			\begin{pmatrix}
				0&-1&g(z)\\
				1&0&0\\
				0&z-q&1
			\end{pmatrix}\frac{dz}{h(z)}
			\endgroup\right),
		\end{tikzcd}
	\end{equation}
	\begin{equation}
		\begin{tikzcd}
			\small
			\begingroup 
			\setlength\arraycolsep{2pt}
			\begin{pmatrix}
				1&0&0\\
				0&p&0\\
				0&0&p^2
			\end{pmatrix}
			(\phi,\nabla)
			\begin{pmatrix}
				p^{-1}&0&0\\
				0&p^{-2}&0\\
				0&0&p^{-3}
			\end{pmatrix}
			\endgroup
			\ar[r,"p\rightarrow \infty"]&
			\footnotesize\left(
			\begingroup 
			\setlength\arraycolsep{2pt}
			\begin{pmatrix}
				0&0&0\\
				0&0&0\\
				0&0&0
			\end{pmatrix}, \;
			\begin{pmatrix}
				0&-1&g(z)\\
				1&-1&0\\
				0&z-q&1
			\end{pmatrix}
			\endgroup
			\frac{dz}{h(z)}\right),
		\end{tikzcd}
	\end{equation}
	where $g(z)=\sum_{i=1}^{3}\frac{1}{(q-t_i)h'(t_i)}\prod_{j\neq i}(z-t_j)$. Put
	\[
	C(q;z):=
	\begin{pmatrix}
		\frac{\left(t_3-t_1\right)h'(t_3)}{\left(t_2-t_1\right)\left(q-t_1\right)\left(q-t_3\right)}&\frac{\left(t_3-t_2\right)\left(z+q-t_1-t_2\right)}{\left(t_1-t_2\right)\left(q-t_2\right)}&\frac{\left(t_3-t_1\right)\left(z+q-t_1-t_2\right)}{\left(t_2-t_1\right)\left(q-t_1\right)}\\
		0&\frac{t_3-t_2}{t_1-t_2}&\frac{t_3-t_1}{t_2-t_1}\\
		0&\frac{\left(t_3-t_2\right)\left(q-t_1\right)}{t_1-t_2}&\frac{\left(t_3-t_1\right)\left(q-t_2\right)}{t_2-t_1}
	\end{pmatrix},
	\]
	\[
	C_1(q;z):=
	\begin{pmatrix}
		-(q-t_2)(q-t_3)&0&z+q-t_2-t_3\\
		0&-(q-t_2)(q-t_3)&0\\
		0&0&1
	\end{pmatrix},
	\]
	\[
	C_2(q;z):=
	\begin{pmatrix}
		-(q-t_2)^{-1}(q-t_3)^{-1}&0&0\\
		0&1&1\\
		0&0&q-t_1
	\end{pmatrix}.
	\]
	Then we have
	\begin{align*}
		\small\hspace{-14pt}
		C_1(q;z)\left(\begin{pmatrix}
			1&0&0\\
			0&0&0\\
			0&0&0
		\end{pmatrix}, \;
		\begin{pmatrix}
			0&-1&g(z)\\
			1&0&0\\
			0&z-q&1
		\end{pmatrix}\frac{dz}{h(z)}\right)C_2(q;z)
		=\left(
		\begin{pmatrix}
			1&0&0 \\
			0&0&0 \\
			0&0&0
		\end{pmatrix},
		\begin{pmatrix}
			0&(z-t_2)(z-t_3)&0 \\
			1&0&0\\
			0&z-q&z-t_1
		\end{pmatrix}
		\frac{dz}{h(z)}\right),
	\end{align*}
	and
	\[
	C(q;z)^{-1}
	\begin{pmatrix}
		0&-1&g(z)\\
		1&-1&0\\
		0&z-q&1
	\end{pmatrix}
	\frac{dz}{h(z)}C(q;z)=
	\frac{(t_3-t_1)(q-t_2)}{h'(t_2)(q-t_1)(q-t_3)}\Phi_0(-\tfrac{(t_3-t_2)(q-t_1)}{(t_3-t_1)(q-t_2)}).
	\]
	So a $\boldsymbol{\nu}$-parabolic $\phi$-connection with $\rank \phi=1$ and a parabolic Higgs bundle is infinitesimally closed to each other. 
	In the case of  $q=t_1,t_2,t_3$, we can also see it by using (\ref{allneqform}) and (\ref{rk2coord}). Therefore we can obtain $\overline{M^w_3(\bm{t},\boldsymbol{\nu})^0}$ from $\widehat{M^{\boldsymbol{\alpha}}_3}(\bm{t},\boldsymbol{\nu})$.
	
	\subsection{Parabolic bundles and apparent singularities}\label{pbapp}
	We fix $2/9<w<1/3$. Let $V_0 \subset P^w(-2)$ be the subset consisting of parabolic bundles of the type (i). The set $V_0$ is the set of $P^w(-2)$ minus 3 points by Proposition \ref{pbmod}. Let $(E,l_*)\in V_0$ and $\nabla$ be a $\lambda \boldsymbol{\nu}$-logarithmic $\lambda$-connection on $(E,l_*)$. Assume that $\nu_{1,0}+\nu_{2,0}+\nu_{3,0}\neq 0$. Then there exists a unique filtration $E=:F_0\supset F_1\supset F_2\supset 0$ such that $F_2\cong \Opl, F_1\cong \Opl\oplus \Opl(-1)$, and $\nabla(F_2)\subset F_1 \otimes \Omega^1_{\pl}(D(\bm{t}))$.  We define the apparent singularity $\App(E,\nabla,l_*)$ by the zero of the nonzero homomorphism
	\[
	\Opl(-1)\cong F_1/F_2\overset{\nabla}{\rightarrow} (E/F_1)\otimes \Omega_{\pl}^1(D(\bm{t})) \cong \Opl.
	\]
	When $\lambda\neq 0$, this definition is the same as the definition in subsection \ref{appsec}. 
	
	\begin{remark}
		Assume that $(E,l_*)\in P^w(-2)\setminus V_0$. Then for any parabolic connection $\nabla$ over $(E,l_*)$,  there exists a unique filtration $E=F_0\supset F_1\supset F_2\supset 0$ such that $F_2\cong \Opl, F_1\cong \Opl\oplus \Opl(-1)$, and $\nabla(F_2)\subset F_1 \otimes \Omega^1_{\pl}(D(\bm{t}))$. However, we can see that for a parabolic Higgs field $\Phi$ over $(E,l_*)$, such filtration is not unique. So we can not define the apparent map $\App$ over $\overline{M^{w}_3(\bm{t},\boldsymbol{\nu})^0}$.
	\end{remark}
	The following is a version of Theorem 4.3 in \cite{LS} in the present setting. 
	
	\begin{proposition}\label{appbun}
		We fix $2/9<w<1/3$ and assume that $\nu_{1,0}+\nu_{2,0}+\nu_{3,0}\neq 0$. Then the morphism
		\[
		\App \times \Bun \colon \Bun^{-1}(V_0)\longrightarrow \pl\times V_0
		\]
		is finite and its generic fiber consists of three points. 
	\end{proposition}
	\begin{proof}
		Consider fibers of $\App \times \Bun$. We have
		\[
		(\mu\nabla_0+\lambda\Phi_0)
		\begin{pmatrix}
			1\\0\\0
		\end{pmatrix}=
		\begin{pmatrix}
			\mu c_{11}(z)\\ \lambda h'(t_3)\\ (\mu c^0_{31}-\lambda a)h'(t_3)
		\end{pmatrix}\frac{dz}{h(z)}.
		\]
		So $F_1$ is generated by the sections $^t(1, 0, 0)$ and $^t(0, \lambda, (\mu c^0_{31}-\lambda a))$. Since
		\begin{align*}
			&(\mu \nabla_0+\lambda\Phi_0)
			\begin{pmatrix}
				0\\ \lambda\\\mu c^0_{31}-\lambda a
			\end{pmatrix}\\
			=&
			\begin{pmatrix}
				*\\
				\mu \lambda((z-t_1)(z-t_2)+c_{22}(z))+(\mu c^0_{31}-\lambda a)(\mu c^0_{23}-\lambda(a+1))(t_3-t_1)(z-t_2)\\
				\lambda(\mu c^0_{32}(a)+\lambda a(a+1))(t_3-t_2)(z-t_1)+\mu (\mu c^0_{31}-\lambda a)((z-t_1)(z-t_2)+c_{33}(z))
			\end{pmatrix}, 
		\end{align*}
		the apparent singularity of $\mu\nabla_0+\lambda\Phi_0$ is the zero of the polynomial
		\begin{align*}
			&\lambda\{\lambda(\mu c^0_{32}(a)+\lambda a(a+1))(t_3-t_2)(z-t_1)+\mu (\mu c^0_{31}-\lambda a)((z-t_1)(z-t_2)+c_{33}(z))\}\\
			&-(\mu c^0_{31}-\lambda a)\{\mu \lambda((z-t_1)(z-t_2)+c_{22}(z))+(\mu c^0_{31}-\lambda a)(\mu c^0_{23}-\lambda(a+1))(t_3-t_1)(z-t_2)\}\\
			=&f_1(a;\mu,\lambda)(z-t_1)+f_2(a;\mu,\lambda)(z-t_2),
		\end{align*}
		where
		\begin{align*}
			f_1(a;\mu,\lambda)=&(t_3-t_2)\{a(a+1)\lambda^3+(c^0_{32}(a)+(\nu_{2,2}-\nu_{2,1})a)\lambda^2\mu -(\nu_{2,2}-\nu_{2,1})c^0_{31}\mu^2 \lambda \},\\
			f_2(a;\mu,\lambda)=&(t_3-t_1)\{a^2(a+1)\lambda^3-((\nu_{1,2}-\nu_{1,1})a+2a(a+1)c^0_{31}+a^2c^0_{32}(a))\lambda^2\mu\\
			&+((\nu_{1,2}-\nu_{1,1})c^0_{31}+2ac^0_{31}c^0_{23}+(a+1)(c^0_{31})^2)\lambda\mu^2-(c^0_{31})^2c^0_{23}\mu^3\}.
		\end{align*}
		Hence $\App\colon \Bun^{-1}((E,(l_a)_*))\cong \mathbb{P}(\mathbb{C}\nabla_0(a)\oplus \mathbb{C}\Phi_0(a)) \rightarrow \pl$ is defined by 
		\[
		\App(\mu\nabla_0+\lambda\Phi_0)=(f_1(a;\mu,\lambda)+f_2(a;\mu,\lambda):t_1f_1(a;\mu,\lambda)+t_2f_2(a;\mu,\lambda)),
		\]
		which implies that a generic fiber consists of three points. Since $\App \times\Bun$ is proper, $\App \times\Bun$ is finite.
	\end{proof}
	
	\appendix
	\section{Computation on the stability}
	\begin{proposition}\label{prostab}
		All points of $R^s$ are properly stable with respect to the action of $G$ and the $G$-linearized $S$-ample line bundle $L^{\otimes N}$.
	\end{proposition}
	\begin{proof}
		Take any geometric point $x$ of $R^s$. Let $y$ be the induced geometric point of $S$. We prove that $x$ is a properly stable point of the fiber $R^s_y$ with respect to the action of $G_y$ and the polarization $L^{\otimes N}$. So we may assume that $S=\Spec K$ with $K$ is an algebraically closed field. We put
		\[
		(E_1,E_2,\Phi,F_*(E_1),F_*(E_2)):=((\mcE_1)_x,(\mcE_2)_x,\tilde{\Phi}_x,F_*(\mcE_1)_x, F_*(\mcE_2)_x))
		\]
		For simplicity, we write the same character $V_1,V_2, W_1,W_2$ to denote $(V_1)_y, (V_2)_y,(W_1)_y,(W_2)_y$, respectively. Let
		\[
		\pi_2 \colon V_1 \otimes W_1 \oplus V_2 \otimes W_2 \rightarrow N_2, \ \pi_1 \colon V_1 \otimes W_2 \rightarrow N_1, \ \pi_{1,i} \colon V_1 \rightarrow N^{(1)}_i, \ \pi_{2,i} \colon V_2 \rightarrow N^{(2)}_i
		\]
		be the quotients of vector spaces corresponding to $\iota(x)$. We will show that $\iota(x)$ is a properly stable point with respect to the action of $G$ and the linearization of $L^{\otimes N}$. Consider the character
		\[
		\chi \colon GL(V_1) \times GL(V_2) \longrightarrow \bm{G}_m; \ (g_1,g_2)\mapsto \det(g_1)\det(g_2).
		\]
		Since the natural composite $\ker \chi \rightarrow GL(V_1) \times GL(V_2)\rightarrow G$ is an isogeny, by Theorem 2.1 \cite{Mu} it is sufficient to show that $\mu^{L^{\otimes N}} (x,\lambda)>0$ for any nontrivial homomorphism $\lambda\colon \bm{G}_m\rightarrow \ker \chi$ ,where $\mu^{L^{\otimes N}} (x,\lambda)$ is defined in Definition 2.2 \cite{Mu}. Let  $\lambda\colon \bm{G}_m\rightarrow \ker \chi$ be a nontrivial homomorphism.
		For  a suitable basis $e^{(1)}_1,\ldots, e^{(1)}_{n_1}$ (resp. $e^{(1)}_1,\ldots, e^{(2)}_{n_2}$), the action of $\lambda$ on $V_1$ (resp. $V_2$) is represented by
		\[
		e^{(1)}_i \mapsto t^{u^{(1)}_i}e^{(1)}_i\ (\text{resp}. \ e^{(2)}_i \mapsto t^{u^{(2)}_i}e^{(2)}_i)\ \ (t \in \bm{G}_m), 
		\]
		where $u^{(1)}_1 \leq \cdots \leq u^{(1)}_{n_1}$ (resp. $u^{(2)}_1 \leq \cdots \leq u^{(2)}_{n_2}$). Then we have $\sum_{i=1}^{n_1}u^{(1)}_i + \sum_{i=1}^{n_2}u^{(2)}_i=0$. Let $f^{(k)}_1,\ldots, f^{(k)}_{b_k}$ be a basis of $W_k$ for each $k=1,2$.
		
		For $q=0,1,\ldots, n_1+n_2$, we define functions  $a_1(q),a_2(q)$ as follows. First, we set $(a_1(q),a_2(q))=(0,0)$ and put
		\[
		(a_1(1),a_2(1))=\left\{
		\begin{array}{ll}
			(1,0) &\text{if} \ u^{(1)}_1\leq u^{(2)}_1 \\
			(0,1) & \text{if} \ u^{(1)}_1> u^{(2)}_1
		\end{array}.
		\right.
		\]
		We inductively define
		\[
		(a_1(q+1),a_2(q+1))=\left\{
		\begin{array}{ll}
			(a_1(q)+1,a_2(q))& \text{if}\; u^{(1)}_{a_1(q)+1}\leq u^{(2)}_{a_2(q)+1},\;  a_1(q)<n_1, \text{and}\; a_2(q)<n_2\\
			(a_1(q),a_2(q)+1)& \text{if}\; u^{(1)}_{a_1(q)+1}> u^{(2)}_{a_2(q)+1}\; a_1(q)<n_1, \text{and}\; a_2(q)<n_2 \\
			(a_1(q)+1,a_2(q))& \text{if} \; a_2(q)=n_2 \\
			(a_1(q),a_2(q)+1) & \text{if} \; a_1(q)=n_1
		\end{array}.
		\right.
		\]
		Then $a_1(q)$ and $a_2(q)$ are integers satisfying $0 \leq a_1(q)\leq n_1$, $0 \leq a_2(q)\leq n_2$, $a_1(q)\leq a_1(q+1)$, $a_2(q)\leq a_2(q+1)$ and $a_1(q)+a_2(q)=q$. We define $v_1,\ldots, v_{n_1+n_2}$ by
		\[
		v_q=\left\{
		\begin{array}{ll}
			u^{(1)}_{a_1(q)} &\text{if} \ (a_1(q), a_2(q))=(a_1(q-1)+1, a_2(q-1)) \\
			u^{(2)}_{a_2(q)} & \text{if} \ (a_1(q), a_2(q))=(a_1(q-1), a_2(q-1)+1)
		\end{array}.
		\right.
		\]
		For $p=1, \ldots,b_1n_1+b_2n_2$, we can find a unique integer $q\in \{1,\ldots,n_1+n_2\}$ such that 
		\[
		p=\left\{
		\begin{array}{ll}
			(a_1(q)-1)b_1+a_2(q)b_2+j &\text{for some}\ 1 \leq j\leq b_1\  \text{if} \  (a_1(q), a_2(q))=(a_1(q-1)+1, a_2(q-1))\\
			a_1(q)b_1+(a_2(q)-1)b_2+j & \text{for some}\ 1 \leq j\leq b_2\  \text{if} \ (a_1(q), a_2(q))=(a_1(q-1), a_2(q-1)+1)
		\end{array}.
		\right.
		\]
		For each $p$, we put $s^{(2)}_p:=v_q$ and
		\[
		h_p:=\left\{
		\begin{array}{ll}
			e^{(1)}_{a_1(q)}\otimes f^{(1)}_j &\text{if} \ (a_1(q), a_2(q))=(a_1(q-1)+1, a_2(q-1)) \\
			e^{(2)}_{a_2(q)}\otimes f^{(2)}_j  & \text{if} \ (a_1(q), a_2(q))=(a_1(q-1), a_2(q-1)+1)
		\end{array}.
		\right.
		\]
		Put $\delta_p:=(v_{q+1}-v_q)(n_1+n_2)^{-1}$. Then we have 
		\begin{equation}\label{vsum}
			v_{n_1+n_2}=\sum_{q=1}^{n_1+n_2-1}q\delta_q, 
		\end{equation}
		\begin{equation}\label{u1sum}
			u^{(1)}_{n_1}=\sum_{\underset{a_1(q)<n_1}{1\leq q \leq n_1+n_2-1}}q\delta_q+\sum_{\underset{a_1(q)= n_1}{1\leq q \leq n_1+n_2-1}}(q-n_1-n_2)\delta_q,
		\end{equation}
		and
		\begin{equation}\label{u2sum}
			u^{(2)}_{n_2}=\sum_{\underset{a_2(q)<n_2}{1\leq q \leq n_1+n_2-1}}q\delta_q+\sum_{\underset{a_2(q)= n_2}{1\leq q \leq n_1+n_2-1}}(q-n_1-n_2)\delta_q.
		\end{equation}
		Let $U^{(2)}_p$ be the vector subspace of $V_1 \otimes W_1 \oplus V_2 \otimes W_2$ generated by $h_1,\ldots,h_p$. For $i=1,\ldots, r_2$, we can find an integer $p^{(2)}_i \in \{1,\ldots,b_1n_1+b_2n_2\}$ such that $\dim \pi_2(U^{(2)}_{p^{(2)}_i})=i$ and $\dim \pi_2(U^{(2)}_{p^{(2)}_i-1})=i-1$. Then
		\begin{align*}
			\sum_{i=1}^{r_2}s^{(2)}_{p^{(2)}_i}
			&=\sum_{i=1}^{r_2}s^{(2)}_{p^{(2)}_i}\left(\dim \pi_2(U^{(2)}_{p^{(2)}_i})-\dim \pi_2(U^{(2)}_{p^{(2)}_i-1})\right)\\
			&=\sum_{p=1}^{b_1n_1+b_2n_2}s^{(2)}_p\left(\dim \pi_2(U^{(2)}_p)-\dim \pi_2(U^{(2)}_{p-1})\right)\\
			&=r_2s^{(2)}_{b_1n_1+b_2n_2}-\sum_{p=1}^{b_1n_1+b_2n_2-1}(s^{(2)}_{p+1}-s^{(2)}_p)\dim \pi_2(U^{(2)}_p)\\
			&=r_2v_{n_1+n_2}-\sum_{q=1}^{n_1+n_2-1}(v_{q+1}-v_q)\dim \pi_2(U^{(2)}_{b_1a_1(q)+b_2a_2(q)})\\
			&\overset{(\ref{vsum})}{=}\sum_{q=1}^{n_1+n_2-1}\left(r_2q-(n_1+n_2)\dim \pi_2(U^{(2)}_{b_1a_1(q)+b_2a_2(q))}\right)\delta_q.
		\end{align*}
		For $p=(i-1)b_2+j\;(1 \leq i\leq n_1, 1 \leq j \leq b_2)$, we put  $s^{(1)}_p=u^{(1)}_i$ and $h'_p=e^{(1)}_i\otimes f^{(2)}_j$. Let $U^{(1)}_p$ be the subspace of $V_1\otimes W_2$ generated by $h'_1,\ldots, h'_p$. For $i=1, \ldots, r_1$,   we can find an integer $p^{(1)}_i \in \{1,\ldots,b_2n_1\}$ such that $\dim \pi_1(U^{(1)}_{p^{(1)}_i})=i$ and $\dim \pi_1(U^{(1)}_{p^{(1)}_i-1})=i-1$. Then we have 
		\begin{align*}
			\sum_{i=1}^{r_1}s^{(1)}_{p^{(1)}_i}
			=&\sum_{q=1}^{n_1+n_2-1}\left(r_1q-(n_1+n_2)\dim \pi_1(U^{(1)}_{a_1(q)b_2})\right)\delta_q
		\end{align*}
		by using (\ref{u1sum}). Let $V^{(1)}_p$ be the subspace of $V_1$ generated by $e^{(1)}_1,\ldots, e^{(1)}_p$. For $i=1,\ldots,l_1$ and for $j=1, \ldots, d^{(1)}_i$, let $p^{(1)}_{i,j}$ be the integer such that $\dim \pi_{1,i} (V^{(1)}_{p^{(1)}_{i,j}})=j$ and $\dim \pi_{1,i} (V^{(1)}_{p^{(1)}_{i,j}-1})=j-1$. Then we have
		\begin{align*}
			\sum_{j=1}^{d^{(1)}_i}u^{(1)}_{p^{(1)}_{i,j}}
			=&\sum_{q=1}^{n_1+n_2-1}\left(d^{(1)}_iq-(n_1+n_2)\dim \pi_{1,i}(V^{(1)}_{a_1(q)})\right)\delta_q
		\end{align*}
		by using (\ref{u1sum}). Let $V^{(2)}_p$ be the subspace of $V_2$ generated by $e^{(2)}_1,\ldots, e^{(2)}_p$. For $i=1,\ldots,l_2$, and for $j=1, \ldots, d^{(2)}_i$, let $p^{(2)}_{i,j}$ be the integer such that $\dim \pi_{2,i} (V^{(2)}_{p^{(2)}_{i,j}})=j$ and $\dim \pi_{2,i} (V^{(2)}_{p^{(2)}_{i,j}-1})=j-1$. Then we have
		\begin{align*}
			\sum_{j=1}^{d^{(2)}_i}u^{(2)}_{p^{(2)}_{i,j}}
			=&\sum_{q=1}^{n_1+n_2-1}\left(d^{(2)}_iq-(n_1+n_2)\dim \pi_{2,i}(V^{(2)}_{a_2(q)})\right)\delta_q
		\end{align*}
		by using (\ref{u2sum}). So we have
		\begin{align*}
			\mu^{L^{\otimes N}} (x,\lambda)
			=&-\left(\xi\sum_{i=1}^{r_1}s^{(k)}_{p^{(k)}_i}+\sum_{i=1}^{l_1}\xi^{(1)}_i\sum_{j=1}^{d^{(1)}_i}u^{(1)}_{p^{(1)}_{i,j}}+\sum_{i=1}^{l_2}\xi^{(2)}_i\sum_{j=1}^{d^{(2)}_i}u^{(2)}_{p^{(2)}_{i,j}}\right)N \\
			=&-\sum_{q=1}^{n_1+n_2-1}N\delta_q\left\{q\sum_{i=1}^{l_1}\xi^{(1)}_id^{(1)}_i+q\sum_{i=1}^{l_2}\xi^{(2)}_id^{(2)}_i-(n_1+n_2)\sum_{i=1}^{l_1}\xi^{(1)}_i\dim \pi^{(
				1)}_i(V^{(1)}_{a_1(q)})\right.\\
			&-(n_1+n_2)\sum_{i=1}^{l_2}\xi^{(2)}_i\dim \pi^{(
				2)}_i(V^{(2)}_{a_2(q)})+(r_1+r_2)q\xi \\
			&-(n_1+n_2)\xi \left(\dim \pi_1(U^{(1)}_{a_1(q)b_2}) +\dim \pi_2(U^{(2)}_{b_1a_1(q)+b_2a_2(q)}) \right) \Biggr\}.
		\end{align*}
		Hence $x$ is a properly stable point if 
		\begin{align*}
			&-q\sum_{i=1}^{l_1}\xi^{(1)}_id^{(1)}_{i+1}-q\sum_{i=1}^{l_2}\xi^{(2)}_id^{(2)}_{i+1}+(n_1+n_2)\sum_{i=1}^{l_1}\xi^{(1)}_i\dim \pi_{1,i}(V^{(1)}_{a_1(q)})+(n_1+n_2)\sum_{i=1}^{l_2}\xi^{(2)}_i\dim \pi_{2,i}(V^{(2)}_{a_2(q)}) \\
			&\quad-q\xi(r_1+r_2)+\xi(n_1+n_2)\left( \dim \pi_1(U^{(1)}_{a_1(q)b_2})+\dim \pi_2(U^{(2)}_{b_1a_1(q)+b_2a_2(q)})\right)>0
		\end{align*}
		for all $q=1,\ldots,n_1+m_2-1$.
		
		For each $q=1,\ldots,n_1+n_2-1$, let $V'_k$ be the vector subspace of $V_k$ generated by $e^{(k)}_1,\ldots,e^{(k)}_{a_k(q)}$ for $k=1,2$. We note that 
		\begin{equation}\label{qeq}
			q=\dim V'_1+\dim V'_2.
		\end{equation}
		Then $U^{(1)}_{a_1(q)b_2}=V'_1 \otimes W_2$ and $U^{(2)}_{b_1a_1(q)+b_2a_2(q)}=V'_1 \otimes W_1 \oplus V'_2 \otimes W_2$. Put
		\[
		E'_1:=\Image (V'_1 \otimes \mcO_{X_y}(-m_0) \rightarrow E_1),\;
		E'_2:=\Image (\Lambda^1_{D_y}\otimes V'_1 \otimes \mcO_{X_y}(-m_0) \oplus V'_2 \otimes \mcO_{X_y}(-m_0+\gamma) \rightarrow E_2).
		\]
		By the choice of $m_1$, we have 
		\begin{equation}\label{image}
			\pi_2(U^{(2)}_{b_1a_1(q)+b_2a_2(q)})=H^0(E'_2(m_0+m_1-\gamma)),\quad \pi_1(U^{(1)}_{a_1(q)b_2})=H^0(E'_1(m_0+m_1)).
		\end{equation}
		Put $r'_1=\rank E_1', r'_2=\rank E'_2$.  Let $\pi'_{k,i}$ be the composite $V'_k\hookrightarrow V_k \overset{\pi_{k,i}}{\rightarrow} N^{(k)}_i$ for $k=1,2$. Then we have 
		\begin{equation}\label{kerdimeq}
			\dim V'_k\leq h^0(E'_k(m_0)),\; \dim \ker \pi_{k,i}\leq h^0(F_{i+1}(E'_k)(m_0)),
		\end{equation}			
		for $k=1,2$, $1\leq i\leq l_1$ for $1\leq j\leq l_2$. So we obtain
		\begin{align*}
			&-q\xi(r_1+r_2)+\xi(n_1+n_2)\left( \dim \pi_1(U^{(1)}_{a_1(q)b_2})+\dim \pi_2(U^{(2)}_{b_1a_1(q)+b_2a_2(q)})\right)-q\sum_{i=1}^{l_1}\xi^{(1)}_id^{(1)}_{i+1}\\
			&-q\sum_{j=1}^{l_2}\xi^{(2)}_jd^{(2)}_{j+1}+(n_1+n_2)\sum_{i=1}^{l_1}\xi^{(1)}_i\dim \pi_{1,i}(V^{(1)}_{a_1(q)})+(n_1+n_2)\sum_{j=1}^{l_2}\xi^{(2)}_j\dim \pi_{2,i}(V^{(2)}_{a_2(q)})\\
			\overset{(\ref{qeq})(\ref{image})}{=}
			&\xi\Big\{-(\dim V'_1+\dim V'_2)(h^0(E_1(m_0+m_1))+h^0(E_2(m_0+m_1-\gamma)))\\
			&+(\dim V_1+\dim V_2)(h^0(E'_1(m_0+m_1))+h^0(E'_2(m_0+m_1-\gamma)))\Big\} \\
			&-(\dim V'_1+\dim V'_2)\sum_{i=1}^{l_1}\xi^{(1)}_id^{(1)}_{i+1}+(\dim V_1+\dim V_2)\sum_{i=1}^{l_1}\xi^{(1)}_i(\dim V'_1-\dim \ker \pi'_{1,i}) \\
			&-(\dim V'_1+\dim V'_2)\sum_{j=1}^{l_2}\xi^{(2)}_jd^{(2)}_{j+1}+(\dim V_1+\dim V_2)\sum_{j=1}^{l_2}\xi^{(2)}_j(\dim V'_2-\dim \ker \pi'_{2,j}) \\
			\overset{(\ref{xidef})}{=}
			&\big(\dim V_1+\dim V_2 -\sum_{i=1}^{l_1}\epsilon^{(1)}_id^{(1)}_{i+1}-\sum_{j=1}^{l_2}\epsilon^{(2)}_jd^{(2)}_{j+1}\big)\big\{-(\dim V'_1+\dim V'_2)(2rd_Xm_1+\dim V_1+\dim V_2)\\
			&+(\dim V_1+\dim V_2)((r'_1+r'_2)d_Xm_1+\chi(E'_1(m_0))+\chi(E'_2(m_0-\gamma)))\big\} \\
			&-2rd_Xm_1(\dim V'_1+\dim V'_2)\sum_{i=1}^{l_1}\epsilon^{(1)}_id^{(1)}_{i+1}+2rd_Xm_1(\dim V_1+\dim V_2)\sum_{i=1}^{l_1}\epsilon^{(1)}_i(\dim V'_1-\dim \ker \pi'_{1,i}) \\
			&-2rd_Xm_1(\dim V'_1+\dim V'_2)\sum_{j=1}^{l_2}\epsilon^{(2)}_jd^{(2)}_{j+1}+2rd_Xm_1(\dim V_1+\dim V_2)\sum_{j=1}^{l_2}\epsilon^{(2)}_j(\dim V'_2-\dim \ker \pi'_{2,j}) \\
			=
			&-2rd_Xm_1(\dim V_1+\dim V_2)\\
			&\times \Big\{\dim V'_1+\dim V'_2-\sum_{i=1}^{l_1}\epsilon^{(1)}_i(\dim V'_1-\dim \ker \pi'_{1,i})-\sum_{j=1}^{l_2}\epsilon^{(2)}_j(\dim V'_2-\dim \ker \pi'_{2,j})\Big\}\\
			&+(r'_1+r'_2)d_Xm_1(\dim V_1+\dim V_2)\big(\dim V_1+\dim V_2 -\sum_{i=1}^{l_1}\epsilon^{(1)}_id^{(1)}_{i+1}-\sum_{j=1}^{l_2}\epsilon^{(2)}_jd^{(2)}_{j+1}\big)\\
			&+(\dim V_1+\dim V_2)\Big(\dim V_1+\dim V_2 -\sum_{i=1}^{l_1}\epsilon^{(1)}_id^{(1)}_{i+1}-\sum_{j=1}^{l_2}\epsilon^{(2)}_jd^{(2)}_{j+1}\Big) \\
			&\times \big\{-(\dim V'_1+\dim V'_2)+\chi(E'_1(m_0))+\chi(E'_2(m_0-\gamma))\big\}\\
			\overset{(\ref{kerdimeq})}{\geq} 
			&(r'_1+r'_2)d_Xm_1(\dim V_1+\dim V_2)\biggl\{h^0(E_1(m_0))+h^0(E_2(m_0-\gamma)) -\sum_{i=1}^{l_1}\epsilon^{(1)}_id^{(1)}_{i+1}-\sum_{j=1}^{l_2}\epsilon^{(2)}_jd^{(2)}_{j+1}\biggr\}\\
			&-2rd_Xm_1(\dim V_1+\dim V_2)\biggl\{ h^0(E'_1(m_0))+h^0(E'_2(m_0-\gamma)) \\
			&-\sum_{i=1}^{l_1}\epsilon^{(1)}_i\big(h^0(E'_1(m_0))-h^0(F_{i+1}(E'_1)(m_0))\big)-\sum_{j=1}^{l_2}\epsilon^{(2)}_j\big(h^0(E'_2(m_0-\gamma))-h^0(F_{j+1}(E'_2)(m_0-\gamma))\big)\biggr\}\\
			&-(\dim V_1+\dim V_2)\left(\dim V_1+\dim V_2 -\sum_{i=1}^{l_1}\epsilon^{(1)}_id^{(1)}_{i+1}-\sum_{j=1}^{l_2}\epsilon^{(2)}_jd^{(2)}_{j+1}\right)\\
			&\times \left(\dim V'_1+\dim V'_2-\chi(E'_1(m_0))-\chi(E'_2(m_0-\gamma))\right)\\
			\overset{(\ref{stableineq})}{>}&0.
		\end{align*}
		Hence $x$ is a properly stable point.
	\end{proof}
	\section{Types of underlying vector bundles}
	In this subsection, we investigate types of underlying vector bundles. Take $\bm{t}=(t_i)_{1\leq i\leq 3}\in T_3, \boldsymbol{\nu}\in \mcN$ and put $D=t_1+t_2+t_3$. Let $(E_1,E_2,\phi,\nabla,l^{(1)}_*,l^{(2)}_*)$ be a $\boldsymbol{\nu}$-parabolic $\phi$-connection. We assume that $0<\alpha_{i,j}\ll 1$ for any $1\leq i, j\leq 3$ and $\gamma \gg 0$.
	
	\begin{proposition}
		For any $\boldsymbol{\alpha}$-stable $\boldsymbol{\nu}$-parabolic $\phi$-connection $(E_1,E_2,\phi,\nabla,l^{(1)}_*,l^{(2)}_*)$ of rank 3 and degree $-2$, we have
		\[
		E_1 \cong E_2 \cong \Opl \oplus \Opl(-1) \oplus \Opl(-1).
		\]
	\end{proposition}
	\begin{proof}
		Take decompositions
		\begin{align*}
			E_1&=\Opl(l_1)\oplus\Opl(l_2)\oplus \Opl(l_3)  \qquad(l_1+l_2+l_3=-2,\: l_1\geq l_2 \geq l_3)\\
			E_2&=\Opl(m_1)\oplus\Opl(m_2)\oplus \Opl(m_3) \qquad(m_1+m_2+m_3=-2,\: m_1\geq m_2 \geq m_3).
		\end{align*}
		If a triple of integers $(n_1,n_2,n_3)$ satisfies $n_1+n_2+n_3=-2$ and $n_1\geq n_2\geq n_3$, then $(n_1,n_2,n_3)$ satisfies one of the following conditions:
		\begin{itemize}
			\setlength{\itemsep}{0cm}
			\item [(i)] $n_1\geq n_2\geq 0>n_3$,
			\item [(ii)] $n_1\geq 1, \; 0> n_2\geq n_3$,
			\item [(iii)] $n_1=0,\; n_2=n_3=-1$.
		\end{itemize}

		If $(l_1,l_2,l_3)$ and $(m_1,m_2,m_3)$ satisfy the condition (i), then we have $\phi(\Opl(l_1) \oplus \Opl(l_2)) \subset \Opl(m_1) \oplus \Opl(m_2)$. The composite 
		\[
		\Opl(l_1) \oplus \Opl(l_2) \rightarrow E_1 \overset{\nabla}{\longrightarrow} E_2 \otimes \Omega_{\pl}^1(D)\rightarrow \Opl(m_3) \otimes \Omega_{\pl}^1(D) \cong \Opl(m_3+1)
		\]
		becomes a homomorphism and must be zero since $m_3+1=-1-m_1-m_2 \leq -1$. So we have $\nabla(\Opl(l_1) \oplus \Opl(l_2)) \subset (\Opl(m_1) \oplus \Opl(m_2)) \otimes \Omega_{\pl}^1(D)$. Since $\mu(\Opl(l_1)\oplus \Opl(l_2))+\mu(\Opl(m_1)\oplus \Opl(m_2)) \geq 0$, the pair $(\Opl(l_1)\oplus \Opl(l_2), \Opl(m_1)\oplus \Opl(m_2))$ breaks the stability of $(E_1,E_2,\phi,\nabla,\l^{(1)}_*,l^{(2)}_*)$.
		
		Suppose that $(l_1,l_2,l_3)$ satisfies (i) and $(m_1,m_2,m_3)$ satisfies (ii). Then the pair  $(\Opl(l_1) \oplus \Opl(l_2), \Opl(m_1)\oplus\Opl(m_2))$ breaks the stability of $(E_1,E_2,\phi,\nabla,\l^{(1)}_*,l^{(2)}_*)$.
		
		Suppose that $(l_1,l_2,l_3)$ satisfies (i) and $(m_1,m_2,m_3)$ satisfies (iii). Then we have $\phi(\Opl(l_1)\oplus \Opl(l_2)) \subset \Opl(m_1)$.  If $l_1\geq 1$, then the pair $(\Opl(l_1),\Opl(m_1))$ breaks the stability.  If $l_1=0$, then we have $l_2=0$. Put $F_1=\Ker \phi |_{\Opl(l_1)\oplus \Opl(l_2)}$. Then the composite 
		\[
		f \colon F_1\longrightarrow E_1 \overset{\nabla}{\longrightarrow} E_2 \otimes \Omega^1_C(D()
		\]
		becomes a homomorphism. Put $F_2 =(\Image f) \otimes (\Omega^1_C(D))^\vee$. The pair  $(F_1, F_2)$ breaks the stability.
		
		Suppose that $(l_1,l_2,l_3)$ satisfies (ii) and $(m_1,m_2,m_3)$ satisfies (i). If $l_1>m_1$, then the composite $\Opl(l_1) \rightarrow E_1\overset{\nabla}{\rightarrow} E_2 \otimes \Omega_{\pl}^1(D)$ becomes a homomorphism. Put $F_2=(\Image \nabla|_{\Opl(l_1)}) \otimes (\Omega_{\pl}^1(D))^\vee$, then $(\Opl(l_1), F_2)$ breaks the stability. If $l_1 \leq m_1$, then we can see that the pair  $(\Opl(l_1)\oplus \Opl(l_2),\Opl(m_1)\oplus \Opl(m_2))$ breaks the stability because 
		\[
		\mu(\Opl(l_1)\oplus \Opl(l_2))+\mu(\Opl(m_1)\oplus \Opl(m_2))=\frac{l_1+l_2-2-m_3}{2}\geq \frac{1}{2}.
		\] 
		
		If $(l_1,l_2,l_3)$ satisfies (ii) and $(m_1,m_2,m_3)$ satisfies (ii) or (iii), then $(\Opl(l_1),\Opl(m_1))$ breaks the stability.
		
		Suppose that $(l_1,l_2,l_3)$ satisfies (iii) and $(m_1,m_2,m_3)$ satisfies (i), then $m_3=-2-m_1-m_2\leq -2$. If $m_3<-2$, then the pair $(E_1, \Opl(m_1)\oplus \Opl(m_2))$ breaks the stability of $(E_1,E_2,\phi,\nabla,\l^{(1)}_*,l^{(2)}_*)$. If $m_3=-2$ , then $m_1=m_2=0$ and $\phi(\Opl(l_2) \oplus \Opl(l_3)) \subset \Opl(m_1) \oplus \Opl(m_2)$. Moreover the composite
		\[
		f: \Opl(l_2) \oplus \Opl(l_3) \rightarrow E_1 \overset{\nabla}{\longrightarrow} E_2 \otimes \Omega_{\pl}^1(D) \rightarrow \Opl(m_3)\otimes \Omega_{\pl}^1(D)
		\]
		becomes a homomorphism. Let $F_1 = \Ker f$. If $F_1 =\Opl(l_2) \oplus \Opl(l_3)$, then the pair $(E_1, \Opl(m_1)\oplus \Opl(m_2))$ breaks the stability. If $F_1 \neq \Opl(l_2) \oplus \Opl(l_3)$, then we have $F_1\cong \Opl(-1)$ since $\Opl(l_2) \cong \Opl(l_3)\cong \Opl(m_3) \otimes \Omega_{\pl}^1(D) \cong \Opl(-1)$. So the pair  $(\Opl(l_1)\oplus F_1, \Opl(m_1)\oplus \Opl(m_2))$ breaks the stability.
		
		Suppose that $(l_1,l_2,l_3)$ satisfies (iii) and $(m_1,m_2,m_3)$ satisfies (ii). If $m_2<-1$, then the pair  $(\Opl(l_1), \Opl(m_1))$ breaks the stability. If $m_2=-1$ and $m_3<-2$, then the pair $(E_1, \Opl(m_1)\oplus \Opl(m_2))$ breaks the stability. If $m_2=-1$ and $m_3=-2$, then we have $\phi(\Opl(l_2) \oplus \Opl(l_3)) \subset \Opl(m_1) \oplus \Opl(m_2)$ and so the composite
		\[
		f: \Opl(l_2) \oplus \Opl(l_3) \rightarrow E_1 \overset{\nabla}{\longrightarrow} E_2 \otimes \Omega_{\pl}^1(D) \rightarrow \Opl(m_3)\otimes \Omega_{\pl}^1(D)
		\]
		becomes a homomorphism. Let $F_1 = \Ker f$. If $F_1 =\Opl(l_2) \oplus \Opl(l_3)$, then the pair $(E_1, \Opl(m_1)\oplus \Opl(m_2))$ breaks the stability. If $F_1 \neq \Opl(l_2) \oplus \Opl(l_3)$, then we have $F_1\cong \Opl(-1)$ since $\Opl(l_2) \cong \Opl(l_3)\cong \Opl(m_3) \otimes \Omega_{\pl}^1(D) \cong \Opl(-1)$. So the pair $(\Opl(l_1)\oplus F_1, \Opl(m_1)\oplus \Opl(m_2))$ breaks the stability.
		
		Hence we have $E_1 \cong E_2 \cong \Opl \oplus \Opl(-1) \oplus \Opl(-1)$.
	\end{proof}
	
	\section{Smoothness of moduli space of parabolic $\phi$-connections}
	
	Let $\tilde{t}_i \subset \pl \times T_3 \times \mcN$ be the section defined by 
	\[
	T_3 \times \mcN \hookrightarrow \pl \times T_3 \times \mcN; \quad ((t_j)_{1\leq j\leq 3}, (\nu_{m,n})^{1\leq m\leq 3}_{0\leq n \leq 2}) \mapsto (t_i, (t_j)_{1\leq j\leq 3}, (\nu_{m,n})^{1\leq m\leq 3}_{0\leq n \leq 2}) 
	\]
	for $i=1,2,3$ and $D(\tilde{\bm{t}})=\tilde{t}_1+\tilde{t}_2+\tilde{t}_3$ be a relative effective Cartier divisor for the projection  $\pl \times T_3 \times \mcN \rightarrow T_3 \times \mcN$. For each $1\leq i\leq 3$ and $0\leq j\leq 2$, let 
	\[
	\tilde{\nu}_{i,j}:=\{(\nu_{i,j}, (t_k)_k, (\nu_{m,n})_{m,n})\} \subset \mathbb{C}\times T_3 \times \mcN.
	\]
	\begin{proposition}
		$\overline{M^{\boldsymbol{\alpha}}_3}(0,0,2)$ is smooth over $T_3 \times \mcN$.
	\end{proposition}
	\begin{proof}
		Let $A$ be an artinian local ring with the residue field $A/\mfm=k$ and $I$ be an ideal of $A$ such that $\mfm I=0$. Let  $\Spec A \rightarrow T_3\times \mcN$ be a morphism and  $t_i \in \pl_{A}, \nu_{i,j}\in A$ be the elements obtained by the pullback of the sections $\tilde{t}_i, \tilde{\nu}_{i,j}$, respectively. By the definition of $\mcN$, we have 
		\begin{equation}\label{eigrel}
			\nu_{i,0}+\nu_{i,1}+\nu_{i,2}=2\res_{t_i}(\tfrac{dz}{z-t_3}).
		\end{equation}
		We take an open subset $U\subset \mathbb{P}^1_A$ such that $ U\cong \Spec A[z]$ and $t_1,t_2,t_3 \in U$. We show that 
		\begin{equation}\label{smsurj}
			\overline{M^{\boldsymbol{\alpha}}_3}(0,0,2)(A)\longrightarrow \overline{M^{\boldsymbol{\alpha}}_3}(0,0,2)(A/I)
		\end{equation}
		is surjective. Put $K:=\Omega_{\pl_{A/I}/(A/I)}(D(\tilde{\bm{t}})_{A/I})$ and take $(E_1,E_2,\phi,\nabla,l^{(1)}_*,l^{(2)}_*) \in \overline{M^{\boldsymbol{\alpha}}_3}(0,0,2)(A/I)$.
		Then $E_1\cong E_2 \cong \mcO_{\pl_{A/I}}\oplus \mcO_{\pl_{A/I}}(-1)\oplus\mcO_{\pl_{A/I}}(-1)$. The homomorphism $\phi$ can be written by the form
		\[
		\phi=
		\begin{pmatrix}
			\phi_{11}&\phi_{12}&\phi_{13} \\
			0&\phi_{22}&\phi_{23} \\
			0&\phi_{32}&\phi_{33}
		\end{pmatrix},
		\]
		where $\phi_{11}, \phi_{22}, \phi_{23}, \phi_{32}, \phi_{33} \in H^0(\mcO_{\pl_{A/I}})\cong A/I$ and $\phi_{12}, \phi_{13} \in H^0(\mcO_{\pl_{A/I}}(1))$. By Lemma \ref{zerostab}, $\phi_{11}$ is a unit, so we may assume that $\phi_{12}=\phi_{13}=0$. Then $\nabla$ can be written  by
		\[
		\quad \nabla=\phi \otimes d+
		\begin{pmatrix}
			0&0&0 \\
			0&\phi_{22}&\phi_{23} \\
			0&\phi_{32}&\phi_{33}
		\end{pmatrix}\frac{dz}{z-t_3}
		+
		\begin{pmatrix}
			\omega_{11}&\omega_{12}&\omega_{13}\\
			\omega_{21}&\omega_{22}&\omega_{23}\\
			\omega_{31}&\omega_{32}&\omega_{33}
		\end{pmatrix}, 
		\]
		where $\omega_{21}, \omega_{31} \in  H^0(K(-1))\cong A/I$, $\omega_{11},\omega_{22},\omega_{23},\omega_{32},\omega_{33}\in H^0(K)$, and $\omega_{12},\omega_{13} \in H^0(K(1))$. Taking decompositions $E_1\cong E_2 \cong \mcO_{\pl_{A/I}}\oplus \mcO_{\pl_{A/I}}(-1)\oplus\mcO_{\pl_{A/I}}(-1)$ well, we may assume that  $\omega_{11}=\omega_{31}=0$ and $\res_{t_i}\omega_{21}\in (A/I)^\times$ for any $i=1,2,3$.  The smoothness of the map $M^{\boldsymbol{\alpha}}_3(0,0,2)\rightarrow T_3 \times \mcN$ is proved in \cite{In}, which means the map (\ref{smsurj}) is surjective when $\wedge^3\phi \notin \mfm/I$. So we consider the case $\wedge^3\phi \in \mfm/I$.

		Assume that $\rank \phi \otimes \id_k=2$. Then applying certain automorphisms of $E_1$ and $E_2$, we may assume that $\phi\otimes \id_k$ and $\nabla\otimes \id_k$ have the form (\ref{rk2form}). Then we may also assume that $\phi_{11}=\phi_{33}=1$ and $\phi_{23}=\phi_{32}=0$ and $\omega_{23}=0$. We note that $\phi_{22}\in \mfm/I$. In the same way of the proof Lemma \ref{detzero}, we obtain $|\res_{t_i}\nabla-\lambda\phi_{t_i}|=(\wedge^3 \phi_{t_i})(\nu_{i,0}-\lambda)(\nu_{i,1}-\lambda)(\nu_{i,2}-\lambda)$. By comparing the coefficients on both sides and using (\ref{eigrel}), we have
		\begin{equation}\label{rk2coe2}
			\omega_{22}(t_i)+\phi_{22}\omega_{33}(t_i)=0,
		\end{equation}
		\begin{equation}\label{rk2coe1}
			\omega_{22}(t_i)\omega_{33}(t_i)-\omega_{21}(t_i)\omega_{12}(t_i)=\phi_{22}(\nu_{i,0}\nu_{i,1}+\nu_{i,0}\nu_{i,2}+\nu_{i,1}\nu_{i,2}-(\res_{t_i}(\tfrac{dz}{z-t_3}))^2), 
		\end{equation}
		\begin{equation}\label{rk2coe0}
			-\omega_{21}(t_i)(\omega_{12}(t_i)(\omega_{33}(t_i)+\res_{t_i}(\tfrac{dz}{z-t_3}))-\omega_{13}(t_i)\omega_{32}(t_i))=\phi_{22}\nu_{i,0}\nu_{i,1}\nu_{i,2},
		\end{equation}
		for each $i=1,2,3$, where $\omega_{ij}(t_m):=\res_{t_m}\omega_{ij}$. 
		From the form (\ref{rk2form}), we have $\omega_{13}(t_i) \in (A/I)^\times$ and $\omega_{32}(t_j) \in (A/I)^\times$ for $j\neq i$. Put
		\[
		v^{(1)}_{i,2}
		=
		\begin{pmatrix}
			\phi_{22}\omega_{13}(t_i)(\omega_{33}(t_i)+\res_{t_i}(\tfrac{dz}{z-t_3})-(\nu_{i,0}+\nu_{i,1}))\\
			\omega_{13}(t_i)\omega_{21}(t_i)\\
			\phi_{22}(\omega_{33}(t_i)+\res_{t_i}(\tfrac{dz}{z-t_3})-\nu_{i,0})(\omega_{33}(t_i)+\res_{t_i}(\tfrac{dz}{z-t_3})-\nu_{i,1})
		\end{pmatrix}, \;
		v^{(1)}_{i,1}
		=
		\begin{pmatrix}
			\omega_{13}(t_i)\\
			0\\
			\omega_{33}(t_i)+\res_{t_i}(\tfrac{dz}{z-t_3})-\nu_{i,0}
		\end{pmatrix},
		\]
		\[
		v^{(2)}_{i,2}
		=
		\begin{pmatrix}
			\omega_{13}(t_i)(\omega_{33}(t_i)+\res_{t_i}(\tfrac{dz}{z-t_3})-(\nu_{i,0}+\nu_{i,1}))\\
			\omega_{13}(t_i)\omega_{21}(t_i)\\
			(\omega_{33}(t_i)+\res_{t_i}(\tfrac{dz}{z-t_3})-\nu_{i,0})(\omega_{33}(t_i)+\res_{t_i}(\tfrac{dz}{z-t_3})-\nu_{i,1})
		\end{pmatrix},\;
		v^{(2)}_{i,1}
		=
		\begin{pmatrix}
			\omega_{13}(t_i)\\
			0\\
			\omega_{33}(t_i)+\res_{t_i}(\tfrac{dz}{z-t_3})-\nu_{i,0}
		\end{pmatrix}
		\]
		and 
		\[
		v^{(1)}_{j,2}
		=
		\begin{pmatrix}
			(\omega_{22}(t_j)+\phi_{22}(\res_{t_i}(\tfrac{dz}{z-t_3})-\nu_{j,2}))(\omega_{33}(t_j)+\res_{t_i}(\tfrac{dz}{z-t_3})-\nu_{j,2})\\
			-\omega_{21}(t_j)(\omega_{33}(t_j)+\res_{t_i}(\tfrac{dz}{z-t_3})-\nu_{i,2})\\
			\omega_{21}(t_j)\omega_{32}(t_j)
		\end{pmatrix}, \;
		v^{(1)}_{j,1}
		=
		\begin{pmatrix}
			-\phi_{22}\nu_{j,0}\\
			\omega_{21}(t_j)\\
			0
		\end{pmatrix},
		\]
		\[
		v^{(2)}_{j,2}
		=
		\begin{pmatrix}
			(\omega_{22}(t_j)+\phi_{22}(\res_{t_i}(\tfrac{dz}{z-t_3})-\nu_{j,2}))(\omega_{33}(t_j)+\res_{t_i}(\tfrac{dz}{z-t_3})-\nu_{j,2})\\
			-\phi_{22}\omega_{21}(t_j)(\omega_{33}(t_j)+\res_{t_i}(\tfrac{dz}{z-t_3})-\nu_{j,2})\\
			\omega_{21}(t_j)\omega_{32}(t_j)
		\end{pmatrix}, \;
		v^{(2)}_{j,1}
		=
		\begin{pmatrix}
			-\nu_{j,0}\\
			\omega_{21}(t_j)\\
			0
		\end{pmatrix} 
		\]
		for $j\neq i$. Then we can see that 
		\[
		l^{(1)}_{j,2}=(A/I)v^{(1)}_{j,2}, \quad l^{(1)}_{j,1}=(A/I)v^{(1)}_{j,1}+(A/I)v^{(1)}_{j,2}, \quad l^{(2)}_{j,2}=(A/I)v^{(2)}_{j,2}, \quad l^{(2)}_{j,1}=(A/I)v^{(2)}_{j,1}+(A/I)v^{(2)}_{j,2}
		\]
		for any $j=1,2,3$ by the conditions $\phi_{t_i}(l^{(1)}_{i,j})\subset l^{(2)}_{i,j}$, $(\res_{t_i}\nabla-\nu_{i,j}\phi_{t_i})(l^{(1)}_{i,j})\subset l^{(2)}_{i,j+1}$ and the relations (\ref{rk2coe2}), (\ref{rk2coe1}), (\ref{rk2coe0}).
		We take lifts $\tilde{\phi}_{22} \in A, \tilde{\omega}_{21}\in H^0(\Omega^1_{\pl_{A}/A}(D(\bm{t})_A)(-1)), \tilde{\omega}_{33}\in H^0(\Omega^1_{\pl_{A}/A}(D(\bm{t})_A))$ and $\tilde{\omega}_{13}^{(i)}\in A^\times$ of $\phi_{22}, \omega_{21},\omega_{33}$ and $\omega_{13}(t_i)$, respectively. Put $\tilde{\omega}_{22}:=-\tilde{\phi}_{22}\tilde{\omega}_{33}$ and let $\tilde{\omega}_{12}\in H^0(\Omega^1_{\pl_{A}/A}(D(\bm{t})_A)(1))$ be a lift of $\omega_{12}$ satisfying 
		\[
		\tilde{\omega}_{21}(t_i)\tilde{\omega}_{12}(t_i)=\tilde{\omega}_{22}\tilde{\omega}_{33}-\tilde{\phi}_{22}(\nu_{i,0}\nu_{i,1}+\nu_{i,0}\nu_{i,2}+\nu_{i,1}\nu_{i,2}-(\res_{t_i}(\tfrac{dz}{z-t_3}))^2).
		\]
		Then we can find a lift $\tilde{\omega}_{32} \in H^0(\Omega^1_{\pl_{A}/A}(D(\bm{t})_A))$ of $\omega_{32}$ satisfying 
		\[
		\tilde{\omega}_{21}(t_i)(\tilde{\omega}_{12}(t_i)(\tilde{\omega}_{33}(t_i)+\res_{t_i}(\tfrac{dz}{z-t_3}))-\tilde{\omega}_{13}^{(i)}\tilde{\omega}_{32}(t_i))=\tilde{\phi}_{22}\nu_{i,0}\nu_{i,1}\nu_{i,2}.
		\]
		Let $\tilde{\omega}_{13}$ be the element of $H^0(\Omega^1_{\pl_{A}/A}(D(\bm{t})_A)(1))$ satisfying
		\[
		-\tilde{\omega}_{21}(t_j)(\tilde{\omega}_{12}(t_j)(\tilde{\omega}_{33}(t_i)+\res_{t_i}(\tfrac{dz}{z-t_3}))-\tilde{\omega}_{13}(t_j)\tilde{\omega}_{32}(t_j))=\tilde{\phi}_{22}\nu_{j,0}\nu_{j,1}\nu_{j,2}.
		\]
		for $j\neq i$ and $\tilde{\omega}_{13}(t_i)=\tilde{\omega}_{13}^{(i)}$.  Put
		\[
		\tilde{\phi}=
		\begin{pmatrix}
			1&0&0 \\
			0&\tilde{\phi}_{22}&0 \\
			0&0&1
		\end{pmatrix},
		\quad \tilde{\nabla}=\tilde{\phi} \otimes d+
		\begin{pmatrix}
			0&0&0 \\
			0&\tilde{\phi}_{22}&0 \\
			0&0&1
		\end{pmatrix}\frac{dz}{z-t_3}
		+
		\begin{pmatrix}
			0&\tilde{\omega}_{12}&\tilde{\omega}_{13}\\
			\tilde{\omega}_{21}&\tilde{\omega}_{22}&0\\
			0&\tilde{\omega}_{32}&\tilde{\omega}_{33}
		\end{pmatrix}, 
		\]
		\[
		\tilde{v}^{(1)}_{i,2}
		=
		\begin{pmatrix}
			\tilde{\phi}_{22}\tilde{\omega}_{13}(t_i)(\tilde{\omega}_{33}(t_i)+\res_{t_i}(\tfrac{dz}{z-t_3})-(\nu_{i,0}+\nu_{i,1}))\\
			\tilde{\omega}_{13}(t_i)\tilde{\omega}_{21}(t_i)\\
			\tilde{\phi}_{22}(\tilde{\omega}_{33}(t_i)+\res_{t_i}(\tfrac{dz}{z-t_3})-\nu_{i,0})(\tilde{\omega}_{33}(t_i)+\res_{t_i}(\tfrac{dz}{z-t_3})-\nu_{i,1})
		\end{pmatrix}, \;
		\tilde{v}^{(1)}_{i,1}
		=
		\begin{pmatrix}
			\tilde{\omega}_{13}(t_i)\\
			0\\
			\tilde{\omega}_{33}(t_i)+\res_{t_i}(\tfrac{dz}{z-t_3})-\nu_{i,0}
		\end{pmatrix},
		\]
		\[
		\tilde{v}^{(2)}_{i,2}
		=
		\begin{pmatrix}
			\tilde{\omega}_{13}(t_i)(\tilde{\omega}_{33}(t_i)+\res_{t_i}(\tfrac{dz}{z-t_3})-(\nu_{i,0}+\nu_{i,1}))\\
			\tilde{\omega}_{13}(t_i)\tilde{\omega}_{21}(t_i)\\
			(\tilde{\omega}_{33}(t_i)+\res_{t_i}(\tfrac{dz}{z-t_3})-\nu_{i,0})(\tilde{\omega}_{33}(t_i)+\res_{t_i}(\tfrac{dz}{z-t_3})-\nu_{i,1})
		\end{pmatrix},\;
		\tilde{v}^{(2)}_{i,1}
		=
		\begin{pmatrix}
			\tilde{\omega}_{13}(t_i)\\
			0\\
			\tilde{\omega}_{33}(t_i)+\res_{t_i}(\tfrac{dz}{z-t_3})-\nu_{i,0}
		\end{pmatrix}
		\]
		and 
		\[
		\tilde{v}^{(1)}_{j,2}
		=
		\begin{pmatrix}
			(\tilde{\omega}_{22}(t_j)+\tilde{\phi}_{22}(\res_{t_i}(\tfrac{dz}{z-t_3})-\nu_{j,2}))(\tilde{\omega}_{33}(t_j)+\res_{t_i}(\tfrac{dz}{z-t_3})-\nu_{j,2})\\
			-\tilde{\omega}_{21}(t_j)(\tilde{\omega}_{33}(t_j)+\res_{t_i}(\tfrac{dz}{z-t_3})-\nu_{i,2})\\
			\tilde{\omega}_{21}(t_j)\tilde{\omega}_{32}(t_j)
		\end{pmatrix}, \;
		\tilde{v}^{(1)}_{j,1}
		=
		\begin{pmatrix}
			-\tilde{\phi}_{22}\nu_{j,0}\\
			\tilde{\omega}_{21}(t_j)\\
			0
		\end{pmatrix},
		\]
		\[
		\tilde{v}^{(2)}_{j,2}
		=
		\begin{pmatrix}
			(\tilde{\omega}_{22}(t_j)+\tilde{\phi}_{22}(\res_{t_i}(\tfrac{dz}{z-t_3})-\nu_{j,2}))(\tilde{\omega}_{33}(t_j)+\res_{t_i}(\tfrac{dz}{z-t_3})-\nu_{j,2})\\
			-\tilde{\phi}_{22}\tilde{\omega}_{21}(t_j)(\tilde{\omega}_{33}(t_j)+\res_{t_i}(\tfrac{dz}{z-t_3})-\nu_{j,2})\\
			\tilde{\omega}_{21}(t_j)\tilde{\omega}_{32}(t_j)
		\end{pmatrix}, \;
		\tilde{v}^{(2)}_{j,1}
		=
		\begin{pmatrix}
			-\nu_{j,0}\\
			\tilde{\omega}_{21}(t_j)\\
			0
		\end{pmatrix} 
		\]
		for $j\neq i$. Let $\tilde{l}^{(m)}_{j,2}=A\tilde{v}^{(m)}_{j,2} \subset A^{\oplus 3}$ and $ \tilde{l}^{(m)}_{j,1}=A\tilde{v}^{(m)}_{j,1}+A\tilde{v}^{(m)}_{j,2} \subset A^{\oplus 3}$ for $m=1,2$ and  $j=1,2,3$. Then we can see that $A^{\oplus 3}/l^{(m)}_{j,n}$ is flat over $A$ and $(\res_{t_j}\tilde{\nabla}-\nu_{j,n}\tilde{\phi}_{t_j})(l^{(1)}_{j,n})\subset l^{(2)}_{j,n+1}$ for any $j=1,2,3$ and $n=0,1,2$ by the way of taking lifts $\tilde{\omega}_{12},\tilde{\omega}_{13},\tilde{\omega}_{22},\tilde{\omega}_{32}$. So $\tilde{\phi}, \tilde{\nabla}, \tilde{l}^{(1)}_{i,j}$ and $ \tilde{l}^{(2)}_{i,j}$ are desire lifts.
		
		Next we consider the case $\rank \phi\otimes \id_k=1$. Then applying certain automorphisms of $E_1$ and $E_2$, we may assume that $\phi\otimes \id_k$ and $\nabla \otimes \id_k$ have the form (\ref{rk1form}). In particular, we may assume that $\omega_{32}(t_i)\in (A/I)^\times$. In the same way as the proof Lemma \ref{detzero}, we also obtain $|\res_{t_i}\nabla-\lambda\phi_{t_i}|=(\wedge^3 \phi)(\nu_{i,0}-\lambda)(\nu_{i,1}-\lambda)(\nu_{i,2}-\lambda)$, and by comparing the coefficients on both sides and using (\ref{eigrel}), we have
		\begin{equation}\label{rk1coe2}
			\phi_{22}\omega_{33}(t_i)+\phi_{33}\omega_{22}(t_i)-\phi_{23}\omega_{32}(t_i)-\phi_{32}\omega_{23}(t_i)=0,
		\end{equation}
		\begin{equation}\label{rk1coe1}
			\begin{split}
				&(\omega_{22}(t_i)\omega_{33}(t_i)-\omega_{23}(t_i)\omega_{32}(t_i))-\omega_{21}(t_i)(\omega_{12}(t_i)\phi_{33}-\omega_{13}(t_i)\phi_{32})\\
				&\quad=(\phi_{22}\phi_{33}-\phi_{23}\phi_{32})(\nu_{i,0}\nu_{i,1}+\nu_{i,0}\nu_{i,2}+\nu_{i,1}\nu_{i,2}-(\res_{t_i}(\tfrac{dz}{z-t_3}))^2),
			\end{split}
		\end{equation}
		\begin{equation}\label{rk1coe0}
			\begin{split}
				&-\omega_{21}(t_i)(\omega_{12}(t_i)(\omega_{33}(t_i)+\phi_{33}\res_{t_i}(\tfrac{dz}{z-t_3}))-\omega_{13}(t_i)(\omega_{32}(t_i)+\phi_{32}\res_{t_i}(\tfrac{dz}{z-t_3})))\\
				&\quad=(\phi_{22}\phi_{33}-\phi_{23}\phi_{32})\nu_{i,0}\nu_{i,1}\nu_{i,2}.
			\end{split}
		\end{equation}
		Put
		\[
		v^{(1)}_{j,2}
		:=
		\begin{pmatrix}
			\omega_{22}(t_j)\omega_{33}(t_j)-\omega_{32}(t_j)\omega_{23}(t_j)+(\phi_{22}\phi_{33}-\phi_{23}\phi_{32})(\res_{t_j}(\tfrac{dz}{z-t_3})-\nu_{j,2})^2\\
			-\omega_{21}(t_j)(\omega_{33}(t_j)+\phi_{33}(\res_{t_j}(\tfrac{dz}{z-t_3})-\nu_{j,2}))\\
			\omega_{21}(t_j)(\omega_{32}(t_j)+\phi_{32}(\res_{t_j}(\tfrac{dz}{z-t_3})-\nu_{j,2}))
		\end{pmatrix},
		\]
		\[
		v^{(1)}_{j,1}:=
		\begin{pmatrix}
			-\nu_{j,0}(\phi_{22}\omega_{32}(t_j)-\phi_{32}\omega_{22}(t_j))+\omega_{21}(t_j)\omega_{12}(t_j)\phi_{32}\\
			(\omega_{32}(t_j)+\phi_{32}(\res_{t_i}(\tfrac{dz}{z-t_3})-\nu_{j,0}))\omega_{21}(t_j)\\
			0
		\end{pmatrix},
		\]
		\[
		v^{(2)}_{j,2}:=(\res_{t_i}\nabla-\nu_{j,1}\phi_{t_j})(v^{(1)}_{j,1}), \;
		v^{(2)}_{j,1}:=
		\begin{pmatrix}
			-\nu_{i,0}\\
			\omega_{21}(t_j)\\
			0
		\end{pmatrix}. 
		\]
		Then we can see that $l^{(1)}_{j,2}=(A/I)v^{(1)}_{j,2}, l^{(1)}_{j,1}=(A/I)v^{(1)}_{j,1}+(A/I)v^{(1)}_{j,2}, l^{(2)}_{j,2}=(A/I)v^{(2)}_{j,2}$ and $ l^{(2)}_{j,1}=(A/I)v^{(2)}_{j,1}+(A/I)v^{(2)}_{j,2}$ for any $j=1,2,3$ by the conditions $\phi_{t_j}(l^{(1)}_{j,m})\subset l^{(2)}_{j,m}$ and  $(\res_{t_j}\nabla-\nu_{j,m}\phi_{t_j})(l^{(1)}_{j,m})\subset l^{(2)}_{j,m+1}$, and the relations (\ref{rk1coe2}), (\ref{rk1coe1}), (\ref{rk1coe0}). We take lifts $\psi_{22},\psi_{23},\tilde{\phi}_{32},\tilde{\phi}_{33} \in A, \tilde{\omega}_{21} \in H^0(\Omega_{\pl_{A}/A}^1(D(\bm{t})_A)(-1)), \tilde{\omega}_{32}, \tilde{\omega}_{33}\in H^0(\Omega_{\pl_{A}/A}^1(D(\bm{t})_A))$ and $\tilde{\omega}_{12} \in H^0(\Omega_{\pl_{A}/A}^1(D(\bm{t})_A)(1))$ of $\phi_{22},\phi_{23},\phi_{32},\phi_{33}, \omega_{21}, \omega_{32},\omega_{33}, \omega_{12}$, respectively. We take lifts $\tilde{\omega}_{13} \in H^0(\Omega_{\pl_{A}/A}^1(D(\bm{t})_A)(1))$, $\tilde{\omega}_{22}, \tilde{\omega}_{23}  \in H^0(\Omega_{\pl_{A}/A}^1(D(\bm{t})_A))$ of $\omega_{13}, \omega_{22},\omega_{23}$, respectively, satisfying
		\begin{equation*}
			\begin{split}
				&-\tilde{\omega}_{21}(t_j)(\tilde{\omega}_{12}(t_j)(\tilde{\omega}_{33}(t_j)+\tilde{\phi}_{33}\res_{t_j}(\tfrac{dz}{z-t_3}))-\tilde{\omega}_{13}(t_j)(\tilde{\omega}_{32}(t_j)+\tilde{\phi}_{32}\res_{t_j}(\tfrac{dz}{z-t_3})))\\
				&\quad=(\psi_{22}\tilde{\phi}_{33}-\psi_{23}\tilde{\phi}_{32})\nu_{j,0}\nu_{j,1}\nu_{j,2},
			\end{split}
		\end{equation*}
		\begin{equation*}
			\begin{split}
				&-\tilde{\omega}_{23}(t_i)\tilde{\omega}_{32}(t_i)-\tilde{\omega}_{21}(t_i)(\tilde{\omega}_{12}(t_i)\tilde{\phi}_{33}-\tilde{\omega}_{13}(t_i)\tilde{\phi}_{32})\\
				&\quad=(\psi_{22}\tilde{\phi}_{33}-\psi_{23}\tilde{\phi}_{32})(\nu_{i,0}\nu_{i,1}+\nu_{i,0}\nu_{i,2}+\nu_{i,1}\nu_{i,2}-(\res_{t_i}(\tfrac{dz}{z-t_3}))^2),
			\end{split}
		\end{equation*}
		\begin{equation*}
			\begin{split}
				&(\tilde{\omega}_{22}(t_j)\tilde{\omega}_{33}(t_j)-\tilde{\omega}_{23}(t_j)\tilde{\omega}_{32}(t_j))-\tilde{\omega}_{21}(t_j)(\tilde{\omega}_{12}(t_j)\tilde{\phi}_{33}-\tilde{\omega}_{13}(t_j)\tilde{\phi}_{32})\\
				&\quad=(\psi_{22}\tilde{\phi}_{33}-\psi_{23}\tilde{\phi}_{32})(\nu_{i,0}\nu_{i,1}+\nu_{i,0}\nu_{i,2}+\nu_{i,1}\nu_{i,2}-(\res_{t_j}(\tfrac{dz}{z-t_3}))^2)
			\end{split}
		\end{equation*}
		for any $j=1,2,3$. 
		Put
		\[
		\eta:=\psi_{22}\tilde{\omega}_{33}+\tilde{\phi}_{33}\tilde{\omega}_{22}-\psi_{23}\tilde{\omega}_{32}-\tilde{\phi}_{32}\tilde{\omega}_{23}.
		\]
		Since $\tilde{\omega}_{32}(t_i)\neq 0$ and $\tilde{\omega}_{33}(t_i)=0$, $\tilde{\omega}_{32}$ and $\tilde{\omega}_{33}$ generate $H^0(\Omega_{\pl_{A}/A}^1(D(\bm{t})_A))\cong A^{\oplus2}$ as $A$-module. In particular, $\eta$ can be written by the form $b_1\tilde{\omega}_{32}+b_2\tilde{\omega}_{33}$, where $b_1,b_2 \in A$.
		Since $\eta \mod I$ is zero by (\ref{rk1coe2}), we have $b_1,b_2 \in I$. Put $\tilde{\phi}_{22}=\psi_{22}-b_2, \tilde{\phi}_{23}=\psi_{23}+b_1$. Then we have
		\begin{equation}\label{rk1coe2tilde}
			\tilde{\phi}_{22}\tilde{\omega}_{33}+\tilde{\phi}_{33}\tilde{\omega}_{22}-\tilde{\phi}_{23}\tilde{\omega}_{32}-\tilde{\phi}_{32}\tilde{\omega}_{23}=0,
		\end{equation}
		\begin{equation}\label{rk1coe1tilde}
			\begin{split}
				&(\tilde{\omega}_{22}(t_j)\tilde{\omega}_{33}(t_j)-\tilde{\omega}_{23}(t_i)\tilde{\omega}_{32}(t_i))-\tilde{\omega}_{21}(t_i)(\tilde{\omega}_{12}(t_i)\tilde{\phi}_{33}-\tilde{\omega}_{13}(t_i)\tilde{\phi}_{32})\\
				&\quad=(\tilde{\phi}_{22}\tilde{\phi}_{33}-\tilde{\phi}_{23}\tilde{\phi}_{32})(\nu_{i,0}\nu_{i,1}+\nu_{i,0}\nu_{i,2}+\nu_{i,1}\nu_{i,2}-(\res_{t_i}(\tfrac{dz}{z-t_3}))^2),
			\end{split}
		\end{equation}
		\begin{equation}\label{rk1coe0tilde}
			\begin{split}
				&-\tilde{\omega}_{21}(t_j)(\tilde{\omega}_{12}(t_j)(\tilde{\omega}_{33}(t_j)+\tilde{\phi}_{33}\res_{t_j}(\tfrac{dz}{z-t_3}))-\tilde{\omega}_{13}(t_j)(\tilde{\omega}_{32}(t_j)+\tilde{\phi}_{32}\res_{t_j}(\tfrac{dz}{z-t_3})))\\
				&\quad=(\tilde{\phi}_{22}\tilde{\phi}_{33}-\tilde{\phi}_{23}\tilde{\phi}_{32})\nu_{j,0}\nu_{j,1}\nu_{j,2}
			\end{split}
		\end{equation}
		for any $j=1,2,3$ because $\mfm I=0$. Put
		\[
		\tilde{\phi}=
		\begin{pmatrix}
			1&0&0\\
			0&\tilde{\phi}_{22}&\tilde{\phi}_{23} \\
			0&\tilde{\phi}_{32}&\tilde{\phi}_{33}
		\end{pmatrix},
		\quad \tilde{\nabla}=\tilde{\phi} \otimes d+
		\begin{pmatrix}
			0&0&0\\
			0&\tilde{\phi}_{22}&\tilde{\phi}_{23} \\
			0&\tilde{\phi}_{32}&\tilde{\phi}_{33}
		\end{pmatrix}\frac{dz}{z-t_3}
		+
		\begin{pmatrix}
			0&\tilde{\omega}_{12}&\tilde{\omega}_{13}\\
			\tilde{\omega}_{21}&\tilde{\omega}_{22}&\tilde{\omega}_{23}\\
			0&\tilde{\omega}_{32}&\tilde{\omega}_{33}
		\end{pmatrix}, 
		\]
		\[
		\tilde{v}^{(1)}_{j,2}
		:=
		\begin{pmatrix}
			\tilde{\omega}_{22}(t_j)\tilde{\omega}_{33}(t_j)-\tilde{\omega}_{32}(t_j)\tilde{\omega}_{23}(t_j)+(\tilde{\phi}_{22}\tilde{\phi}_{33}-\tilde{\phi}_{23}\tilde{\phi}_{32})(\res_{t_j}(\tfrac{dz}{z-t_3})-\nu_{j,2})^2\\
			-\tilde{\omega}_{21}(t_j)(\tilde{\omega}_{33}(t_j)+\tilde{\phi}_{33}(\res_{t_j}(\tfrac{dz}{z-t_3})-\nu_{j,2}))\\
			\tilde{\omega}_{21}(t_j)(\tilde{\omega}_{32}(t_j)+\tilde{\phi}_{32}(\res_{t_j}(\tfrac{dz}{z-t_3})-\nu_{j,2}))
		\end{pmatrix},
		\]
		\[
		\tilde{v}^{(1)}_{j,1}:=
		\begin{pmatrix}
			-\nu_{j,0}(\tilde{\phi}_{22}\tilde{\omega}_{32}(t_j)-\tilde{\phi}_{32}\tilde{\omega}_{22}(t_j))+\tilde{\omega}_{21}(t_j)\tilde{\omega}_{12}(t_j)\tilde{\phi}_{32}\\
			(\tilde{\omega}_{32}(t_j)+\tilde{\phi}_{32}(\res_{t_i}(\tfrac{dz}{z-t_3})-\nu_{j,0}))\tilde{\omega}_{21}(t_j)\\
			0
		\end{pmatrix},
		\]
		\[
		\tilde{v}^{(2)}_{j,2}:=(\res_{t_i}\tilde{\nabla}-\nu_{j,1}\tilde{\phi}_{t_j})(\tilde{v}^{(1)}_{j,1}), \;
		\tilde{v}^{(2)}_{j,1}:=
		\begin{pmatrix}
			-\nu_{i,0}\\
			\tilde{\omega}_{21}(t_j)\\
			0
		\end{pmatrix}. 
		\]
		Let $\tilde{l}^{(m)}_{j,2}:=A\tilde{v}^{(m)}_{j,2} \subset A^{\oplus 3}$ and $ \tilde{l}^{(m)}_{j,1}=A\tilde{v}^{(m)}_{j,1}+A\tilde{v}^{(m)}_{j,2} \subset A^{\oplus 3}$ for $m=1,2$ and  $j=1,2,3$. Then we can see that $A^{\oplus 3}/l^{(m)}_{j,n}$ is flat over $A$ and $(\res_{t_j}\tilde{\nabla}-\nu_{j,n}\tilde{\phi}_{t_j})(l^{(1)}_{j,n})\subset l^{(2)}_{j,n+1}$ for any $j=1,2,3$ and $n=0,1,2$ by the way of taking lifts $\tilde{\omega}_{12},\tilde{\omega}_{13},\tilde{\omega}_{22},\tilde{\omega}_{32}$. So $\tilde{\phi}, \tilde{\nabla}, \tilde{l}^{(1)}_{i,j}$ and $ \tilde{l}^{(2)}_{i,j}$ are desire lifts.
	\end{proof}
	
	\section*{Acknowledgments}
	The author are very grateful to Michiaki Inaba,  Arata Komyo, Ryo Ohkawa, Masa-Hiko Saito, Yasuhiko Yamada, and  K\=ota Yoshioka for valuable comments, useful suggestions, and helpful discussions. He also would like to thank the anonymous referee for improving this paper. He is supported by Japan Society for the Promotion of Science KAKENHI Grant Numbers 22J10695. This work was supported by the Research Institute for Mathematical Sciences, an International Joint Usage/Research Center located in Kyoto University.
	

\begin{thebibliography}{9}
		\bibitem[AB]{AB}D. Arinkin and A. Borodin, Moduli spaces of d-connections and difference Painleve equations, em Duke Math. J., 134(3):515--556, 2006. Math/0411584.
		\bibitem[AK]{AK}A.Altman and S.Kleiman, Compactifying the Picard Scheme, Adv. in Math., 35(1980),50--112.
		\bibitem[AL]{AL}D. Arinkin and S. Lysenko, On the moduli of SL(2)-bundles with connections on $\mathbf{P}^1 \setminus \{x_1, \ldots, x_4\}$, Internat. Math. Res. Notices (1997), no. 19, 983--999.
		\bibitem[Bo1]{Bo1}P. Boalch, Quivers and difference Painlev\'e equations, Groups and symmetries, 25--51, CRM Proc. Lecture Notes, 47, Amer. Math. Soc., Providence, RI, 2009.
		\bibitem[Bo2]{Bo2}P. Boalch, Simply-laced isomonodromy systems, Publ. Math. I.H.E.S. 116 (2012), no. 1, 1--68,
		\bibitem[DL]{DL}K. Diarra, and F. Loray, Normal forms for rank two linear irregular differential equations and moduli spaces, Period. Math. Hungar., Periodica Mathematica Hungarica. Journal of the J\'{a}nos Bolyai Mathematical Society, 84, 2022, 2, 303--320.
		\bibitem[DM]{DM}B. Dubrovin and M. Mazzocco, Canonical structure and symmetries of the Schlesinger equations, Comm. Math. Phys.,271, (2), 289--373, 2007.
		\bibitem[DST]{DST}A. Dzhamay, H. Sakai, and T. Takenawa, Discrete Schlesinger transformations, their Hamiltonian formulation, and Difference Painlev\'e equations, arXiv:1302.2972v2 [math-ph], 2013, pp. 1--29.
		\bibitem[Do]{Do} J. Dou\c{c}ot, Diagrams and irregular connections on the Riemann sphere, Preprint (2023),
		arXiv:2107.02516v3.
		\bibitem[DT]{DT}A. Dzhamay and T. Takenawa, Geometric analysis of reductions from Schlesinger transformations to difference Painlev\'e equations, Algebraic and analytic aspects of integrable systems and Painlev\'e equations, Contemp. Math., 651, 87--124, Amer. Math. Soc., Providence, RI,2015.
		\bibitem[FL]{FL}T. Fassarella and F. Loray, Flat parabolic vector bundles on elliptic curves, Journal f\"{u}r die Reine und Angewandte Mathematik, 761, (2020), 81--122
		\bibitem[FLM]{FLM}T. Fassarella, F. Loray and A. Muniz, On the moduli of logarithmic connections on elliptic curves. Math. Z. 301, 4079--4118 (2022)
		\bibitem[Ha]{Ha} R. Hartshorne, Algebraic geometry, Graduate Texts in Mathmatics, No. 52. Springer-Verlag, New York, 1997.
		\bibitem[HY]{HY}K. Hiroe and D. Yamakawa. Moduli spaces of meromorphic connections and quiver varieties. Advances in Mathematics, 266:120--151, 2014.
		\bibitem[IIS1]{IIS1} M. Inaba and K. Iwasaki and M. -H. Saito, Moduli of stable parabolic connections, {R}iemann-{H}ilbert
		correspondence and geometry of {P}ainlev\'{e} equation of type
		{VI}. {I},  Publ. Res. Inst. Math. Sci. 42 (2006), no. 4, 987--1089. 
		\bibitem[IIS2]{IIS2}  M. Inaba and K. Iwasaki and M. -H. Saito, Moduli of stable parabolic connections, {R}iemann-{H}ilbert
		correspondence and geometry of {P}ainlev\'{e} equation of type
		{VI}. {II}, Adv. Stud. Pure Math., 45, Math. Soc. Japan, Tokyo, (2006). 
		\bibitem[In]{In} M. Inaba, Moduli of parabolic connections on curves and the
		{R}iemann-{H}ilbert correspondence, J. Algebraic Geom. 22 (2013), no. 3, 407--480. 
		\bibitem[KS]{KS}A. Komyo and M. -H. Saito, Explicit description of jumping phenomenon moduli spaces of parabolic connections and Hilbert schemes of points on surfaces. Kyoto J. Math. 59(3), 515--552 (2019)
		\bibitem[KLSS]{KLSS} A. Komyo, F. Loray, M. -H. Saito, and S. Szab\'{o}, canonical coordinates for moduli spaces of rank two irregular connections on curves, Preprint arXiv:2309.05012 (2023).
		\bibitem[LS]{LS}F. Loray and M. -H. Saito, Lagrangian fibrations in duality on moduli spaces of rank 2 logarithmic connections over the projective line, Int. Math. Res. Not. IMRN (2015), no. 4, 995--1043.  
		\bibitem[Ma]{Ma} T. Matsumoto, Birational description of moduli spaces of rank 2 logarithmic connections, Preprint arXiv:2105.06892 (2021).
		\bibitem[Mu]{Mu} D. Mumford, J. Fogarty, and F. Kirwan, Geometric invariant theory, third ed., Springer, 1993.
		\bibitem[Ob]{Ob}S.Oblezin, Isomonodromic deformations of sl(2)-Fuchsian systems on the Riemann sphere, Mosc.Math. J. 5, 415--441 (2005)
		\bibitem[Ok]{Ok}K. Okamoto, Sur les feuilletages associ\'es aux \'equations du second ordre \`a points critiques fixes de P. Painlev\'e, Espaces des conditions initiales, Japan. J. Math. 5, (1979), 1--79.
		\bibitem[Sa1]{Sa1}H. Sakai, Rational surfaces associated with affine root systems and geometry of the Painlev\'e equations, Comm. Math. Phys., 220(1):165--229, 2001.
		\bibitem[Sa2]{Sa2}H. Sakai, Problem: discrete Painlev\'e equations and their Lax forms, Algebraic, analytic and geometric aspects of complex differential equations and their deformations. Painlev\'e hierarchies, RIMS K\^{o}ky\^{u}roku Bessatsu, B2, Res. Inst. Math. Sci. (RIMS), Kyoto, 2007, pp. 195--208. MR 2310030 (2008c:33020)
		\bibitem[SS]{SS}M. -H. Saito and S. Szab\'{o}, On apparent singularities and canonical coordinates for moduli spaces of connections, in preparation.
		\bibitem[Yo]{Yo} K. Yokogawa, Moduli of stable pairs, J. Math. Kyoto Univ. 31 (1991) 311--327.
	\end{thebibliography}
\end{document}